\documentclass{article}

\usepackage{microtype}
\usepackage{graphicx}
\usepackage{subcaption}
\usepackage{booktabs} 

\usepackage{hyperref}

\newcommand{\prinangle}{\cos\theta}

\usepackage[preprint]{icml2026}

\usepackage{amsmath}
\usepackage{amssymb}
\usepackage{mathtools}
\usepackage{amsthm}

\usepackage[capitalize,noabbrev]{cleveref}

\usepackage{backref} 
\usepackage{stfloats} 
\usepackage{siunitx} 
\usepackage{tcolorbox} 
\usepackage{multirow}
\usepackage{multicol}
\usepackage[table]{xcolor}
\usepackage{enumitem} 
\usepackage{array} 

\usepackage{setspace}              

\setlist[itemize]{noitemsep, topsep=2pt}
\setlist[enumerate]{noitemsep, topsep=2pt}

\theoremstyle{plain}
\newtheorem{theorem}{Theorem}
\newtheorem{proposition}[theorem]{Proposition}
\newtheorem{lemma}[theorem]{Lemma}
\newtheorem{corollary}[theorem]{Corollary}
\theoremstyle{definition}
\newtheorem{definition}[theorem]{Definition}

\theoremstyle{remark}
\newtheorem{remark}[theorem]{Remark}
\newtheorem{example}[theorem]{Example}

\floatname{algorithm}{Example Construction}

\crefname{algorithm}{Example Construction}{Example Constructions}
\Crefname{algorithm}{Example Construction}{Example Constructions}

\usepackage[disable,textsize=tiny]{todonotes}

\icmltitlerunning{Limits of Convergence-Rate Control for Open-Weight Safety}

\begin{document}

\twocolumn[
  \icmltitle{Limits of Convergence-Rate Control for Open-Weight Safety}


  \icmlsetsymbol{equal}{*}

\begin{icmlauthorlist}
\icmlauthor{Domenic Rosati}{1,2}
\icmlauthor{Xijie Zeng}{1,2}
\icmlauthor{Hong Huang}{1}
\icmlauthor{Sebastian Dionicio}{1}
\icmlauthor{Subhabrata Majumdar}{3}
\icmlauthor{Frank Rudzicz}{1,2}
\icmlauthor{Hassan Sajjad}{1}
\end{icmlauthorlist}

\icmlaffiliation{1}{Dalhousie University}
\icmlaffiliation{2}{Vector Institute}
\icmlaffiliation{3}{Indian Institute of Management Bangalore}

\icmlcorrespondingauthor{Domenic Rosati}{domenic.rosati@dal.ca}

  \icmlkeywords{Machine Learning, ICML, Open-Weight Safety, Spectral Methods}

  \vskip 0.3in
]



\printAffiliationsAndNotice{}  


%
\begin{abstract}
 Open-weight foundation models can be fine-tuned for harmful purposes after release, yet no existing training resistance methods provide theoretical guarantees. Treating these interventions as convergence-rate control problems allows us to connect optimization speed to the spectral structure of model weights. We leverage this insight to develop a novel understanding of convergence rate control through spectral reparameterization and derive an algorithm, SpecDef, that can both provably and empirically slow first- and second-order optimization in non-adversarial settings. In adversarial settings, we establish a fundamental limit on a broad class of convergence rate control methods including our own: an attacker with sufficient knowledge can restore fast convergence at a linear increase in model size. In order to overcome this limitation, future works will need to investigate methods that are not equivalent to controlling convergence rate.
\end{abstract}

\section{Introduction}
\label{sec:introduction}
Regardless of how safe foundation models \textit{behave} (\textbf{inference-time safety}), there is currently no known method to definitively prevent \textit{training} open weights \citep{qi2023fine} for purposes such as deepfake generation (\textbf{training-time safety}). Open-weight governance \citep{nevo2024securing} typically involves pro-open interventions like enhanced user licensing \citep{seger2023open} and pro-closed interventions like throttled releases \citep{anderljung2023frontier}. Training-time resistance \citep{tamirisa2025tamperresistantsafeguardsopenweightllms,rosati-etal-2024-immunization,zheng2024imma} offers emerging alternatives, though existing approaches lack theoretical grounding and fail under systematic evaluation (\cref{tab:unlearning-defence-evaluation} and \citealp{qi2024evaluating}).

Training-time safety can be understood as increasing the number of optimization steps required to fine-tune a model for harmful purposes without degrading its original performance. More formally, we consider an open-weight foundation model $f \in \mathcal{F}$, parameterized by $\theta_{t=0}$. We seek an intervention $\tau \in \mathcal{T}: \mathcal{F} \to \mathcal{F}$ applied \textbf{\textit{solely}} at time $t = 0$ that maximizes the number of iterations $t$ required for empirical risk minimization starting with $\tau[f_{\theta_{t=0}}]$ using loss $\mathcal{L}: \mathcal{X} \times \mathcal{Y} \to \mathbb{R}$ over dataset $\{x_i \in \mathcal{X}, y_i \in \mathcal{Y}\}_i^n$ sampled from a distribution representing unsafe behaviour. 

Open-weight \textbf{threat models} \citep{wallace2025estimating, casper2025open} consider an attacker with full access to modify pretrained model weights to make it useful for  e.g., weapons development. Attackers with budget to train from scratch or distill the model are excluded. In this paper, we distinguish between (1) a non-adversarial \textbf{safety regime} where increasing resistance via convergence rate control is desirable for preventing accidental misuse and hardening safety guards (e.g., to enhance responsibility and avoid liability), and (2) an adversarial \textbf{security regime} where motivated attackers exist and may have varying levels of skill and budget.

While recent literature (cf. \citealp{huang2024harmful,che2025model}) proposes a number of intervention methods, they lack a unified theoretical explanation of \textit{why} they provide resistance. This paper provides a foundational theory of training-time safety by framing it as controlling convergence rates in fine-tuning settings. To do this, we utilize classical iteration complexity results for first-order methods (\S~\ref{sec:convergence-rate-control}; second-order extensions in \cref{app:second-order-analysis}) where curvature of the loss landscape determines loss minimization speed. 

\textbf{Our contributions are:}
\begin{itemize}[leftmargin=*,nolistsep]
\item[(i)] We develop a tractable method to characterize and control the Hessian spectrum using weight matrices alone (\S~\ref{sec:hessian_lower_bound}). 
\item[(ii)] We identify a class of symmetric transformations that preserve zeroth-order behaviour under arbitrary manipulation of Hessian spectral values (\S~\ref{sec:spectral-reparameterization}) and derive an algorithm, SpecDef, that implements these.
\item[(iii)] We prove a tight characterization of fundamental limits: a broad class of convergence-rate control methods can be undone at linear cost in model size (\S~\ref{sec:security-analysis-of-convergence-rate-control}).
\item[(iv)] We provide methodological improvements to open-weight safety assessment by evaluating across vision and language domains (\cref{tab:nsfw_prevention,tab:unlearning-defence-evaluation}), introducing curvature-aware optimizers (\cref{tab:optimizer-comparison-main}), and deploying spectral reparameterization as a novel attack (\cref{app:previous-defence-analysis}).

\end{itemize}
At a high level, our approach is: first-order optimizers must choose learning rates inversely proportional to loss curvature. By inflating weight singular values, we force optimizers into a regime where updates are numerically ineffective.
\section{Convergence Rate is Determined By Spectral Values}

In this section, we show how to control convergence rates of first-order methods \textbf{\textit{even after the release of the model}} by manipulating the principal curvature directions of the optimization objective i.e., the Hessian. It is well known that first-order convergence rates are governed by the Hessian spectrum \citep{Wright_Recht_2022} (\cref{fig:quadratic} illustrates this).

Consider the sequence of parameter updates $\{\theta_t\}$ from a gradient descent process. The number of iterations $t=k$ required to reduce (i) distance: $\|\theta_t - \theta_*\|$, (ii) loss: $\|\mathcal{L}_{\theta_t} - \mathcal{L}_{\theta_*}\|$, or (iii) gradient: $\|\nabla \mathcal{L}_{\theta_t}\| \leq \epsilon$ is called the \textit{iteration complexity} of an optimization algorithm. In machine learning, $\theta$ parameterizes a model $f_{\theta}$ and $\theta_*$ is a stationary point of interest such as a $\epsilon$-local minimum  or saddle.

\subsection{Iteration Complexity Bounds}
While stochastic gradient descent and Adam are most commonly used in practice, our analysis focuses on the best-possible first-order iteration complexity given by Nesterov's method of acceleration \citep{nesterov2013introductory}. These results assume local convexity and $L$-smoothness (the first derivative is Lipschitz continuous). For non-convex problems, we can still provide first-order iteration complexity results for $L$-smooth functions. This assumption is reasonable for foundation model fine-tuning since these networks are often compositions of smooth functions, and nonsmooth results do not improve upon presented rates. In the remainder of this section, we summarize how iteration complexity in both cases depends on the constant $L$, which is lower bounded by the largest singular value of the Hessian. We omit adaptive, stochastic, and non-smooth methods (see \citealp{Wright_Recht_2022,de2018convergence, dfossez2022a}) since they yield worse rates than Nesterov's and still depend on $L$. Curvature-aware methods such as AdaHessian \citep{yao2021adahessian}, Muon \citep{jordan2024muon}, and Sophia \citep{liusophia} can achieve better convergence rates, but due to approximation errors and third-order effects, they are still controlled by the methods we develop below (\cref{app:second-order-analysis}).

\textbf{Theory Summary:} the key quantity governing convergence speed is the smoothness constant $L$. Larger curvature forces smaller stable learning rates, increasing the number of required iterations. Both worst-case and optimal first-order convergence bounds depend monotonically on $L$.

\cref{prop:l-smooth-descent-iteration} demonstrates sublinear convergence in worst-case smooth settings. We assume that optimal learning rates are chosen; for $L$-smooth gradient descent, this is $\eta = 1/L$.

\begin{proposition}[$L$-smooth gradient descent iteration complexity]
\label{prop:l-smooth-descent-iteration}
For gradient descent on an $L$-smooth loss function $\mathcal{L}$, the number of iterations $k$ needed to achieve $\underset{0 \leq i \leq k-1}{\min} \|\nabla \mathcal{L}_{\theta_i}\| \le \epsilon$, is \(
 k \ge ( L\|\theta_0 - \theta_*\| )^2/\epsilon^2
\), where $\theta_0$ is the initial model parameters, $i$ is the iteration index, and $\theta_*$ is the parameters at a stationary point.
\end{proposition}

For best-case settings (e.g., starting from $\theta_0$ in a convex basin), Nesterov's optimal rates for first-order methods apply. While it is known that no purely first-order method can have better rates, the complexity is still dependent on $L$.

\begin{proposition}[Nesterov iteration complexity]
\label{prop:nesterov-iteration}
Let $\mathcal{L}$ be an $L$-smooth convex function. For Nesterov's optimal method, the guarantee is that to obtain $\mathcal{L}_{\theta_k} - \mathcal{L}_{\theta_*} \le \epsilon$, the number of steps is given by \(
 k \ge\sqrt{ \left[ 2L\|\theta_0 - \theta_*\|^2 \right]/\epsilon} - 1 
\).
\end{proposition}

Both results depend critically on $L$, which satisfies $L \geq \sigma_1(H^{\mathcal{L}}_{\theta})$, where $H^{\mathcal{L}}_{\theta}$ is the loss Hessian (shown in \cref{prop:L-lower-bound-app}). Detailed proofs, and reviews of $L$-smooth gradient descent and Nesterov's method, appear in \cref{app:review-of-gradient-decent-methods}. Analysis of these rates reveals three opportunities for slowing down the iteration complexity in the following section.

\subsection{Controlling Convergence Rates}
\label{remark:why-curvature}

\paragraph{Bad local minima} The first opportunity to control convergence rates is to seek high-loss local minima from which gradient descent cannot escape. However, random matrix theory for deep networks shows that such bad minima are rare and may not always exist \citep{choromanska2015loss, baskerville2022universal, baskerville2022appearance}. Even when they exist, finding them is computationally intractable \citep{bubeck2020trapgradientflow,hollender23a,huanjian2025adaptive}. Furthermore, methods like stochastic gradient Langevin dynamics can escape such local minima \citep{raginsky2017non,tzen2018local}, making this approach easy to defeat.
\paragraph{Weight distance} The second approach is to maximize the weight distance $\|\theta_0 - \theta_{\ast}\|$. However, in foundation model fine-tuning, the pre-trained weights are already close to task-specific optima \citep{neyshabur2020being}, making this distance inherently small. Gradient ascent methods explicitly designed to maximize this distance have proven ineffective \citep{rosati-etal-2024-immunization, golatkar2019eternal}.
\paragraph{Increasing curvature} The final option is to manipulate $L$. Increasing $L$ forces smaller learning rates for convergence. We show later that by making $L$ sufficiently large, learning rates can be driven into a regime where gradient updates become numerically ineffective, resisting convergence.

Beyond these approaches, one might consider obfuscating gradients with high variance to prevent unbiased estimation. However, this has been shown to be easily defeated in \citet{athalye2018obfuscated}. Conversely, naively reducing curvature to create flat landscapes is unlikely to succeed given that it would allow larger stable learning rates. Although analysis of previous methods in \cref{app:previous-defence-analysis} reveals that a more nuanced curvature reduction based on singular vector alignment might provide some training time resistance.

\section{A Hessian Spectral Lower Bound Dependent on Weight Spectrum}
\label{sec:hessian_lower_bound}
We now state our main theoretical result. The Hessian is denoted $H^{\mathcal{L}}_{\theta} \equiv \nabla^2_{\theta}\mathcal{L}  = \left[\frac{\partial^2 \mathcal{L}}{\partial \theta_i \partial \theta_j}\right]_{i, j}$ where, in this context, $\theta$ is the concatenated vector of all parameters and $i,j$ index Hessian blocks or the concatenated vector of a single layer ($\theta_i$). The Hessian, computed for a neural network function $f_\theta(x): x \mapsto  \theta_{n+1} \circ \phi_{n} \circ \theta_{n} \circ \cdots \circ \phi_1 \circ \theta_1x$ with element-wise activations $\phi : \mathbb{R}^d \to \mathbb{R}^d$, admits a lower bound on its largest singular value depending on $\sigma_1(\theta_i)$ of a chosen weight matrix. Controlling $\sigma_1(\theta_i)$ thus controls $\sigma_1(H^{\mathcal{L}}_{\theta})$ and, via \cref{prop:L-lower-bound-app}, the convergence rate. An illustrated intuition of the principal angle quantity $\prinangle(A,B)$ (see \cref{def:principal_angle}) that we use below can be found in \cref{fig:svlb-inequality-illustration}. 

\begin{theorem}[Hessian Singular Value Lower Bound]
\label{thm:hessian-singular-value-bound}
Let $\nabla^2_{\theta}\mathcal{L} \in \mathbb{R}^{n \times n}$ be the Hessian of $\mathcal{L}$ with network function defined earlier. Assume there exists a submatrix of the Hessian $\nabla^2_{\theta_i,\theta_j} \mathcal{L} :=  \left[\frac{\partial^2 \mathcal{L}}{\partial \theta_i \partial \theta_j}\right]_{i, j} \in \mathbb{R}^{p \times q}$ where $p = |\theta_i|, q = |\theta_j|$ that has the matrix product structure $ABC$ with arbitrary matrices $A,C$. Then there exists a $B := \theta_k$ such that  $\sigma_1(\nabla^2_{\theta}\mathcal{L})$ is bounded below as \[
\sigma_1(\nabla^2_{\theta}\mathcal{L}) \geq \sup_{r_1,r_2 \geq 1} \sigma_{r_1}(A)\sigma_{1}(B)\sigma_{r_2}(C) \prinangle_1\prinangle_2,
\] where $\prinangle(\cdot, \cdot)$ is the cosine of the largest principal angle between the subspace as defined in \cref{def:principal_angle} with $\prinangle_1 := \prinangle(A_{r_1},BC_1)$ and $\prinangle_2 :=\prinangle(B_1, C_{r_2})$.
\end{theorem}

\begin{proof}[Proof Sketch]
The full proof appears in \cref{app:hessian-lower-bound}; see \cref{app:singular-value-inequalities} for background on singular value inequalities. Hessian spectral values are bounded below by spectral values of submatrices (via Poincaré's separation theorem, \cref{corr:block-interlacing-for-hessian}). These submatrices correspond to second partial derivatives w.r.t.\ parameter matrices $\theta_i$ and $\theta_j$. These block matrices decompose following the algebraic structure above, yielding the lower bound via \cref{thm:principal-angle-lower-bound}.
\end{proof}

Standard architectures (MLPs, CNNs, Transformers) satisfy similar bounds (\cref{app:hessian-lower-bound}). The tightness of the bound depends on the alignment of the matrix products that compose the Hessian and is extensively discussed in \cref{remark:bound-tightness,app:bound-tightness} with the implication that small principal angles resulting from rank deficiency or alignment may require selecting more layers or additional singular value indices. This bound is tighter than classical matrix analysis results \citep{Horn_Johnson_1991, bhatia2013matrix} and can uniquely provide non-vacuous results under rank deficiency.

\begin{example}[Three-layer MLP]
For a three-layer network with ReLU activations, the Hessian block w.r.t.\ $\theta_3,\theta_1$ is given by \[
\frac{\partial^2 f}{\partial \theta_3 \partial \theta_1} =
\underbrace{(x^{\top} \otimes I_{m})^{\top} D_{z_1}}_{A} 
\:
\underbrace{\theta_2^{\top}}_B\:\underbrace{D_{z_2}}_C.
\] Here, $A$ and $C$ are activation ($z_1,z_2$) dependent Jacobian (D) factors, while $B=\theta_2^{\top}$ is the intermediate weight matrix. Consequently, the spectrum of this Hessian block is governed by the singular values of $\theta_2$, together with the spectra of the Jacobian factors and the principal angles.
\end{example}

We provide two empirical validations of the bound. First, \cref{tab:spectral_alignment_main} illustrates \cref{thm:hessian-singular-value-bound} using Hessian derivations from \cref{app:example-hessian-derivations}.
Second, \cref{fig:toy-network-iteration-hessian-experiment} shows how $\sigma_1(H^{\mathcal{L}}_{\theta})$ and convergence rate evolve as $\sigma_1(\theta_i)$ increases in practical architectures. These results confirm that controlling weight matrix spectral values yields convergence rate control. Details on the numerical analysis appear in \cref{app:experimental-details}.

\label{sec:convergence-rate-control}
\begin{figure*}[bht]
    \centering
\includegraphics[width=1\linewidth]{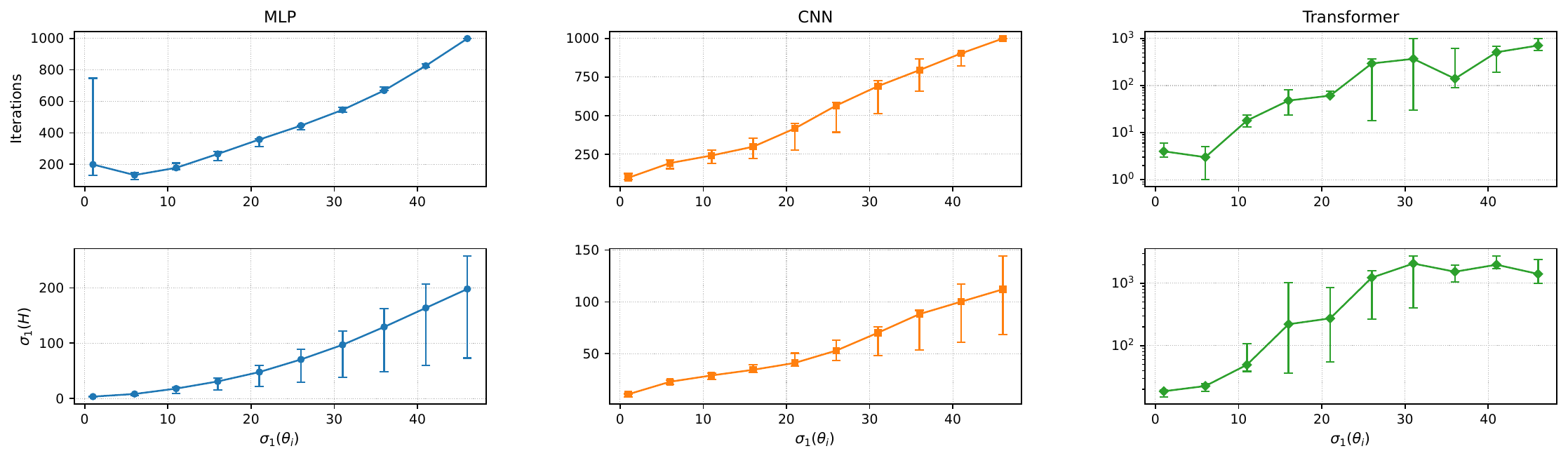}
    \caption{Convergence Rate Control: $\sigma_1(H^{\mathcal{L}}_{\theta})$ and convergence rate (iterations) is increased with $\sigma_1(\theta_i)$}
    \label{fig:toy-network-iteration-hessian-experiment}
\end{figure*}

\section{Spectral Reparameterization for Convergence Control}
\label{sec:spectral-reparameterization}

Having established that convergence rates depend on weight matrix singular values, we now ask: how can we manipulate them algorithmically? We first formalize the class of necessary transformations and derive from it SpecDef (\cref{alg:spectral-deformation}). Arbitrary spectral modifications can destroy model behaviour, which is unacceptable in foundation models. We therefore seek a transformation $\mathcal{T}$ that can: (1) increase the largest Hessian singular value $\sigma_1$ arbitrarily, and (2) keeps zeroth-order behaviour (e.g., the function output or loss values) unchanged. Such a transformation is a second-order spectral symmetry, preserving zeroth-order behaviour under Hessian spectrum deformations.
Formally, let $f(x) = \theta_{n+1} \circ \phi_{n} \circ \theta_{n} \circ \cdots \circ \phi_1 \circ \theta_1x$ be a composition of linear transformations $s_{i} = \theta_{i}\, z_{i - 1}$ with weight matrices $\theta_{i} \in \mathbb{R}^{d_i \times d_{i-1}}$ and twice-differentiable activation functions $\phi_{i}: s_{i} \mapsto z_{i}$, where $i$ is the layer index and $z_{0}= x \in \mathbb{R}^d$. The spectral values of $\theta_{i}$ appear in $\Sigma_i = \mathrm{diag}(\sigma_1, \ldots, \sigma_{\min\{d_i, d_{i-1}\}})$ from the SVD $\theta_{i} = U_{i}\Sigma_{i} V^{\top}_{i}$ with $\theta_i \in \mathbb{R}^{d_i \times d_{i-1}}$.

\begin{definition}[Lower-Max Spectral Reparameterization]
\label{def:upper-max-spec-reparam}
For a given function composition $f$ as defined above, a constant $c > 0$ which will define our minimum spectral value, and $\epsilon \geq 0$ which will define our maximal function distance. A map $\mathcal{T}_c: f_\theta \mapsto f_{\theta^\prime}$ is a lower-max spectral reparameterization of $f$ if
\begin{enumerate}[nolistsep,leftmargin=*]
    \item \textbf{Spectral Control} $\sigma_1(H^{\mathcal{L}}_{\theta^\prime}) \geq c$ where $H^{\mathcal{L}}_{\theta^\prime}$ is the Hessian w.r.t.\ the parameters of $\mathcal{T}_c[f_{\theta}]$.
    \item \textbf{Functional Invariance up to $\epsilon$} Given a distance function $d(f, g)$ on the space of functions $f,g \in \mathcal{F}$, $d(\mathcal{T}_c[f], f) \leq \epsilon $ for a small $\epsilon \geq 0$.
\end{enumerate}
\end{definition}

This definition captures transformations that leave model predictions unchanged while making optimization arbitrarily slow. We consider only a lower bound on the maximum Hessian spectral value---a more general version with upper bounds appears in \cref{def:luk-spectral-reparameterization}, used later for attack construction. Importantly, in this paper we only consider  \textbf{global convergence rate control for all distributions}, though future definition could be relaxed to specific distributions over $\mathcal{L}$.

Our observations build on prior work in symmetric reparameterizations, sharpness, and generalization \citep{dinh2017sharp,foret2020sharpness}. In contrast to scale invariance analyses, control here arises from a small top-$k$ singular value, indicating a deeper symmetry. \citet{kristiadi2023geometry} studied related geometric effects of reparameterization on first- and second-order properties,  but applied them to improving optimization. Symmetry-based perspectives are now standard in optimization for machine learning \citep{simsek2021geometry, zhao2022symmetry}, which we apply to open-weight safety.

\begin{algorithm}[tb]
\caption{Spectral Deformation (SpecDef)}
\begin{algorithmic}[1]
    \STATE Model parameters \(M = \{ \theta_i \in \Theta \}\) where $i$ is a layer index, $n \in \mathbb{N}$ number of layers to select, $k$ is the top $k$ singular values to select, $\alpha$ is the singular value multiplier to use, and layer selection function $\text{Select}: \Theta \times \mathbb{N} \to \Theta$.
    \STATE \(\Theta' \gets \text{Select}(\Theta, n)\) \textit{; e.g., random selection}
    \FOR{$\theta_i \in \Theta'$}
        \STATE $M \gets \left[M_{i+1:}; I; M_{:i}\right]$ \textit{; Inject identity layer}
        \STATE \(U, \Sigma, V \gets \text{SVD}(\theta_i)\)
        \STATE \(T \gets \mathrm{Diag}(\alpha, \ldots, \alpha, 1, \ldots, 1)\) \textit{; Indices $1..k$ get $\alpha$}
        \STATE \(\widetilde{\Sigma} \gets T\Sigma\)
        \STATE \(\theta_i^{\prime} \gets U\widetilde{\Sigma}V^{\top}\)
        \STATE \(\theta_i^{comp} \gets \,U \Sigma \widetilde{\Sigma}^{-1} U^{\top}\) \textit{; or pseudoinverse}
        \STATE $M \gets \left[M_{i+1:}; \theta_i^{comp};\theta_i^{\prime}; M_{:i-1}\right]$ \textit{; replace $I,\theta_i$}
    \ENDFOR
\end{algorithmic}
\label{alg:spectral-deformation}
\end{algorithm}

\cref{alg:spectral-deformation} presents Spectral Deformation (SpecDef) as a spectral reparameterization that can directly be obtained from \cref{thm:spectral-deformation-is-spec-reparam-main} by first inserting identity linear layers adjacent to $\theta_j$, then applying a compensation update. Intuitively, SpecDef works by scaling selected singular values and inserting compensating layers so that the overall function remains unchanged. The algorithm runs once before model release, and is fast: SVD applies only to a few selected layers, which can be batched on GPU, and the matrices are small. For \texttt{GPT-OSS-20b} with 10 layers, this takes under 15 seconds. Further orthogonal decompositions like QR may be used to improve speed of numerical stability.

\begin{theorem}[SpecDef is a (Lower-Max) Spectral Reparameterization]
\label{thm:spectral-deformation-is-spec-reparam-main}
The following spectral deformation is a (lower-max) spectral reparameterization:
For $\theta_i$, set $\widetilde{\Sigma}_{\theta_i} \leftarrow T\Sigma_{\theta_i}$ such that $\sigma_1(U\widetilde{\Sigma}_{\theta_i}V^{\top}) = \alpha$. Where $\alpha$ is set so that $\sigma_1(H_{\theta}^{\mathcal{L}}) \geq c$ (Spectral Control) where $c > 0$ is given. Next, compensate either with $\theta_{j>i} \leftarrow \theta_j U_{\theta_i} \Sigma_{\theta_i} \widetilde{\Sigma}_{\theta_i}^{-1} U^{\top}_{\theta_i}$ or $\theta_{j<i} \leftarrow V_{\theta_i}\widetilde{\Sigma}^{-1}_{\theta_i} \Sigma_{\theta_i} V^{\top}_{\theta_i} \theta_j$ (Functional Invariance). Where $\phi_{i:j}$ and $\phi_{j:i}$ must be degree-$1$ homogeneous and $\widetilde{\Sigma}^{-1}$ is the pseudoinverse in the case of rectangular matrices.
\end{theorem}

The proof appears in \cref{app:algorithm-details} and follows simply from the application of $\sigma_1(U\widetilde{\Sigma}_{\theta_i}V^{\top}) = \alpha$ implying $\sigma_1(H_{\theta}^{\mathcal{L}}) \geq c$ by \cref{thm:hessian-singular-value-bound}. Standard architectures typically are not composed of purely $1$-homogeneous activation functions preventing cross-layer compensation (at least without adaptation which can be obtained ala \citealp{mirzadeh2023relu}). For such architectures, we use layer injection: inserting an adjacent layer with the identity activation, a special case of \cref{thm:spectral-deformation-is-spec-reparam-main}.

\textbf{Implementation Details} To connect SpecDef to convergence control, the smallest effective learning rate, i.e., the rate below which gradients become ineffective, can be determined empirically (\cref{app:further-ablations}). Operationally, we are forcing the attacker to use this rate, since higher rates would cause divergence. The multiplier $\alpha$ to increase $\sigma_k$ for the top $k$ singular values of selected layers can then start at the reciprocal of this rate. For example, many language models fail to learn below $\eta = 10^{-6}$ using SGD or Adam, so we can start with $10^6$, or higher to account for the other factors in the bound in \cref{thm:hessian-singular-value-bound}. Layer selection and how much to increase the singular value multiplier can be determined empirically with test losses and hyperparameter sweeps as part of a ``certification'' process. In our experiments, we find several cases where the singular value can be much lower than the reciprocal value of the effective learning rate, and only a handful of layers are needed. Setting $k$ to be higher gives more opportunities to align the tail of $\sigma_k$ with the largest singular values in the matrix product (\cref{remark:bound-tightness}). Finally, readers may remark that SpecDef is not secure in that an adversary may combine the compensated layers.  We don't attempt a spectral reparameterization with \textit{stealthy} properties (though this may be obtained) due to the results of \citep{thm:layer-injection-attack} which show a general attack on spectral reparameterization methods.

\begin{table*}[t!]
\centering
\caption{
Relearning attack evaluation. Each entry considers number of training steps and WMDP-bio \citep{li2024wmdp} final accuracy in brackets. 
Dagger ($^\dagger$) indicates perplexity increased by more than 100\%. All methods use \texttt{Llama-3.1-8B-Instruct}. SpecDef is applied to the \texttt{ELM}-unlearned model. The unfiltered base model for \texttt{DeepIgnorance} has an effective LR floor starting at $10^{-5}$.
}
\label{tab:unlearning-defence-evaluation}
\begin{sc}
\resizebox{\textwidth}{!}{%
\begin{tabular}{lcccccc}
\toprule
 & Start Acc & $10^{-6}$ & $5 \times 10^{-6}$ & $8 \times 10^{-6}$ & $10^{-5}$ & $3 \times 10^{-5}$ \\
\midrule
\texttt{ELM} & 0.204 & 120 (0.607 $\pm$ 0.008) & 33 $\pm$ 5 (0.635 $\pm$ 0.004) & 30 (0.634 $\pm$ 0.013) & 23 $\pm$ 5 (0.621 $\pm$ 0.027) & 20 (0.671 $\pm$ 0.014) \\
\texttt{RepNoise} & 0.325 & 510 (0.468 $\pm$ 0.044) & 40 (0.629 $\pm$ 0.005) & 30 (0.616 $\pm$ 0.007) & 30 (0.627 $\pm$ 0.008) & 33 $\pm$ 15 (0.650 $\pm$ 0.023) \\
\texttt{RMU} & 0.259 & 510 (0.558 $\pm$ 0.022) & 40 (0.625 $\pm$ 0.009) & 33 $\pm$ 5 (0.626 $\pm$ 0.018) & 30 (0.634 $\pm$ 0.022) & 23 $\pm$ 5 (0.651 $\pm$ 0.030) \\
\texttt{RR} & 0.251 & 300 $\pm$ 181 (0.595 $\pm$ 0.009) & 33 $\pm$ 5 (0.609 $\pm$ 0.013) & 30 (0.639 $\pm$ 0.020) & 23 $\pm$ 5 (0.610 $\pm$ 0.010) & 23 $\pm$ 5 (0.643 $\pm$ 0.008) \\
\texttt{TAR} & 0.290 & 510 (0.307 $\pm$ 0.008) & 126 $\pm$ 20 (0.609 $\pm$ 0.006) & 76 $\pm$ 5 (0.620 $\pm$ 0.014) & 66 $\pm$ 5 (0.638 $\pm$ 0.035) & 66 $\pm$ 25 (0.618 $\pm$ 0.010) \\
\texttt{GradDiff-WMDP} & 0.227 & 510 (0.444 $\pm$ 0.030) & 223 $\pm$ 248 (0.574 $\pm$ 0.067) & 90 $\pm$ 60 (0.618 $\pm$ 0.022) & 70 $\pm$ 26 (0.616 $\pm$ 0.017) & 50 $\pm$ 17 (0.624 $\pm$ 0.020) \\
\texttt{NPO-SAM} & 0.275 & 510 (0.456 $\pm$ 0.014) & 73 $\pm$ 11 (0.605 $\pm$ 0.006) & 46 $\pm$ 5 (0.617 $\pm$ 0.009) & 43 $\pm$ 5 (0.626 $\pm$ 0.008) & 63 $\pm$ 11 (0.651 $\pm$ 0.020) \\
\texttt{NPO} & 0.282 & 510 (0.506 $\pm$ 0.022) & 66 $\pm$ 11 (0.613 $\pm$ 0.010) & 46 $\pm$ 11 (0.612 $\pm$ 0.008) & 40 (0.616 $\pm$ 0.007) & 46 $\pm$ 25 (0.631 $\pm$ 0.024) \\
\texttt{SimNPO} & 0.337 & 510 (0.499 $\pm$ 0.022) & 63 $\pm$ 5 (0.607 $\pm$ 0.006) & 43 $\pm$ 5 (0.629 $\pm$ 0.005) & 40 $\pm$ 10 (0.629 $\pm$ 0.029) & 30 $\pm$ 10 (0.618 $\pm$ 0.010) \\
\texttt{DeepIgnorance} (strong filter) & 0.408 & -- & -- & -- & 510 (0.476 $\pm$ 0.013) & 126 $\pm$ 20 (0.618 $\pm$ 0.016) \\
\texttt{Llama-3.1-8B} & 0.482 & 93 $\pm$ 5 (0.607 $\pm$ 0.006) & 30 (0.651 $\pm$ 0.010) & 20 (0.618 $\pm$ 0.019) & 20 (0.629 $\pm$ 0.005) & 10 (0.609 $\pm$ 0.005) \\
\midrule
\rowcolor{green!10}
SpecDef $\alpha=10\mathrm{k}$ & 0.204 & 10 (0.195 $\pm$ 0.065)$^{\ddagger}$ & 10 (0.171 $\pm$ 0.039)$^{\ddagger}$ & 10 (0.176 $\pm$ 0.020)$^{\ddagger}$ & 10 (0.191 $\pm$ 0.061)$^{\ddagger}$ & 10 (0.195 $\pm$ 0.024)$^{\ddagger}$ \\
\bottomrule
\end{tabular}
}
\end{sc}
\end{table*}

\section{Experimental Validation}
\label{sec:empirical-results}
We now validate SpecDef (and, by extension, our theoretical framework) across four attack scenarios: relearning attacks on unlearned models, harmful fine-tuning, NSFW content generation, and artistic style transfer. Experiments span language models (\texttt{Llama-3.1-8B}, \texttt{GPT-OSS-20B}) and vision models (\texttt{Stable Diffusion V1-4}). We compare against 10 prior defences for relearning and the strongest known methods for each other domain (\texttt{TAR}, \texttt{RepNoise}, \texttt{IMMA}). We also evaluate curvature-aware optimizers (Sophia, Muon, AdaHessian) as adaptive attacks.

The key findings are: (1) all prior methods fail under standard fine-tuning when learning rates are varied (\cref{tab:unlearning-defence-evaluation}), (2) SpecDef induces divergence across all tested learning rates at sufficient $\alpha$ (\cref{tab:combined-llama-convergence,tab:harmful_results_lora_gpt20b,tab:nsfw_prevention,fig:vangogh_defence}), (3) curvature-aware optimizers do not circumvent resistance (\cref{tab:optimizer-comparison-main}), and (4) model utility is preserved (\cref{tab:spectral_impact_combined,tab:qualitative_epoch_comparison_style_relearning}).

\begin{table}[htbp!]
\centering
\caption{SpecDef preserves model utility. Performance differences are negligible across benchmarks even at extreme multipliers.}
\label{tab:spectral_impact_combined}
\resizebox{\columnwidth}{!}{%
\begin{sc}
\begin{tabular}{lcccccc}
\toprule
$\alpha$ & WMDP & PPL & MMLU & Winogrande & ARC & HellaSwag \\
\midrule
1k   & $\Delta$=-0.1 & $\Delta$=-0.01 & $\Delta$=+0.1 & $\Delta$=+0.0 & $\Delta$=+1.0 & $\Delta$=+0.3 \\
100k & $\Delta$=-0.1 & $\Delta$=-0.01 & $\Delta$=+0.6 & $\Delta$=+0.3 & $\Delta$=+0.3 & $\Delta$=+0.0 \\
1B   & $\Delta$=+0.0 & $\Delta$=-0.02 & $\Delta$=+0.1 & $\Delta$=-1.0 & $\Delta$=+1.0 & $\Delta$=+0.0 \\
\bottomrule
\end{tabular}
\end{sc}
}
\end{table}

\paragraph{Setup} All experiments use AdamW with standard \texttt{PyTorch} hyperparameters unless noted. The multiplier $\alpha$ scales the top-$k{=}25$ singular values of $5$ randomly selected layers; $\alpha{=}1$ is unmodified. We report results across 3 random seeds. We sweep learning rates and mark divergence with $\dagger$ (at least one seed) or $\ddagger$ (all seeds). Divergence indicates that perplexity has doubled. As per \citet{henderson2023self}, model collapse counts as successful. The tables should be read as follows: each cell reports the mean steps it takes to reach a task score of $>60\%$ or to diverge. The value in brackets is the mean score on the final step. Full experimental details are reported in \cref{app:experimental-details}.

\begin{table}[b!]
    \centering
    \small
    \caption{Style relearning: SpecDef prevents Van Gogh style recovery; ESD and IMMA fail within 10 epochs of LoRA.}
\label{tab:qualitative_epoch_comparison_style_relearning}
    \begin{sc}
    \begin{tabular}{@{} >{\centering\arraybackslash}m{1.4cm} >{\centering\arraybackslash}m{0.075\textwidth} >{\centering\arraybackslash}m{0.075\textwidth} >{\centering\arraybackslash}m{0.075\textwidth} >{\centering\arraybackslash}m{0.075\textwidth}@{}}
        \toprule
        Method & Before & 1 Epoch & 10 & 50 \\
        \midrule
        ESD &
        \includegraphics[width=0.075\textwidth]{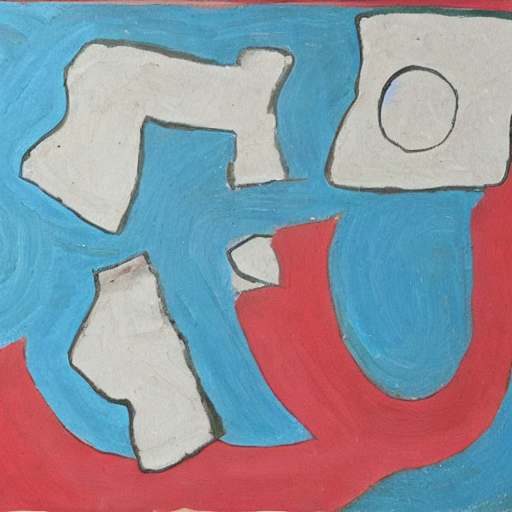} &
        \includegraphics[width=0.075\textwidth]{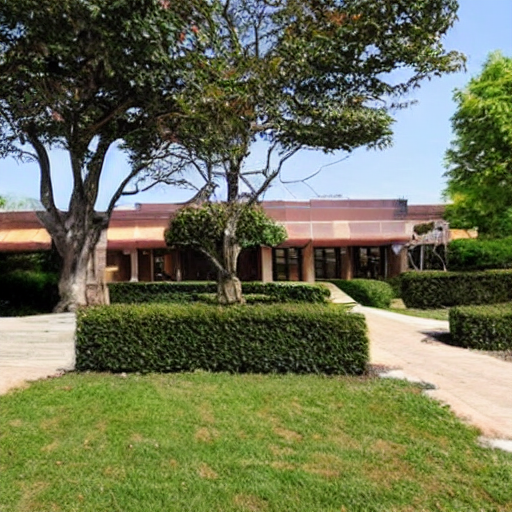} &
        \includegraphics[width=0.075\textwidth]{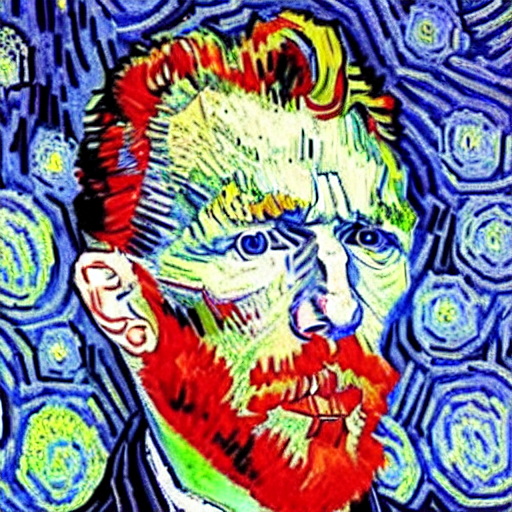} &
        \includegraphics[width=0.075\textwidth]{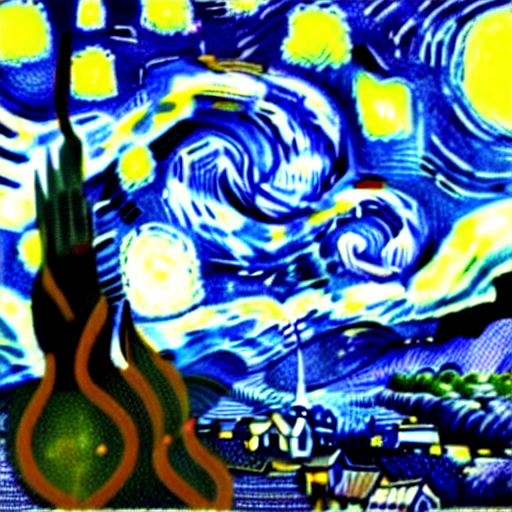} \\
        IMMA &
        \includegraphics[width=0.075\textwidth]{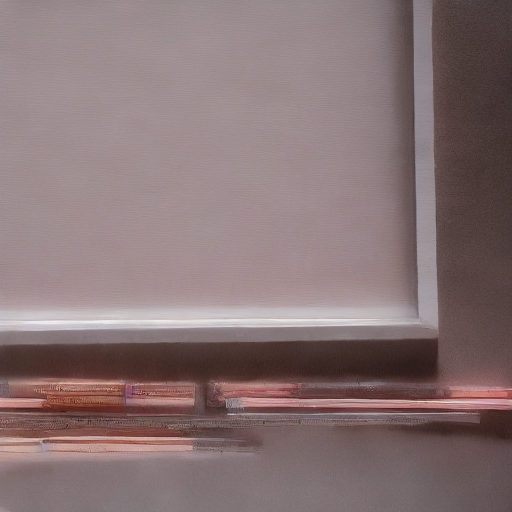} &
        \includegraphics[width=0.075\textwidth]{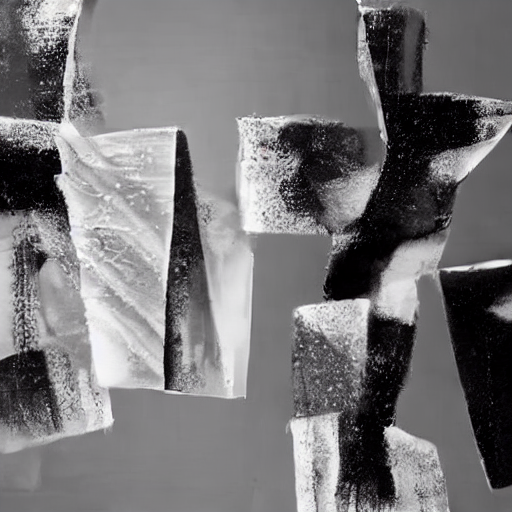} &
        \includegraphics[width=0.075\textwidth]{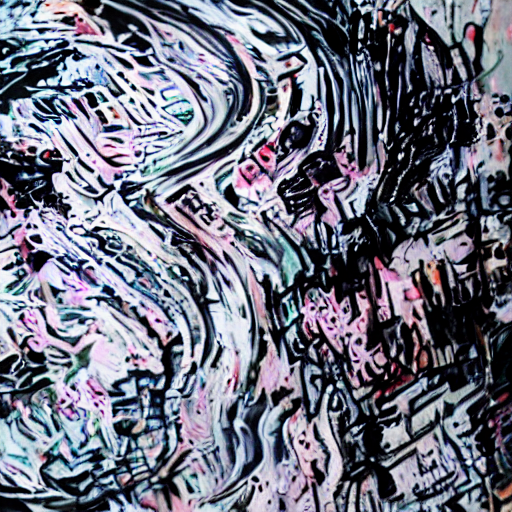} &
        \includegraphics[width=0.075\textwidth]{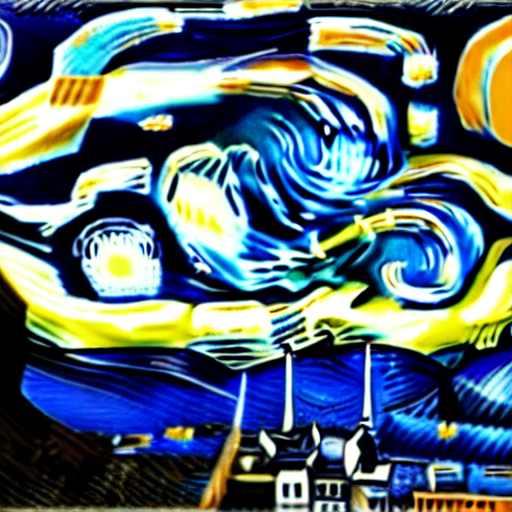} \\
        SpecDef &
        \includegraphics[width=0.075\textwidth]{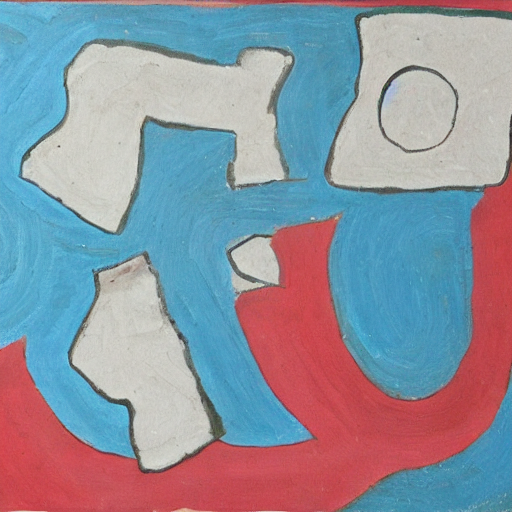} &
        \includegraphics[width=0.075\textwidth]{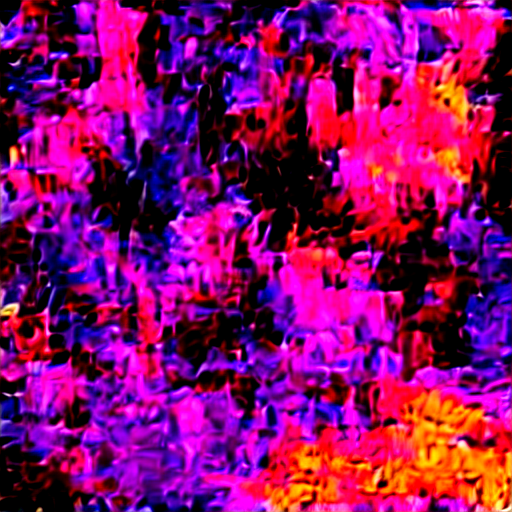} &
        \includegraphics[width=0.075\textwidth]{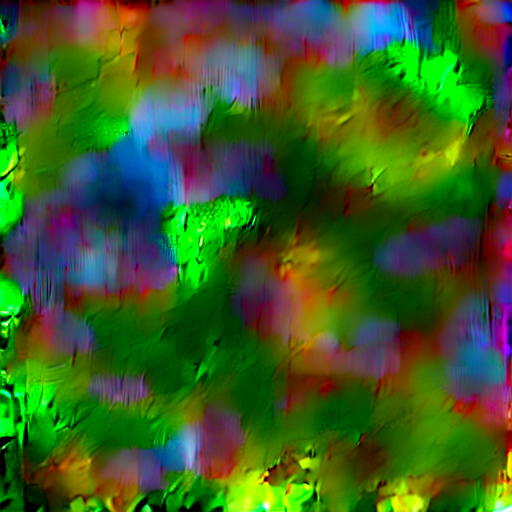} &
        \includegraphics[width=0.075\textwidth]{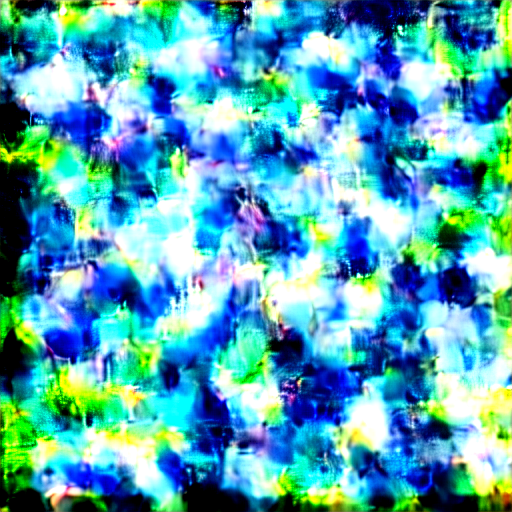} \\
        \bottomrule
    \end{tabular}
    \end{sc}
\end{table}
\begin{table*}[t]
\caption{
Relearning and harmful fine-tuning attacks on \texttt{Llama-3.1-8B} models. 
SpecDef causes divergence ($\ddagger$) at high enough $\alpha$.
}
\label{tab:combined-llama-convergence}
\centering
\setlength{\tabcolsep}{4pt}
\begin{sc}
\resizebox{\textwidth}{!}{%
\begin{tabular}{lccccccc}
\toprule
\multicolumn{8}{c}{Unlearning with WMDP-Bio (\texttt{ELM}-unlearned \texttt{Llama-3.1-8B})} \\
\midrule
Learning Rate & $\alpha=1$ & $1\mathrm{k}$ & $2500$ & $5\mathrm{k}$ & $7500$ & $10\mathrm{k}$ & $1\mathrm{M}$ \\
\midrule
$10^{-6}$ & 120 $\pm$ 10 (0.616 $\pm$ 0.010) & 176 $\pm$ 288 (0.424 $\pm$ 0.098)$^{\dagger}$ & 46 $\pm$ 56 (0.305 $\pm$ 0.051)$^{\ddagger}$ & 16 $\pm$ 10 (0.231 $\pm$ 0.077)$^{\ddagger}$ & 16 $\pm$ 10 (0.222 $\pm$ 0.078)$^{\ddagger}$ & 13 $\pm$ 5 (0.204 $\pm$ 0.068)$^{\ddagger}$ & 10 (0.216 $\pm$ 0.062)$^{\ddagger}$ \\
$5\times 10^{-6}$ & 33 $\pm$ 5 (0.644 $\pm$ 0.054) & 26 $\pm$ 28 (0.278 $\pm$ 0.168)$^{\ddagger}$ & 13 $\pm$ 5 (0.243 $\pm$ 0.095)$^{\ddagger}$ & 10 (0.197 $\pm$ 0.055)$^{\ddagger}$ & 10 (0.187 $\pm$ 0.032)$^{\ddagger}$ & 10 (0.224 $\pm$ 0.073)$^{\ddagger}$ & 10 (0.182 $\pm$ 0.047)$^{\ddagger}$ \\
$8\times 10^{-6}$ & 26 $\pm$ 5 (0.637 $\pm$ 0.022) & 20 $\pm$ 17 (0.267 $\pm$ 0.164)$^{\ddagger}$ & 13 $\pm$ 5 (0.222 $\pm$ 0.101)$^{\ddagger}$ & 10 (0.194 $\pm$ 0.061)$^{\ddagger}$ & 10 (0.303 $\pm$ 0.125)$^{\ddagger}$ & 10 (0.216 $\pm$ 0.109)$^{\ddagger}$ & 10 (0.179 $\pm$ 0.036)$^{\ddagger}$ \\
$10^{-5}$ & 26 $\pm$ 5 (0.658 $\pm$ 0.030) & 16 $\pm$ 11 (0.243 $\pm$ 0.146)$^{\ddagger}$ & 10 (0.221 $\pm$ 0.087)$^{\ddagger}$ & 10 (0.200 $\pm$ 0.033)$^{\ddagger}$ & 10 (0.195 $\pm$ 0.026)$^{\ddagger}$ & 10 (0.217 $\pm$ 0.039)$^{\ddagger}$ & 10 (0.201 $\pm$ 0.069)$^{\ddagger}$ \\
$3\times 10^{-5}$ & 33 $\pm$ 23 (0.651 $\pm$ 0.027) & 10 (0.218 $\pm$ 0.079)$^{\ddagger}$ & 10 (0.174 $\pm$ 0.033)$^{\ddagger}$ & 10 (0.183 $\pm$ 0.059)$^{\ddagger}$ & 10 (0.217 $\pm$ 0.059)$^{\ddagger}$ & 10 (0.190 $\pm$ 0.016)$^{\ddagger}$ & 10 (0.184 $\pm$ 0.030)$^{\ddagger}$ \\
\midrule
\multicolumn{8}{c}{Harmful Fine-tuning with Beavertails (\texttt{Llama-3.1-8B})} \\
\midrule
Learning Rate & $\alpha=1$ & $1\mathrm{k}$ & $2500$ & $5\mathrm{k}$ & $7500$ & $10\mathrm{k}$ & $1\mathrm{M}$ \\
\midrule
$10^{-6}$ & 293 $\pm$ 11 (0.619 $\pm$ 0.005) & 33 $\pm$ 15 (0.674 $\pm$ 0.044)$^{\dagger}$ & 20 $\pm$ 10 (0.650 $\pm$ 0.042)$^{\dagger}$ & 13 $\pm$ 5 (0.627 $\pm$ 0.130)$^{\dagger}$ & 13 $\pm$ 5 (0.616 $\pm$ 0.126)$^{\dagger}$ & 13 $\pm$ 5 (0.617 $\pm$ 0.127)$^{\dagger}$ & 10 (0.043 $\pm$ 0.040)$^{\ddagger}$ \\
$5\times 10^{-6}$ & 96 $\pm$ 5 (0.689 $\pm$ 0.069) & 13 $\pm$ 5 (0.602 $\pm$ 0.144)$^{\dagger}$ & 13 $\pm$ 5 (0.519 $\pm$ 0.274)$^{\dagger}$ & 10 (0.275 $\pm$ 0.369)$^{\dagger}$ & 10 (0.267 $\pm$ 0.418)$^{\dagger}$ & 10 (0.284 $\pm$ 0.431)$^{\dagger}$ & 10 (0.023 $\pm$ 0.019)$^{\ddagger}$ \\
$8\times 10^{-6}$ & 76 $\pm$ 5 (0.701 $\pm$ 0.015) & 13 $\pm$ 5 (0.606 $\pm$ 0.152)$^{\dagger}$ & 10 (0.214 $\pm$ 0.347)$^{\dagger}$ & 10 (0.275 $\pm$ 0.414)$^{\dagger}$ & 10 (0.266 $\pm$ 0.443)$^{\dagger}$ & 10 (0.278 $\pm$ 0.429)$^{\dagger}$ & 10 (0.037 $\pm$ 0.060)$^{\ddagger}$ \\
$10^{-5}$ & 63 $\pm$ 5 (0.651 $\pm$ 0.031) & 13 $\pm$ 5 (0.622 $\pm$ 0.138)$^{\dagger}$ & 10 (0.257 $\pm$ 0.368)$^{\dagger}$ & 10 (0.264 $\pm$ 0.436)$^{\dagger}$ & 10 (0.282 $\pm$ 0.432)$^{\dagger}$ & 10 (0.268 $\pm$ 0.441)$^{\dagger}$ & 10 (0.029 $\pm$ 0.047)$^{\ddagger}$ \\
$3\times 10^{-5}$ & 40 (0.747 $\pm$ 0.041) & 10 (0.249 $\pm$ 0.400)$^{\dagger}$ & 10 (0.262 $\pm$ 0.449)$^{\dagger}$ & 10 (0.360 $\pm$ 0.392)$^{\dagger}$ & 10 (0.238 $\pm$ 0.333)$^{\ddagger}$ & 10 (0.159 $\pm$ 0.211)$^{\ddagger}$ & 10 (0.028 $\pm$ 0.039)$^{\ddagger}$ \\
\bottomrule
\end{tabular}%
}
\end{sc}
\end{table*}

\begin{table}[hbpt]
\centering
\caption{
\label{tab:nsfw_prevention}
NSFW generation: SpecDef achieves 100\% reduction in nudity counts; IMMA achieves 63\%.}
\begin{sc}
\resizebox{\columnwidth}{!}{%
\begin{tabular}{lcccc}
\toprule
Method &Nudity Counts $\downarrow$ &  Images w/ Nudity $\downarrow$ & Img.\ Reduction $\uparrow$ & Nud.\ Reduction $\uparrow$ \\
\midrule
Vanilla SD V1-4 & 346.0 & 205 / 205 & --- & --- \\
Nudity-Erased SD (ESD)  & 59.0 & 32 / 205  & 84.4\% & 82.9\%\\
Unprotected ESD w/ LoRA   & 311.0 ± 23.4 & 119.3 ± 2.9 / 205 & 41.8 ± 1.4\% & 10.1 ± 6.8\% \\
IMMA \citep{zheng2024imma} & 127.7 ± 42.7 & 62.3 ± 14.1 / 205 & 69.6 ± 6.9\% & 63.1 ± 12.4\% \\
\midrule
SpecDef $\alpha=10\mathrm{k}$ & \textbf{0.0 ± 0.0} & \textbf{0.0 ± 0.0 / 205} & \textbf{100.0 ± 0.0}\% & \textbf{100.0 ± 0.0}\% \\
\bottomrule
\end{tabular}
}
\end{sc}
\end{table}

\begin{table}[hbpt]
\centering
\caption{Scaling to 20B parameters with LoRA: training resistance is possible on large models with high enough $\alpha$.}
\label{tab:harmful_results_lora_gpt20b}
\begin{sc}
\resizebox{\columnwidth}{!}{%
\begin{tabular}{lccc}
\toprule
 & LR = $10^{-4}$ & LR = $5{\times}10^{-4}$ & LR = $10^{-3}$ \\
\midrule
Baseline & 580 (0.61 $\pm$ 0.00) & 236 $\pm$ 37 (0.62 $\pm$ 0.02) & 193 $\pm$ 9 (0.62 $\pm$ 0.02) \\
$\alpha=1000$ & 16 $\pm$ 4 (0.67 $\pm$ 0.04)$^{\dagger}$ & 10 (0.57 $\pm$ 0.23)$^{\dagger}$ & 10 (0.26 $\pm$ 0.23)$^{\ddagger}$ \\
$\alpha=10,000$ & 10 (0.26 $\pm$ 0.27)$^{\ddagger}$ & 10 (0.00 $\pm$ 0.00)$^{\ddagger}$ & 10 (0.00 $\pm$ 0.00)$^{\ddagger}$ \\
\bottomrule
\end{tabular}
}
\end{sc}
\end{table}
\begin{figure*}[hbt!]
\centering
\includegraphics[width=1\linewidth]{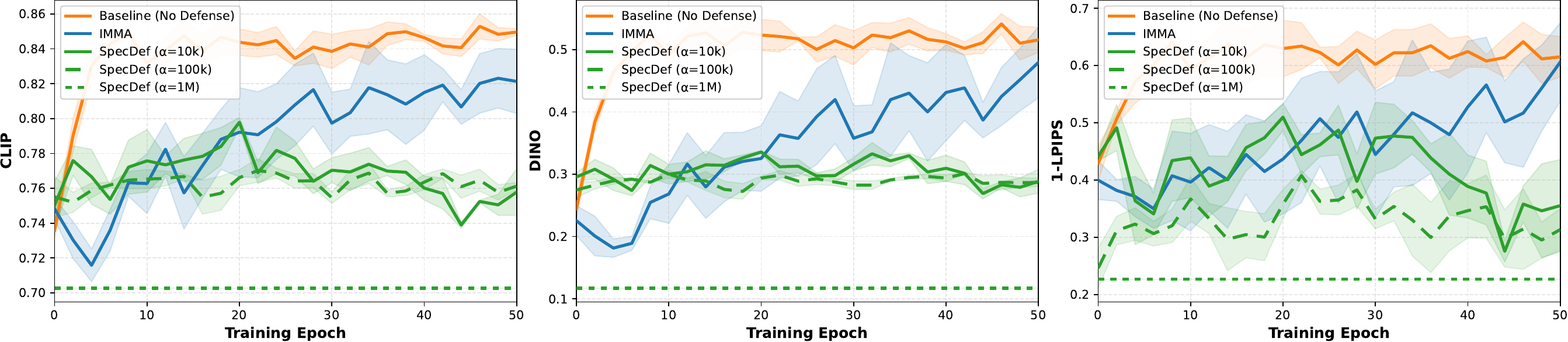}
\caption{
\label{fig:vangogh_defence} Style transfer attack: SpecDef maintains low similarity to Van Gogh references throughout 50 epochs of LoRA fine-tuning, while ESD and IMMA recover the erased style within 10 epochs. Lower similarity indicates stronger resistance.
}
\end{figure*}

\paragraph{Comparison with baselines} Unlearning robustness against relearning attacks is the most well-developed setting in training-time safety \citep{che2025model,fan2025llmunlearningresilientrelearning,o2025deep} so it is our primary focus for the analysis of baseline methods. \cref{tab:unlearning-defence-evaluation} shows no previous methods are able to withstand standard fine-tuning as learning rates are varied. While varying starting accuracies could theoretically confound convergence comparisons, all baseline methods eventually reach the target accuracy regardless of starting point and SpecDef does not give sufficient $\alpha$.


Our results also hold in the case for harmful fine-tuning with \texttt{TAR} \citep{tamirisa2025tamperresistantsafeguardsopenweightllms} and \texttt{RepNoise} \citep{rosati2024representation} (\cref{app:harmful-finetuning-details}) as well as \texttt{IMMA} \citep{zheng2024imma} for NSFW training (\cref{tab:nsfw_prevention}) and style transfer attacks (\cref{fig:vangogh_defence}). For SpecDef, we present comprehensive sweeps of learning rates of our method for unlearning and harmful fine-tuning in \cref{tab:combined-llama-convergence} and vision experiments in \cref{tab:vangogh_lr_ablations_1} showing that as long a sufficient $\alpha$ multiplier is used then convergence rate control is achieved. \cref{tab:harmful_results_lora_gpt20b} corroborates our findings in a large $20\mathrm{b}$ parameter setting using LoRA which reveals a phase transition: at $\alpha=1\mathrm{k}$ with 5 layers, the defence fails to prevent convergence, while $\alpha=10\mathrm{k}$ consistently induces divergence. This identifies the minimum effective $\alpha$ for multi-layer configurations. A more comprehensive study of past methods and where they fit within our framework as well as a novel attack that defeats them based our framework is presented in \cref{app:previous-defence-analysis}.

For adaptive optimizers, convergence rate can be initially accelerated as we increase $\alpha$. This is expected because the gradient can now be much larger and if the step sizes are not diverging then faster convergence would occur. Small multipliers are not intended to provide resistance as our method is intended to operate in the high‑$\alpha$ regime where divergence occurs. \cref{app:previous-defence-analysis} shows cases where this acceleration can be leveraged to attack previous defences.

\paragraph{Curvature-aware methods} We evaluate three classes of curvature-aware optimizers: Sophia \citep{liusophia} (second-order approximate method with Hessian diagonal estimation using Gauss-Newton-Bartlett, $\beta_1=0.965$, $\beta_2=0.99$, $\rho=0.04$), Muon \citep{jordan2024muon} (geometry-aware method using Newton-Schulz orthogonalization with momentum $0.95$), and AdaHessian \citep{yao2021adahessian} (second-order approximate method using Hutchinson's diagonal Hessian estimator). \cref{tab:optimizer-comparison-main} presents the results of these methods. The same setting as \cref{tab:combined-llama-convergence} is used. Our findings were: (i) Sophia was largely ineffective at undoing SpecDef, (ii) AdaHessian was more effective than Sophia but less effective than Muon, and (iii) Muon’s ability to undo SpecDef was roughly comparable to the earlier Adam-based attacks with some cases where increasing $\alpha$  below the subcritical regime can accelerate Muon. \cref{app:second-order-analysis} provides a thorough empirical and theoretical analysis of why this is the case. Unlike SpecDef, previous training resistance methods were very easy to undo with curvature-aware methods (see \cref{tab:curvature-by-model-sigma1}) and therefore these should be to strengthen methodology in future studies.

Ablation experiments on layer selection and top-$k$ selection are presented in \cref{app:further-ablations}. We tried several natural adaptive attacks such as quantization and naive rescaling but did not report them as they immediately destroy numerical stability due to the poor conditioning SpecDef introduces.

\begin{table}[t]
\caption{Curvature-aware optimizers fail to circumvent SpecDef. Muon performs best but still diverges at $\alpha{\geq}100\mathrm{k}$; Sophia is ineffective even at $\alpha{=}1\mathrm{k}$.}
\label{tab:optimizer-comparison-main}
\centering
\setlength{\tabcolsep}{5pt}
\begin{sc}
\resizebox{\columnwidth}{!}{%
\begin{tabular}{lccc}
\toprule
\multicolumn{4}{c}{Muon} \\
\midrule
$\alpha$ & $10^{-5}$ & $2\times10^{-5}$ & $5\times10^{-5}$ \\
\midrule
1 & 765 (0.29 $\pm$ 0.05) & 765 (0.43 $\pm$ 0.03) & 316 $\pm$ 45 (0.61 $\pm$ 0.01) \\
1K & 335 $\pm$ 375 (0.53 $\pm$ 0.15) & 305 $\pm$ 399 (0.54 $\pm$ 0.11) & 140 $\pm$ 156 (0.62 $\pm$ 0.03) \\
10K & 308 $\pm$ 396 (0.50 $\pm$ 0.13)$^{\dagger}$ & 83 $\pm$ 66 (0.48 $\pm$ 0.11)$^{\ddagger}$ & 36 $\pm$ 11 (0.41 $\pm$ 0.09)$^{\ddagger}$ \\
100K & 23 $\pm$ 5 (0.40 $\pm$ 0.10)$^{\ddagger}$ & 20 $\pm$ 10 (0.37 $\pm$ 0.10)$^{\ddagger}$ & 13 $\pm$ 5 (0.33 $\pm$ 0.11)$^{\ddagger}$ \\
\midrule
\multicolumn{4}{c}{Sophia} \\
\midrule
$\alpha$ & $5\times10^{-6}$ & $10^{-5}$ & $2\times10^{-5}$ \\
\midrule
1 & 36 $\pm$ 5 (0.63 $\pm$ 0.02) & 60 $\pm$ 60 (0.63 $\pm$ 0.02) & 36 $\pm$ 28 (0.58 $\pm$ 0.13)$^{\dagger}$ \\
1K & 13 $\pm$ 5 (0.25 $\pm$ 0.13)$^{\ddagger}$ & 13 $\pm$ 5 (0.21 $\pm$ 0.07)$^{\ddagger}$ & 10 (0.19 $\pm$ 0.05)$^{\ddagger}$ \\
\midrule
\multicolumn{4}{c}{AdaHessian} \\
\midrule
$\alpha$ & $10^{-6}$ & $3\times10^{-6}$ & $5\times10^{-6}$ \\
\midrule
1 & 126 $\pm$ 11 (0.62 $\pm$ 0.01) & 60 (0.64 $\pm$ 0.00) & 66 $\pm$ 15 (0.66 $\pm$ 0.04) \\
1K & 153 $\pm$ 230 (0.47 $\pm$ 0.17)$^{\dagger}$ & 30 $\pm$ 34 (0.31 $\pm$ 0.08)$^{\ddagger}$ & 20 $\pm$ 17 (0.22 $\pm$ 0.05)$^{\ddagger}$ \\
10K & 13 $\pm$ 5 (0.19 $\pm$ 0.03)$^{\ddagger}$ & 10 (0.23 $\pm$ 0.05)$^{\ddagger}$ & 10 (0.22 $\pm$ 0.03)$^{\ddagger}$ \\
\bottomrule
\end{tabular}
}
\end{sc}
\end{table}

\section{Fundamental Limits of Convergence-Rate-Based Resistance}
\label{sec:security-analysis-of-convergence-rate-control} 
While convergence-rate control imposes practical friction, it provides only limited security guarantees. To this end, we prove a fundamental limitation: any mechanism that slows fine-tuning by changing convergence rates can be efficiently bypassed by an adaptive attacker. In particular, methods that control convergence via weight matrix singular values---the only viable approach for reparameterization under standard assumptions (e.g., Lipschitz network components; see \cref{app:applicability-impossiblity-result-to-modern-deep-networks})---can be defeated given sufficient attacker budget and expertise.  Such methods exhibit a fundamental trade-off: an attacker can achieve a $k$-th root reduction in $\max_i{\sigma_1(\theta_i)}$ by training a model $k$ times larger.

To this end, we first establish universality.

\begin{theorem}[Only Weight Matrices Provide Unbounded Spectral Control]
\label{thm:only-weight-matrices-main}
Consider a feed forward network $f$ as defined in \S\ref{sec:spectral-reparameterization}. Additionally assume the activation functions $\phi$ to be approximately non-expansive, with approximately non-expansive first two derivatives if they exist. Also, let the data be bounded with $\|x\|_2 \leq c$ for any $x$ drawn from the dataset $\mathcal{D}$. A transformation of that function $\mathcal{T}: \mathcal{F} \to \mathcal{F}$ is a lower-max spectral reparameterization (i.e., a $(l, \infty, 1)$-spectral reparameterization) only if the transformation is a reparameterization of the parameter matrices $\theta_i$.
\end{theorem}
\begin{proof}[Proof Sketch]
The proof (\cref{app:layer-injection-is-universal}) proceeds from two standard assumptions (see \citealp{du2019gradient, xiong2020layer}). (1) bounded data: token embeddings are bounded by construction, pixel normalization ensures boundedness, training datasets have finite maximum norms, (2) nonlinear components have bounded Jacobian and Hessian norms. This includes element-wise activations, residual connections, and normalization layers under standard variance-floor assumptions; attention mechanisms are Lipschitz with constants depending multiplicatively on query, key, and value weight norms. Under these conditions, we show by contradiction that Hessian spectral values must be controlled by weight matrix factors.
\end{proof}
Having established this universal class of convergence-rate control methods, we now show that all such methods can be undone by an adversary---a fundamental limitation of possible security guarantees.

\begin{theorem}[Cost of Undoing Spectral Reparameterization]
\label{thm:layer-injection-attack}
Assume the feed forward networks defined in \S\ref{sec:spectral-reparameterization}. Let $M = \max_{1 \leq i \leq n} \sigma_1(\theta_i)$. Then to achieve $\max_i \sigma_1(\theta_i) \leq M^{1/k}$ it suffices to insert at most $(k - 1) \cdot n$ additional linear layers and perform spectral deformation.
\end{theorem}

\begin{proof}[Proof Sketch]
The proof and practical details are given in \cref{app:layer-injection-details}. Assuming that the defender controls the convergence rate by maximizing $\{\sigma_1(\theta_i)\}_i$ (for which we do not know alternatives), we show an upper bound analogous to \cref{thm:hessian-singular-value-bound}. An attacker achieves $k$-th root reduction by factoring each  layer into $k$ layers using \cref{thm:spectral-deformation-is-spec-reparam-main}.
\end{proof}

Extensions to common deep architectures follow under mild technical assumptions (\cref{app:applicability-impossiblity-result-to-modern-deep-networks}). While worst-case attacks can be made expensive, these results show that lower-max spectral reparameterizations have limited security guarantees. By \cref{thm:only-weight-matrices-main} and \cref{remark:why-curvature}, any open-weight safety method maximizing convergence time for unsafe objectives is subject to this attack. Cheaper attacks are unlikely given that spectral control directly bounds iteration complexity. For example (1) naively scaling layers down might give more tractable gradients but now $\sigma_n$ causes conditioning that is very poor which would result in  instability and (2) naive quantization or clipping would not capture the necessary numerical range of the original transformation effectively destroying model utility. Upper-max spectral reparameterizations that flatten curvature are even easier to defeat (we demonstrate this in \cref{app:previous-defence-analysis}).

\paragraph{Experiments}
We now empirically evaluate the minimal \textit{attack budget} required to undo convergence-rate control. The attacker restores original training cost via \cref{thm:layer-injection-attack} by inserting ${k -1} \cdot n$ identity layers and factorizing defended layers via SVD to achieve $\sigma_1^{1/k}$. We first evaluate relearning attacks on WMDP-Bio for unlearning with learning rate fixed at $3 \times 10^{-5}$. The $\alpha$ multiplier varies across columns; attack cost (number of added layers, resulting in linear parameter increase) varies across rows.

As predicted, \cref{tab:injection-attack} shows that required compute grows linearly in $k$ and logarithmically in the defender's singular-value multiplier, with WMDP-Bio accuracy recovering within few steps after (but not before) appropriate parameter addition. Due to compute limitations, for \texttt{Llama-3.1-8B} with \texttt{ELM} unlearning we only demonstrate spectral deformation on one layer.

We now empirically examine the attack on a fully reparameterized model. We evaluate the harmful fine-tuning attack on \texttt{SmolLM2-360M-Instruct}, which includes a safety guard where the model does not generate tokens when the query is harmful. In the absence of reparameterization, this safety guard can be undone within 10 fine-tuning steps, achieving over 70\% harmfulness.

When we reparameterize all layers with $\alpha = 100\text{k}$ using the smallest effective learning rate ($10^{-7}$), undoing SpecDef requires increasing the model size by approximately $4\times$. These model size increases are relative to post-SpecDef application where SpecDef initially doubles the model size with injected layers (360m to 876m). For $\alpha = 1\text{M}$, undoing SpecDef requires approximately $6\times$ as many parameters and $10\times$ for $\alpha=10\text{M}$, increasing VRAM usage from 1.63GB to 9.30GB with 4.99B total parameters (some layers like LM head are ignored). This translates into an increase from 2.7 average seconds per step to 7.5 seconds per step. In \cref{tab:injection-attack}, fewer injected layers are required to undo these $\alpha$ multipliers. This occurs because when only a single layer (or a few layers) is reparameterized, the network can tolerate one or two, moderate but higher than the original, singular values without destabilizing the overall dynamics. However, when many layers are reparameterized, even moderate singular values across multiple layers compound multiplicatively, leading to large activations and  gradient norms.

\begin{table}[htbp]
\centering
\caption{Layer injection attack validates \cref{thm:layer-injection-attack}: adding $k{-}1$ layers reduces effective resistance from $\alpha$ to $\alpha^{1/k}$, restoring convergence at linear cost in parameters.}
\label{tab:injection-attack}
\begin{sc}
\resizebox{\columnwidth}{!}{%
\begin{tabular}{lcccc}
\toprule
Layers Added & \textbf{$\alpha=1$} & \textbf{$\alpha=10^5$} & \textbf{$\alpha=10^6$} & \textbf{$\alpha=10^{9}$} \\
\midrule
0 ($\sigma^{1 / k=1}$) & 80 (0.61) & 10 (0.16)$^\dagger$ & 10 (0.15)$^\dagger$ & 10 (0.18)$^\dagger$ \\
1 ($\sigma^{1 / k=2}$) & 80 (0.62) & 130 (0.62) & 10 (0.21)$^\dagger$ & 10 (0.17)$^\dagger$ \\
2 ($\sigma^{1 / k=3}$) & 30 (0.64) & 80 (0.61) & 100 (0.60) & 10 (0.17)$^\dagger$ \\
3 ($\sigma^{1 / k=4}$) & 80 (0.61) & 30 (0.61) & 30 (0.61) & 30 (0.61) \\
\bottomrule
\end{tabular}
}
\end{sc}
\end{table}
\section{Discussion}
\label{sec:discussion}
Our work provides theoretical grounding for arbitrarily slowing convergence rates under non-adversarial first-order (second in \cref{app:second-order-analysis}) 
methods common in foundation model fine-tuning. We derive a general class of convergence-rate control methods (spectral reparameterization) and provide an illustrative example (SpecDef) that is practically effective in a non-adversarial regime. However, we also show that adaptive attackers can undo these resistance methods, albeit with a linear increase in budget cost.

\vspace{-1em}
\paragraph{Practical impact} The utility of our results depends on whether one seeks \textit{safety} guarantees (settings without adversaries) or \textit{security} guarantees (assuming adversaries). Geometric safety constraints in open-weight models are practically useful in the safety regime---for example by making accidental removal of safety guards more difficult \citep{qi2023fine} or protecting against reward-incentivized unsafe behaviour \citep{dai2024safe} i.e., ``reward hacking''. One limitation of our method is that it prevents training on all distributions but our framework provides a foundation for remedying this. For example, future work could localize convergence results by aligning distributions with specific singular values \citep{zheng2025model}, creating distribution-specific training constraints.

\paragraph{Security} A main contribution of this work is \textit{provably} ruling out convergence-rate control as a path to security guarantees on open-weight models. Future work could explore means of making $k$-th layer injection even more expensive, for example, by layering in defences such as gradient obfuscation. However, even requiring a substantial training budget only deters low-skilled and low-resourced attackers. These constraints would not prevent state actors developing weapons motivated to undo safety guards. Open-weight models also have a large inference-time attack surface \citep{rando2025adversarial} as well as vulnerability to model exfiltration, such as through distillation \citep{dionicio2025undistillable} that our work does not address. Future work should explore applications of traditional security solutions such as hardware attestation, functional encryption, and model compilation, which could provide future pathways towards security guarantees.

\paragraph{Structural insights for neural network optimization} We develop new bounds linking Hessian spectral values to parameter matrix spectral values in deep networks. Developing this bound further could provide new viable alternatives to approaches for curvature control and analysis based on Hessian-vector products (e.g., \citealp{wang2025better}).

\section{Conclusion}

Our work establishes a concrete theoretical foundation for training-time safety. \cref{thm:hessian-singular-value-bound} enables tractable curvature control at scale, and SpecDef provides resistance where all tested prior methods fail (\cref{tab:unlearning-defence-evaluation}), including against second-order optimizers (\cref{tab:optimizer-comparison-main}). Theorems~\ref{thm:layer-injection-attack} and \ref{thm:only-weight-matrices-main} show that training-time safety mechanisms relying solely on optimization dynamics are inherently limited. While they can impose meaningful barriers for compute-limited attackers and suggest new avenues for guaranteed safety in non-adversarial settings, they offer little protection against well-resourced adversaries, sharply delineating the limits of what is possible in open-weight model safety.

\section*{Acknowledgements}

We would like to acknowledge the generous support of the Killam Foundation, the Vector Institute of Artificial Intelligence, and the Natural Sciences and Engineering Research Council of Canada for funding this work. The compute was made available by the Digital Research Alliance of Canada, Vector Institute, and a grant from the Center for AI Safety.



\section*{Impact statement}

This paper establishes fundamental limits that may narrow the space of viable open-weight safety approaches. Clearly understanding what is and is not possible is essential for the field's maturity. Impossibility results guide future research toward promising directions (cryptographic methods, architectural interventions) while avoiding dead ends. By providing a theoretical basis for understanding resistance mechanisms and their limitations, we demonstrate significant progress in training-time safety. We hope this paper influences the community to approach this topic from first principles and sets a standard for formal theoretical guarantees in future open-weight safety methods.

\bibliography{bibliography}
\bibliographystyle{icml2026}

\newpage
\appendix
\onecolumn
\section{Review of First-Order Methods}
\label{app:review-of-gradient-decent-methods}
This section reviews convergence rates of first-order (gradient-based) steepest descent methods. A deeper review appears in \citet{Wright_Recht_2022}. In foundation model fine-tuning, we minimize a loss $\mathcal{L}: \mathbb{R}^{p} \to \mathbb{R}$ over parameters $\theta \in \mathbb{R}^{p}$ (the concatenation of all weight matrices) where predictions are made by a prediction function $f_{\theta}: \mathbb{R}^d \to \mathbb{R}$ for a dataset $\mathcal{D} = \{(x_j \in \mathbb{R}^d, y_j \in\mathbb{R})\}_{j=1}^{n}$ with $n$ samples and some convex loss like mean-squared error $\|f_{\theta}(x_i) - y_i \|^2_2$ defines the loss $\mathcal{L}$. For simplicity in this analysis we will treat full batch gradient descent. Under common assumptions like i.i.d.\ sampling, empirical risk minimization converges to the population risk as $n$ increases \citep{bach2024learning}. Gradient descent minimizes $\mathcal{L}$ via the recurrence relation:
\[
\theta_{k+1} = \theta_{k} - \eta \nabla \mathcal{L}_{\theta_k}, \quad k = 0, 1, 2, \ldots
\]
for learning rate $\eta \in \mathbb{R}_+$. First-order optimality conditions state that when the gradient achieves $\nabla \mathcal{L}_{\theta} = 0$ a stationary point $\theta_*$ is found. Under assumptions below, Gradient descent finds such points with appropriate $\eta$ (proof below). The convergence rate specifies iterations $k$ required to reach a stationary point. We assume $\mathcal{L}$ is $L$-smooth (\cref{def:l-smooth}), standard in optimization theory and deep learning. Sobolev-type regularity handles non-smooth components (e.g., weak differentiability of ReLU), but incorporating these would complicate exposition without changing overall conclusions of this paper.

\begin{definition}[$L$-smooth functions]
\label{def:l-smooth}
A function $f$ is said to have $L$-Lipschitz continuous gradients, or be $L$-smooth, if, with appropriate norm $\|\cdot\|$, for all $x$ and $y$, \[
\|\nabla f(x) - \nabla f(y)\| \leq L\|x - y\|.
\]
\end{definition}

Application of the $L$-smooth property to a Taylor series approximation of $\mathcal{L} $ yields the following well-known descent lemma.

\begin{lemma}[Descent Lemma]
\label{lemma:descent-lemma}
For the gradient descent update rule of $L$-smooth functions, we have the following upper bound \[
\mathcal{L}_{\theta_{k + 1}} \leq \mathcal{L}_{\theta_k} - \frac{1}{2L} \|\nabla \mathcal{L}_{\theta_k}\|^2.
\]
\end{lemma}

The full proof appears in \citet{Wright_Recht_2022} and involves minimizing an $L$-smooth Taylor approximation of $\mathcal{L}$. The optimal learning rate is $\eta = 1/L$---the maximal rate before oscillation ($1/L < \eta \leq 2/L$) or divergence ($\eta > 2/L$). We assume convergent behaviour, as otherwise optimization never finds an stationary point. This justifies our assumption throughout that learning rate is at most $1/L$.

\subsection{Convergence Rate of GD on Smooth Functions}

We can use \cref{lemma:descent-lemma} to derive the rate of convergence of gradient descent on smooth functions. 

Summing over the descent lemma from iteration $i=0$ to $k-1$ results in \[
\mathcal{L}_{\theta_k} \leq \mathcal{L}_{\theta_0} - \frac{1}{2L} \sum_{i = 0}^{k - 1}  \|\nabla \mathcal{L}_{\theta_i}\|^2.
\]
Assuming $\mathcal{L}$ has a local minimum $\mathcal{L_{\theta_*}}$ such that \(
\mathcal{L_{\theta_*}} 
\leq \mathcal{L}_{\theta_k}
\) for all $k$ in some neighbourhood of $\theta_*$, then we have \[
\sum_{i = 0}^{k-1} \|\nabla \mathcal{L}_{\theta_i}\|^2 \leq 2L\left[\mathcal{L}_{\theta_0} - \mathcal{L}_{\theta_*}\right],
\] which implies that \(\lim_{k \to \infty} \|\nabla \mathcal{L}_{\theta_k}\| = 0\) otherwise the upper bound could not hold. Finally we have:
\[
\underset{0 \leq i \leq k-1}{\min} \|\nabla \mathcal{L}_{\theta_i}\|^2 
\leq 
\frac{1}{k} \sum_{i = 0}^{k-1} \|\nabla \mathcal{L}_{\theta_i}\|^2 
\leq  
    \frac{2L
        \left[ \mathcal{L}_{\theta_0} - \mathcal{L}_{\theta_*} \right] 
    }{k},
\] which follows from the fact that the minimum of a function is always less than or equal its mean.

This formula shows that after $k$ steps of gradient descent, the minimum gradient norm shrinks as 
\begin{equation*}
\underset{0 \leq i \leq k-1}{\min} \|\nabla \mathcal{L}_{\theta_i}\|
\leq  \sqrt{ 
    \frac{2L
        \left[ \mathcal{L}_{\theta_0} - \mathcal{L_{\theta_*}} \right] 
    }{k}
}.
\end{equation*}
We can restate this using another property that follows from the Taylor approximation of $L$-smooth functions that $\mathcal{L}_{\theta_0} - \mathcal{L}_{\theta_*} \leq \frac{L}{2}\|\theta_0 - \theta_*\|^2$ resulting in 
\begin{equation*}
\label{eq:sublinear-smooth-gd-convergence-rate}
\underset{0 \leq i \leq k-1}{\min} \|\nabla \mathcal{L}_{\theta_i}\|
\leq  \sqrt{ 
    \frac{
        \left[ L\|{\theta_0} - {\theta_*}\| \right]^2
    }{k}
}.
\end{equation*}

\subsection{Iteration Complexity}
From this we can show the iteration complexity of \cref{eq:sublinear-smooth-gd-convergence-rate}. The iteration complexity tells us how many steps $k$ are needed to achieve $\underset{0 \leq i \leq k-1}{\min} \|\nabla \mathcal{L}_{\theta_i}\| \leq \epsilon$. We can show the iteration complexity is given by \cref{prop:l-smooth-descent-iteration} from the main text with the following proof. The proposition is restated for clarity here first.
\begin{proposition}[$L$-smooth gradient descent iteration complexity]
\label{prop:l-smooth-descent-iteration-app}
Let $\mathcal{L}$ be $L$-smooth. For $L$-smooth gradient descent, the guarantee is that to obtain $\underset{0 \leq i \leq k-1}{\min} \|\nabla \mathcal{L}_{\theta_i}\| \le \epsilon$, the number of steps is given by the lower bound \(
 k \ge \left[ L\|\theta_0 - \theta_*\|  \right]^2/\epsilon^2
\) where $\theta_0$ is the initial model weights, $i$ is the sequence index, and $\theta_*$ is a stationary point. and $k$ is the number of iteration steps.
\end{proposition}

\begin{proof}
 By directly substituting $k = \frac{1}{\epsilon^2}\left[ L\|{\theta_0} - {\theta_*}\| \right]^2$ in \cref{eq:sublinear-smooth-gd-convergence-rate} we achieve $\underset{0 \leq i \leq k-1}{\min} \|\nabla \mathcal{L}_{\theta_i}\| \leq \epsilon$. Therefore, the iteration complexity of smooth GD is \(
 k \ge \left[ L\|\theta_0 - \theta_*\|  \right]^2/\epsilon^2
\).
    
\end{proof}

\subsection{Nesterov's Optimal Method Convergence Analysis}
\label{app:nesterov-convergence}

Nesterov's method belongs to a category of ``momentum''-based methods and is often motivated as a way to ``dampen'' the oscillatory behaviour introduced by gradient descent. Oscillatory behaviour occurs when there are both small and large curvature directions, causing nearly orthogonal gradients (This phenomenon is illustrated in \citealp{goh2017why}). Nesterov's optimal (or accelerated) method \citep{nesterov1983method} involves two modifications to the typical gradient descent method.

First, we introduce a momentum or drag term $\beta(\theta_i - \theta_{i - 1})$ which prevents overshooting and yields the so-called heavy ball method or Polyak's momentum \citep{polyak1964some}. Second, instead of evaluating our gradient at $\theta_i$, we evaluate the gradient at $\theta_i + \beta(\theta_i - \theta_{i-1})$ in order to correct for the path taken by momentum. This results in the following descent equations:
\begin{align*}
\phi_u &= \theta_i + \beta_i (\theta_i - \theta_{i - 1}) \\
\theta_{i + 1}  &= \phi_u - \alpha \nabla_{\theta_i} \mathcal{L}(\phi_u).
\end{align*}

Deriving the convergence rate of Nesterov's method and why it is optimal is outside the scope of this paper, since we are simply using the known convergence rate to show that best-case function convergence still depends on $L$. Interested readers are referred to \citet{Wright_Recht_2022}, from which we obtain the convergence rate result restated below.

\begin{equation}
\label{eq:nesterovs-convergence-rate}
\mathcal{L}_{\theta_k} - \mathcal{L}_{\theta_*} \leq \frac{
    2L
}{
    (k + 1)^2
} \|\theta_0 - \theta_*\|^2.
\end{equation}

Using this we can prove \cref{prop:nesterov-iteration} from the main text (restated below for clarity).

\begin{proposition}[Nesterov iteration complexity]
\label{prop:nesterov-iteration-app}
Let $\mathcal{L}$ be an $L$-smooth convex function. For Nesterov's optimal method, the guarantee is that to obtain $\mathcal{L}_{\theta_k} - \mathcal{L}_{\theta_*} \le \epsilon$, the number of steps is given by the lower bound \(
 k \ge\sqrt{ \left[ 2L\|\theta_0 - \theta_*\|^2 \right]/\epsilon} - 1  
\).
\end{proposition}

\begin{proof}
As above, our goal is to set $k$ such that \(\mathcal{L}_{\theta_k} - \mathcal{L}_{\theta_*} \leq \epsilon \). Let $\Delta = \|\theta_0 - \theta_*\|^2$. Set the L.H.S of \cref{eq:nesterovs-convergence-rate} to $\epsilon$ to find $k$.
\begin{align*}
\frac{2L\Delta}{(k + 1)^2} &= \epsilon \\
\sqrt{2L\Delta / \epsilon} - 1  &= k. \\
\end{align*}
From above, we see that we must set $k$ to $\sqrt{2L\Delta / \epsilon} - 1$ in order to achieve $\epsilon$ error in the L.H.S of \cref{eq:nesterovs-convergence-rate}. Proving the iteration complexity is $k \geq \sqrt{2L\Delta / \epsilon} - 1$.
\end{proof}

\subsection{Hessian Singular Value Dependence}

The results of our paper depend on the following property.
\begin{proposition}
\label{prop:L-lower-bound-app}
The smoothness quantifier $L \in \mathbb{R}$ for the function $\mathcal{L}$ is greater than or equal to the maximum singular value of the Hessian $H$ of that function. In other words the following formula holds
$L \geq \sigma_1(H_{\theta}^{\mathcal{L}})$.
\end{proposition}

\begin{proof}
For an $L$-smooth function $f$, smoothness is equivalent to a uniform operator-norm bound on the Hessian:
\[
\|\nabla^2 f(\theta)\|_2 \le L, \quad \text{for all } \theta
\]
(see, e.g., \citealp{Wright_Recht_2022}). Equivalently, for any unit vector $v$,
\[
|v^\top \nabla^2 f(\theta)\, v| \le L.
\]
Since $\nabla^2 f(\theta)$ is symmetric, all of its eigenvalues lie in the interval $[-L, L]$. Therefore, the largest singular value of the Hessian satisfies
\[
\sigma_1\!\left(H_\theta^{\mathcal L}\right)
= \|H_\theta^{\mathcal L}\|_2
= \max_i |\lambda_i(H_\theta^{\mathcal L})|
\le L.
\]
Equivalently, this implies the two-sided Loewner bound
\[
-LI \preceq \nabla^2 f(\theta) \preceq LI,
\]
which completes the proof.
\end{proof}

To connect this back to convergence rate, we present a classical demonstration of the impact of principal curvature ($L$) on number of steps needed to minimize a quadratic under gradient descent in \cref{fig:quadratic}.

\begin{figure*}[h!]
\centering
\includegraphics[width=1\linewidth]{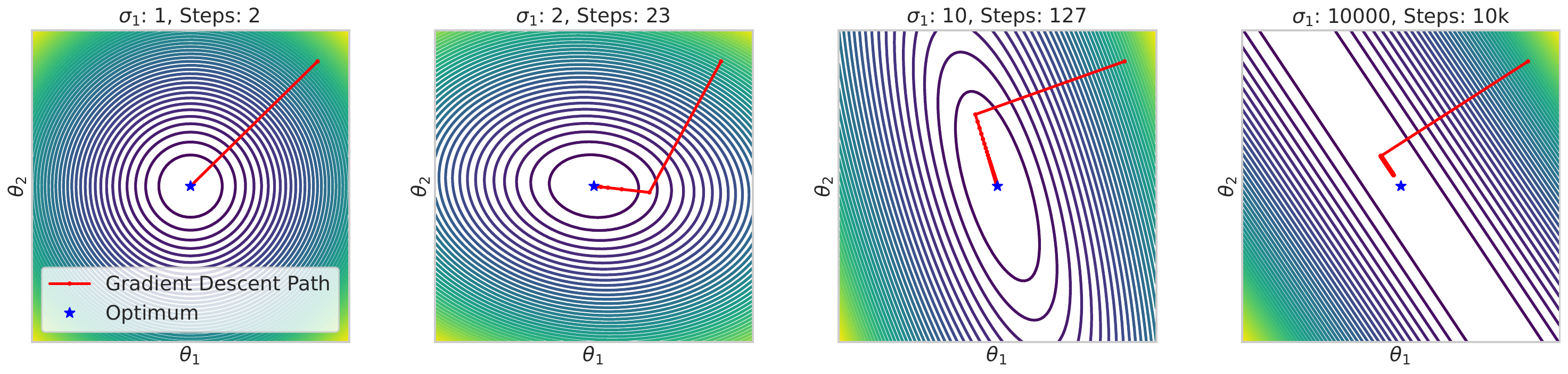}
\caption{
\label{fig:quadratic}
 Steps needed for optimizing a 2D convex quadratic function depend on Hessian spectral values,  $\sigma_1, \sigma_2$. Maximum convergent learning rate is used and $\sigma_{2}$ is fixed at 1.
}
\end{figure*}

\section{Review of Singular Value Inequalities}
\label{app:singular-value-inequalities}
To prepare for understanding the results of this paper and the proof techniques used, we briefly review singular value inequalities for general real matrices. \citet[Chapter~3]{Horn_Johnson_1991} provides a readable overview from which much of this material is sourced (\citep{bhatia2013matrix} is another good resource).

Recall the spectral value decomposition states that given a real-valued matrix $A \in \mathbb{R}^{m \times n}$ and index $q = \min(m, n)$, there is a diagonal matrix $\Sigma =\mathrm{diag}(\sigma_1, \ldots, \sigma_q)$ with $\sigma_1 \ge \dots \ge \sigma_q$ and two orthogonal matrices $U \in \mathbb{R}^{m \times m}$ and $V \in \mathbb{R}^{n \times n}$ such that $A = U\Sigma V^{\top}$. The individual elements $\sigma_i$ are the singular values of $A$, $U$ is known as the left singular vectors of $A$ and $V$ as the right singular vectors.

The Poincaré separation Theorem for singular values allows us to characterize matrix singular values by bounding them in terms of their submatrices, where a submatrix $A_r$ is given by deleting a total number of $r$ rows and/or columns from a matrix $A \in \mathbb{R}^{m \times n}$. We restate this result from \citet[3.1.3, p.~149]{Horn_Johnson_1991} where a proof can be obtained. We use this result extensively to find tractably computable submatrices of the Hessian and bound its singular values.

\begin{theorem}[Poincaré separation Theorem for singular values]
\label{thm:poincare-seperation}
Given $A \in \mathbb{R}^{m \times n}$ and submatrix $A_r$ the following holds 
\begin{equation*}
\label{eq:poincare-seperation}
\sigma_i(A) \ge \sigma_i(A_r) \ge \sigma_{i + r}(A), \quad i = 1, \ldots, \min\{m, n\}
\end{equation*}
where for $X \in \mathbb{R}^{p \times q}$, let $\sigma_i \equiv 0$ if $i \ge \min\{p, q\}.$
\end{theorem}

The Courant-Fischer Theorem allows us to have a variational characterization of singular values based on the subspaces of matrices. We use this later in the paper to construct lower bounds that consider the subspace alignment between matrices $A$ and $B$ in matrix product $AB$. For now, we can consider it as an alternative way to obtain singular values given the norm of a matrix-vector product:

\begin{theorem}[Courant-Fischer]
\label{thm:courant-fischer}
Let $A \in \mathbb{R}^{m \times n}$ and $\sigma_1(A) \geq \sigma_2(A) \ge \ldots$ be the ordered singular values of $A$ for all $1 \le i \le \min\{m, n\}$. Then
\begin{align*}
\label{eq:courant-fischer}
\sigma_i(A) &= \ \ \  \underset{\dim S = i}{\max} \quad \underset{
\substack{x \in S \\ \|x\|_2 = 1}
}{\min}  \|Ax\|_2 \\
&= \underset{\dim S = n - i + 1}{\min} \underset{
\substack{x \in S \\ \|x\|_2 = 1}
}{\max}  \|Ax\|_2.
\end{align*}
$\mathcal{S}$ are the subspaces of $\mathbb{R}^n$ with the indicated dimension $i$.
\end{theorem}
The proof of this is stated in \citet[3.1.2]{Horn_Johnson_1991}.

We now state two-sided inequalities for products and sums of matrices that we need for decomposing Hessian block (i.e., $\frac{\partial^2 \mathcal{L}}{\partial \theta_i \partial \theta_k})$ formulas. 

\begin{theorem}[Additive Singular Value Bounds]
\label{thm:additive-singular-value-bounds}
Let $A, B \in \mathbb{R}^{m \times n}$ be given and let $q = \min\{m, n\}$. The following inequality holds for decreasingly ordered singular values of A, B, and A + B. 
\begin{equation*}
\sigma_i(A) - \sigma_1(B) \leq \sigma_i(A + B) \leq \sigma_i(A) + \sigma_1(B).
\end{equation*}
For $1 \leq i \leq q$.
\end{theorem}

The proof of this and the Theorem below are presented in \citet[Th 3.3.16]{Horn_Johnson_1991}.

\begin{theorem}[Multiplicative Singular Value Bounds]
\label{thm:multiplicative-singular-value-bounds}
Let $A \in \mathbb{R}^{n \times m}$ and $B \in \mathbb{R}^{m \times p}$ and let $q = \min\{n, p\}$. The following inequality holds for decreasingly ordered singular values of A, B, and AB. 
\begin{equation*}
\sigma_i(A) \sigma_{q}(B) \leq \sigma_i(AB) \leq \sigma_i(A) \sigma_1(B).
\end{equation*}
For $1 \leq i \leq q$
\end{theorem}

Note that the bound for \cref{thm:multiplicative-singular-value-bounds} is symmetric, that is \(
    \sigma_{q}(A) \sigma_i(B) \leq \sigma_i(AB) \leq \sigma_1(A) \sigma_i(B)
\) also holds.

Tighter upper bounds can be derived from majorization results such as those of Lidskii and Ky Fan, but the focus of our paper is largely on lower bounds. Additional characterization can be obtained probabilistically using random matrix theory, but this is outside the scope of this paper. Unfortunately, the lower bounds presented can be vacuous for rank-deficient matrices, which are common in deep learning, and are generally uninformative for small singular values. There are a number of ways to remedy this, including Gel'fand-Naimark's results in \citet[III.20]{bhatia2013matrix}, but these require strong assumptions (e.g., positive definiteness in the case of Gel'fand-Naimark). To address this, we develop a stronger result without additional assumptions in \cref{app:principal-angle-lower-bound}.

\section{Principal Angle Lower Bound}
\label{app:principal-angle-lower-bound}

Now that we have provided a review of singular value inequalities in \cref{app:singular-value-inequalities}, we develop a tighter lower bound that depends on the overlap between the image of $B$ and the span of $A$'s left singular vectors. We develop this result because, in our setting, the Hessian blocks often factor into components that the designer cannot control ($A$), namely the data distribution or architectural components, and those the designer can control ($B$), for example rotations of $\theta_i$ that change the image of $B$ and scaling that changes the singular values.

Before stating this result, let us review subspaces, principal angles, orthogonal projections, and images. The image, or range, of a linear transformation represented by a matrix $B \in \mathbb{R}^{m \times n}$ and denoted $\mathrm{Im}(B)$ is the span of its column vectors, that is $\mathrm{Im}(B) := \mathrm{span}(B_{:,1}, \ldots, B_{:,n})$. To make things easier for later definitions, recall that by the singular value decomposition any matrix $B$ can be written as $B = U_B\Sigma_B V_B^{\top}$ with orthogonal $U_B \in \mathbb{R}^{m \times m}$ and $V_B \in \mathbb{R}^{n \times n}$ singular vector matrices. The image of $B$ is spanned by the first $r$ columns of the left singular vectors $\mathrm{Im}(B) := \mathrm{span}(U_{B_{:,1}}, \ldots, U_{B_{:,r}})$ where $r = \mathrm{rank}(B)$, a fact we use later. A subspace $\mathcal{M} \subset \mathcal{V}$ is a subset that is closed under linear combination of its vectors. The dimension of this subspace is the number of elements in the basis of $\mathcal{M}$, i.e., the linearly independent vectors that span $\mathcal{M}$. To illustrate, a rank-$k$ subspace of $\mathrm{Im}(B)$ can be represented by the first $k$ left singular vectors of $B$ denoted $U_{B_k}$. An orthogonal projection $\Pi_{E}$ is a linear transformation that takes any vector in the ambient space onto a subspace spanned by $E$. If $U_E \in \mathbb{R}^{m \times r}$ is the left singular vector matrix corresponding to nonzero singular values of $E$, then $\Pi_{E} = U_E U_E^{\top}$ (since $U_E$ is orthogonal, the term $(U_E^{\top}U_E)^{-1}$ in a standard projection drops). We use $\Pi_{U_k}$, which is a projection onto the rank-$k$ subspace spanned by the columns of $U_k$.

Finally, we define a distance between two subspaces using their \textbf{principal angles} (also called \textbf{canonical angles}). We illustrate with subspaces arising from a matrix product. Let $A \in \mathbb{R}^{p \times n}$ and $B \in \mathbb{R}^{n \times q}$, and let $\mathcal{A} \subseteq \mathbb{R}^n$ and $\mathcal{B} \subseteq \mathbb{R}^n$ be the subspaces spanned by the rows of $A$ and the columns of $B$, respectively. Both subspaces live in the same ambient space $\mathbb{R}^n$, which is precisely the space where the product $AB$ ``contracts.''

Let $r_A = \dim \mathcal{A} \leq \min\{p, n\}$ and $r_B = \dim \mathcal{B} \leq \min\{n, q\}$. Let $V_A \in \mathbb{R}^{n \times r_A}$ and $U_B \in \mathbb{R}^{n \times r_B}$ be matrices whose columns form orthonormal bases for $\mathcal{A}$ and $\mathcal{B}$ (e.g., obtained from the SVD of $A^\top$ and $B$, respectively).

\begin{definition}{(Principal Angle)}
\label{def:principal_angle}
The \textbf{principal angles} $\theta_1, \ldots, \theta_r$ between $\mathcal{A}$ and $\mathcal{B}$, where $r = \min\{r_A, r_B\}$, are defined by
\[
\theta_i = \cos^{-1}(\sigma_i),
\]
where $\sigma_1 \geq \cdots \geq \sigma_r \geq 0$ are the singular values of $V_A^\top U_B \in \mathbb{R}^{r_A \times r_B}$. These angles quantify how the row space of $A$ and column space of $B$ align, which directly affects the singular values of the product $AB$.
\end{definition}

There are several ways to utilize principal angles. Since we are going to construct lower bounds, we are concerned with the worst-case alignment (largest angle) between the subspaces $\mathcal{A}$ and $\mathcal{B}$; this would be $\sigma_{min}(V_A^{\top}U_B)$. We can further use a variational characterization to analyze the angle in terms of a minimization of a projection (which we need for a proof below). In this work, we use the cosine of the largest principal angle denoted $\prinangle(\cdot, \cdot) \in \mathbb{R}_+$.
\begin{align*}
    \prinangle(A,B) &= \sigma_{\min}(V_A^{\top}U_B) \\
    &= \underset{\|c\| = 1}{\min}\|V_AV^{\top}_A U_{B}\,c\| \tag{Variational Definition} \\
    &= \underset{\|c\| = 1}{\min}\|\Pi_{V_A}U_{B}\,c\| \tag{Definition of Projection} \\
    &= \underset{y \in \mathrm{Im}(U_B), \|y\| = 1}{\min}\|\Pi_{V_A}\,y\|. \tag{Change of Variable} 
\end{align*} While this characterization is somewhat more complex, we present it here to clarify a step we make in the proof below. There is only one more slight change in our notation to fully define the largest principal subspace angle: we consider the angle between $V_{A_r}$ truncated to a rank-$r$ subspace of $A$ and $U_{B_k}$ truncated to a rank-$k$ subspace of $B$. We are now ready to state and prove our presented on the next page.  

\clearpage

\begin{theorem}[Principal angle-weighted singular value bound]
\label{thm:principal-angle-lower-bound}
Let $A \in \mathbb{R}^{m \times n}, B \in \mathbb{R}^{n \times p}$. Let $V_{A_r}$ be the top-$r$ right singular-vector subspace of $A$. Let $\Pi_{V_{A_r}}$ denote the orthogonal projection onto the top-$r$ right singular subspace $V_{A_r}$ of $A$. Define \[
\prinangle(A_r,B_k) := \min_{y \in \mathrm{Im}(U_{B_k}),\: \|y\|=1} \|\Pi_{V_{A_r}}y\| \Leftrightarrow \sigma_{\min}(V_{A_r}^{\top}U_{B_k})
\] Where $U_{B_k}$ is the top-$k$ left singular vectors of $B$. Then for every $k$ with $1 \leq k \leq \mathrm{rank}(B)$ and $r \geq k$ we have \[
    \sigma_k(AB)\geq \prinangle(A_r,B_k) \sigma_r(A)\sigma_k(B).
\]
\end{theorem}
\begin{proof}
Let $V_{B_k} \subset \mathbb{R}^p$ denote the span of the top-$k$ right singular vectors of $B$. By the Courant-Fischer (\cref{thm:courant-fischer}) variational characterization for the $k$-th singular value of $AB$, \[
\sigma_k(AB) = \underset{
\substack{S \subset \mathbb{R}^p \\ \dim S = k}
}{\max} \quad \underset{
\substack{x \in S \\ \|x\| = 1}
}{\min}  \|ABx\|.
\]
Since the maximum over all $k$-subspaces is \textit{at least} the value on any \textit{particular} $k$-subspace, we may choose $S = V_{B_k}$ to obtain 
\begin{equation*}
\label{eq:palb_cf_ineq}
\sigma_k(AB) \geq \underset{
\substack{x \in V_{B_k} \\ \|x\| = 1}
}{\min}  \|ABx\|.
\end{equation*}
Fix any unit $x \in V_{B_k}$. By properties of the SVD of B, \[
\|Bx\| \geq \sigma_k(B).
\]
Define the unit vector $y := \frac{Bx}{\|Bx\|} \in \mathrm{Im}(U_{B_k})$. Then  
\begin{equation}
\label{eq:palbp_ineq_1}
\|ABx\| = \|A(Bx)\| = \|Bx\| \cdot \|Ay\| \geq \sigma_k(B) \cdot \|Ay\|.
\end{equation}
Expand $Ay$ in the right-singular basis of A. Since the squared norm satisfies 
\begin{align*}
\|Ay\|^2 &= \|U\Sigma V^{\top} y\|^2 = \|\Sigma V^{\top} y\|^2 \tag{Unitary Invariance}\\
&= \sum_{i =1}^n \sigma_i(A)^2 |v_i^\top y|^2 \geq  \sum_{i =1}^r \sigma_i(A)^2 |v_i^\top y|^2 \tag{Rank reduction to $V_{A_r}$} \\ 
&\ge \sigma_r(A)^2\sum_{i =1}^r  |v_i^\top y|^2  = \sigma_r(A)^2 \|\Pi_{V_{A_r}} y\|^2,
\end{align*}
we obtain 
\begin{equation}
\label{eq:palbp_ineq_2}
\|Ay\| \geq \sigma_r(A)\|\Pi_{V_{A_r}} y\|.
\end{equation}
Combining inequalities \cref{eq:palbp_ineq_1} and \cref{eq:palbp_ineq_2} gives, for every $x \in V_{B_k}$,
\[
\|ABx\| \geq \sigma_k(B)\sigma_r(A) \|\Pi_{V_{A_r}} y\|.
\]
By definition of $\prinangle(A_r, B_k)$ we have $\|\Pi_{V_{A_r}} y\| \ge \prinangle(A_r, B_k)$ for all unit $y \in \mathrm{Im}(U_{B_k})$. Hence \[
    \|ABx\| \geq \sigma_k(B) \sigma_r(A) \prinangle(A_r, B_k).
\]
Taking the minimum over all unit $x \in V_{B_k}$ and using the earlier Courant-Fischer inequality \cref{eq:palb_cf_ineq} yields \[
    \sigma_k(AB) \geq \sigma_k(B) \sigma_r(A) \prinangle(A_r, B_k),
\]
which is the desired bound.
\end{proof}

\subsection{Bound Tightness and Characterization}
\label{app:bound-tightness}

\begin{remark}[Principal Alignment Intuition]
Our lower bound captures how the geometric alignment between subspaces of $A$ and $B$ impacts the singular values of the product $AB$.

This results in the following bound characteristics:
\begin{enumerate}
    \item When $\prinangle(A_r, B_k) = 1$: the subspaces have perfect alignment in at least one direction, resulting in very tight bounds.
    \item When $\prinangle(A_r, B_k) = 0$ there exists at least one unit vector in $\mathrm{Im}(U_{B_k})$ whose projection onto $V_A$ is zero i.e., at least one orthogonal direction such that the lower bound becomes trivial.
    \item When $0 < \prinangle(A_r, B_k) < 1$: we have partial alignment where meaningful lower bounds can be obtained.
\end{enumerate}
\cref{fig:svlb-inequality-illustration} below illustrates three examples and provides an intuitive characterization of our principal angle lower bound.
\end{remark}

\begin{figure}[h!]
    \centering
\includegraphics[width=1\linewidth]{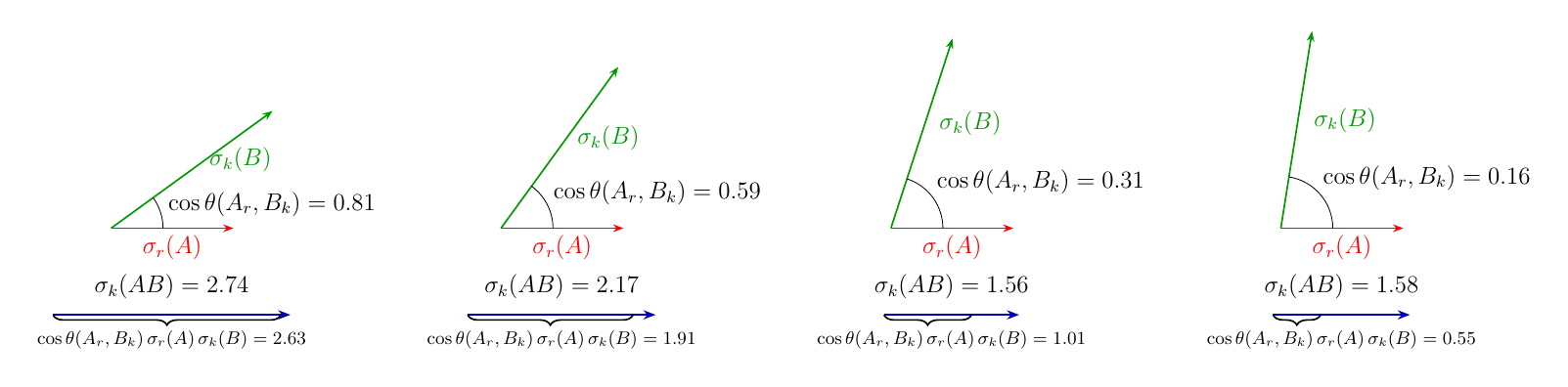}
    \caption{Illustration of the principal angle lower bound}
    \label{fig:svlb-inequality-illustration}
\end{figure}

We provide a numerical analysis validating the bounds in \cref{tab:bound-under-rotation} under 2D and 3D rotation and \cref{tab:bound-k-r-rand-def} illustrating varying $k$ and $r$ under rank deficiency. Note that while these bounds are similar, at least in spirit, to the results of Davis and Kahan on rotation of eigenvectors by a perturbation \citep{davis1970rotation} and general work on principal angles \citep{knyazev2012principal}, to the best of our ability we were not able to find this bound in the literature of matrix analysis and believe that it is novel. Some similar work was presented in \citet{ipsen1995angle,miao1992principal,kaur2023new}.

\begin{remark}{Bound Tightening via Supremum}
It follows from these analysis that the bound can be tightened by taking \[
  \sigma_k(AB)\geq \sup_{r \geq k} \prinangle(A_r,B_k) \sigma_r(A)\sigma_k(B).
\] for a specific $k$ or for the tightest relationship \[
  \sup_{r \geq k, k} \prinangle(A_r,B_k) \sigma_r(A)\sigma_k(B),
\] for a $\sigma_k(AB)$ determined by the supremum. If $k > 1$ then we can use the fact that $\sigma_1 \geq \sigma_2 \geq \ldots$ to bound all $\sigma_{k - i}$ where $i < k$. We will use this supremum tightening later in the paper.
\end{remark}

\begin{table}[h]
\centering
\sisetup{round-mode=places, round-precision=2, table-format=2.4}
\begin{tabular}{c S[table-format=1.4] S[table-format=1.4] S[table-format=1.4] S[table-format=1.4] | c S[table-format=1.4] S[table-format=1.4] S[table-format=1.4] S[table-format=1.4]}
\toprule
\multicolumn{5}{c|}{Example 2: 2D Matrices} & \multicolumn{5}{c}{Example 3: 3D Matrices} \\
\midrule
{$\theta$} & {$\prinangle(A_r, B_k)$} & {Bound} & {Actual} & {Ratio} &
{$\theta$} & {$\prinangle(A_r, B_k)$} & {Bound} & {Actual} & {Ratio} \\
\midrule
0  & 1.0000 & 2.7929 & 2.7929 & 1.0000 & 0  & 1.0000 & 7.0951 & 7.0951 & 1.0000 \\
18 & 0.9511 & 2.6562 & 2.7694 & 1.0426 & 18 & 0.9674 & 6.8636 & 6.9369 & 1.0107 \\
36 & 0.8090 & 2.2595 & 2.7066 & 1.1979 & 36 & 0.8727 & 6.1917 & 6.4996 & 1.0497 \\
54 & 0.5878 & 1.6416 & 2.6260 & 1.5997 & 54 & 0.7252 & 5.1453 & 5.8873 & 1.1442 \\
72 & 0.3090 & 0.8630 & 2.5579 & 2.9640 & 72 & 0.5393 & 3.8267 & 5.2667 & 1.3763 \\
90 & 0.0000 & 0.0000 & 2.5312 & 0.0000 & 90 & 0.3333 & 2.3650 & 4.8151 & 2.0359 \\
\bottomrule
\end{tabular}
\caption{Numerical analysis of \cref{thm:principal-angle-lower-bound} under 2D and 3D matrix rotation. Each matrix has Gaussian random entries.}
\label{tab:bound-under-rotation}
\end{table}

\begin{table}[h]
\centering
\sisetup{round-mode=places, round-precision=2, table-format=1.4}
\begin{tabular}{c c S[table-format=1.4] S[table-format=1.4] S[table-format=1.4] S[table-format=1.4]}
\toprule
$r$ & $k$ & {$\prinangle(A_r, B_k)$} & {Bound} & {Actual} & {Ratio} \\
\midrule
1 & 1 & 0.3694 & 4.9781 & 6.7136 & 1.3486 \\
2 & 1 & 0.4243 & 3.9340 & 6.7136 & 1.7066 \\
2 & 2 & 0.4243 & 2.1083 & 4.7753 & 2.2650 \\
3 & 1 & 0.8142 & 3.0090 & 6.7136 & 2.2312 \\
3 & 2 & 0.8072 & 1.5988 & 4.7753 & 2.9868 \\
3 & 3 & 0.1830 & 0.1968 & 0.3970 & 2.0173 \\
4 & 1 & 0.8144 & 0.0000 & 6.7136 & 0.0000 \\
4 & 2 & 0.8077 & 0.0000 & 4.7753 & 0.0000 \\
4 & 3 & 0.1922 & 0.0000 & 0.3970 & 0.0000 \\
4 & 4 & 0.1861 & 0.0000 & 0.0000 & 0.0000 \\
\bottomrule
\end{tabular}
\caption{Numerical analysis of \cref{thm:principal-angle-lower-bound}. Grid of $(r,k)$ values for $A \in \mathbb{R}^{4 \times 6}, B \in \mathbb{R}^{6 \times 4}$, both matrices are made to be Rank deficient (Rank 3). Entries are drawn Gaussian random.}
\label{tab:bound-k-r-rand-def}
\end{table}

\subsection{Products of 3 matrices}

We conclude this section by stating a lemma for application of the principal angle lower bound in the case of a product of 3 matrices, to clarify how to apply the bound for isolating matrices in a product. A more general statement could be obtained for $n$ matrices by the same type of induction, but is not necessary for this paper.

\begin{lemma}
\label{lemma:principal-angle-lower-bound-3-matrices}
$\sigma_k(ABC)$ can be lower bounded from below in the following ways:
\begin{enumerate}
    \item $\sigma_{r_1}(A)\sigma_{r_2}(B)\sigma_k(C) \prinangle(AB_{r_1},C_k)\prinangle(A_{r_2}, B_{r_1})$ where $r_2 \geq r_1 \ge k$.
    \item $\sigma_{r_1}(A)\sigma_{k}(B)\sigma_{r_2}(C) \prinangle(A_{r_1},BC_k)\prinangle(B_k, C_{r_2}) $ where $r_1 \geq k$ and $r_2 \ge k$.
    \item $\sigma_{k}(A)\sigma_{r_2}(B)\sigma_{r_1}(C) \prinangle(A_{k},BC_{r_1})\prinangle(B_{r_2}, C_{r_1})$ where $r_2 \geq r_1 \ge k$.
\end{enumerate}
\end{lemma}
\begin{proof}
Each case follows from the associativity of matrix multiplication and recursive application of \cref{thm:principal-angle-lower-bound}. We only show this explicitly for the first case.

First, apply \cref{thm:principal-angle-lower-bound} to the product of two matrices $(AB)$ and C.
\begin{align*}
\sigma_k((AB)(C)) &\geq \sigma_{r_1}(AB)\sigma_k(C) \prinangle(AB_{r_1},C_k).
\end{align*}
Now for $\sigma_{r_1}(AB)$ apply \cref{thm:principal-angle-lower-bound} again to get the result
\begin{align*}
\sigma_{r_1}(AB)\sigma_k(C) \prinangle(AB_{r_1},C_k) \geq \sigma_{r_1}(A)\sigma_{r_2}(B)\sigma_k(C) \prinangle(AB_{r_1},C_k)\prinangle(A_{r_1}, B_{r_2}).
\end{align*}
The recursive application of the Theorem requires that $r_2 \geq r_1 \ge k$.

For the next two cases notice that we can apply the Theorem in different order due to associativity and because \cref{thm:principal-angle-lower-bound} is symmetric.
\begin{align*}
\sigma_k\left((A)(BC)\right) &\geq \sigma_{r_1}(A)\sigma_k(BC) \prinangle(A_{r_1},BC_k)  \\
&\geq \sigma_{r_1}(A)\sigma_{k}(B)\sigma_{r_2}(C) \prinangle(A_{r_1},BC_k)\prinangle(B_k, C_{r_2}), 
\end{align*} with $r_1 \geq k$ and $r_2 \ge k$.
and 
\begin{align*}
\sigma_k\left((A)(BC)\right) &\geq \sigma_{k}(A)\sigma_{r_1}(BC) \prinangle(A_k,BC_{r_1})  \\
&\geq \sigma_{k}(A)\sigma_{r_2}(B)\sigma_{r_1}(C) \prinangle(A_{k},BC_{r_1})\prinangle(B_{r_2}, C_{r_1}),
\end{align*} with $r_2 \geq r_1 \ge k$
which completes the proof. 
\end{proof}

\section{Hessian Lower Bound Details}
\label{app:hessian-lower-bound-details}

In this section we develop the analytical tools necessary for characterizing the Hessian in deep neural models in terms of their parameterization. We first present example Hessian derivations for MLPs, CNNs, and self-attention in \cref{app:example-hessian-derivations}. We then develop the tools that apply the results in \cref{app:singular-value-inequalities} and \cref{app:principal-angle-lower-bound} to the context of the Hessian of a deep neural network in \cref{app:hessian-lower-bound}, enabling us to prove the main result \cref{thm:hessian-singular-value-bound}. At a high level, we will rely on a \textbf{block sufficiency} argument that states that if a Hessian block (i.e., submatrix) exists then other arguments that follow based on the block apply to the full Hessian in some sense.

Second-order partial derivatives of neural networks contain many computational artifacts that can obscure presentation. For clarity of exposition in this paper, we consider example decompositions of the network function $f: x \in \mathbb{R}^d \to \hat{y} \in \mathbb{R}$ without regard to the loss function $\mathcal{L}: \hat{y}, y \in \mathbb{R} \ \mapsto \ell \in \mathbb{R}$. This also allows us to avoid needing to comment on specific loss functions (discussion on this in \cref{rem:network-to-loss-hessian}). We don't present analysis of multi-head attention, layer normalization, and residuals since the content of our self-attention example would not meaningfully change as these amount to either scaling in the activation magnitude (layer normalization, \citealp{xiong2020layer, xu2019understanding}) which is unchanged in our methods, or additive factors in the case of residuals which are often dropped.

\subsection{Example Derivations}
\label{app:example-hessian-derivations}

We now consider example Hessian blocks for feed-forward, self-attention, and convolution architectures to establish that for common machine learning architectures, there exist Hessian blocks that contain, as part of their formula, matrix multiplication factors with at least one weight matrix factor. This is both the necessary and sufficient condition for our singular value inequalities we develop later to hold.

\paragraph{Multi-layer Perceptron (MLP)}
Consider the following three-layer multilayer perceptron, \( f: x \mapsto \theta_3^{\top} \phi_2 (\theta_2\,  \phi_1(\theta_1 x))\)), where weight matrices are given by $\theta_1 \in \mathbb{R}^{m \times d}$ and $\theta_2 \in \mathbb{R}^{p \times m}$ with linear layer $\theta_3 \in \mathbb{R}^{p}$. $\phi$ is the element-wise activation function $\phi_i: \mathbb{R}^d \to \mathbb{R}^d$. In this paper, we assume $\phi_i$ is approximately $1$-Lipschitz which captures almost all activation functions used in practice.

Given this architecture we have the following intermediate variables defined:
\begin{equation}
\label{eq:mlp-annotated}
f(x) = \theta_3^{\top} \overbrace{
\phi_2 (
\underbrace{\theta_2\,  \overbrace{\phi_1(\underbrace{\theta_1 x}_{h_1}}^{z_1}))}_{h_2}
}^\text{$z_2$}.
\end{equation}

We define $\phi'(z) = D_z := \mathrm{Diag}(\phi'(z_i))$ as well as $D_z^{\prime\prime} := \mathrm{Diag}(\phi''(z_i))$. We can now specify $\nabla_{\theta_1} f$
\begin{align*}
\nabla_{\theta_1} f &= \frac{\partial  h_1}{\partial \,\mathrm{vec}(\theta_1)}^{\top} \frac{\partial  z_1}{\partial h_1}^{\top} \frac{\partial  h_2}{\partial z_1}^{\top}  \frac{\partial  z_2}{\partial h_2}^{\top}  \frac{\partial  f}{\partial z_2} \\
&=  (x^{\top} \otimes I_{m})^{\top} (D_{z_1} \theta_2^{\top}    D_{z_2} \theta_3) \tag{ Vectorized form} \\
&=  (D_{z_1} \theta_2^{\top}    D_{z_2} \theta_3) \: x^{\top}. \tag{Reshaped by $\text{vec}(uv^{\top})= v \otimes u$}
\end{align*}

For the Hessian we use vectorization.
\begin{align*}
\frac{\partial^2 f}{\partial \theta_3 \partial \,\mathrm{vec}(\theta_1)} &= \frac{\partial}{\partial \theta_3}  \frac{\partial f}{\partial \,\mathrm{vec}{(\theta_1)}}  \\ &= (x^{\top} \otimes I_{m})^{\top} (D_{z_1} \theta_2^{\top}    D_{z_2}). \tag{Numerator Layour}
\end{align*}

Since this Hessian block is clearly dependent on the weight matrix $\theta_2$, we have shown what we set out to show.

\paragraph{Convolutional Neural Network (CNN)} Let us consider a one-dimensional convolution setting. Consider a two-layer CNN with kernel matrices $K_1 \in \mathbb{R}^{m \times d}$ and $K_2 \in \mathbb{R}^{p \times m}$. We assume these kernel matrices have the Toeplitz structure induced by convolution with weight-sharing. This gives the following function \(f(x) = \theta^{\top} \phi_2\left(K_2 \phi_1 \left(  K_1x \right) \right) \) for a vectorized input $x \in \mathbb{R}^d$ where $\theta \in \mathbb{R}^p$ is a linear layer. It is easy to see that this results in the same Hessian matrix as above with $K_2$ taking the place of $\theta_2$. This is not surprising given that a CNN, in this low-dimensional case, is a special case of an MLP with sparse weight structure. Higher-dimensional kernels with additional channels require a specialized tensor treatment that does not change our conclusions.

\paragraph{Self-Attention Transformer} For this derivation we consider a single self-attention layer with a linear layer $\theta \in \mathbb{R}^{p}$. 
\begin{equation*}
\theta^{\top}VX\,\mathrm{Softmax}\underbrace{\left[(KX)^{\top}Qx_i
\right]}_{h_1}.
\end{equation*}
$K, Q, V \in \mathbb{R}^{n \times d}$ are the Key, Query, and Value weight matrix parameters respectively, whose dimensions are the number of inputs $n$ and dimensions $d$ for a set of input values $X \in \mathbb{R}^{d \times n}$. We consider only the attention values for a single query $x_i$ attending to keys/values from $X$. $\mathrm{Softmax}$ is an element-wise function $\mathrm{Softmax}: \mathbb{R}^n \to \mathbb{R}^n,\, x \mapsto [e^{x_i}/\sum_j^{|x|} e^{x_j}]_i^{|x|}$. We denote $s[k]$ as the activation corresponding to $i = k$, with $s$ being the vector of activations.

First we obtain the Jacobian of $s$ by examining the gradient of $s[k] = e^{x_k}/\sum_j^{|x|} e^{x_j}$. 
\begin{align*}
\frac{\partial s[k]}{\partial x_i} 
&= \frac{1}{\sum_{j=1}^{|x|} e^{x_j}} \, \frac{\partial}{\partial x} e^{x_k} 
   + e^{x_k} \, \frac{\partial}{\partial x} \left(\frac{1}{\sum_{j=1}^{|x|} e^{x_j}}\right) 
   && \text{(Product Rule)} \\
&= \frac{e^{x_k}}{\sum_{j=1}^{|x|} e^{x_j}} \,
   -  \,  \frac{ e^{x_k}  e^{x_i}}{\left( \sum_{j=1}^{|x|}  e^{x_j}\right)^2}
   \\
\frac{\partial s[k]}{\partial x_i} 
&= \begin{cases}
s[k]\bigl(1 - s[k]\bigr), & i = k, \\[6pt]
-\,s[k]\,s[i], & i \neq k.
\end{cases}
\end{align*}

By stacking this column-wise, we have the following Jacobian: \(J_{s} = \left(\mathrm{Diag}(s) - ss^\top \right).\)

Returning to self-attention we have  
\begin{align*}
\nabla_Q f &=  \frac{\partial h_1}{\partial Q}^{\top} \frac{\partial s}{\partial h_1} ^{\top} \frac{\partial f}{\partial s} \\
 &=   (x_i \otimes KX) 
\left(\mathrm{Diag}(s) - ss^\top \right)^\top
 (VX)^\top \theta \\ 
 &=    KX 
\left(\mathrm{Diag}(s) - ss^\top \right)^{\top} 
 (VX)^{\top} \theta x_i^{\top}.
\end{align*}

The Hessian block of $H_{\theta, Q}^{\mathcal{L}}$ follows simply from dropping $\theta$ in the vectorized form resulting in \[
(x_i \otimes I_k) \left( KX 
\left(\mathrm{Diag}(s) - ss^\top \right)^{\top}
 (VX)^\top \right),
\] where $I_k$ is the key matrix row dimension. Clearly this block depends on the matrix products of the weight matrices $K$ and $V$. We will explain later in \cref{example:self-attention} how this relates to structure needed for the Hessian lower bound.

\subsection{Hessian Lower Bound}
\label{app:hessian-lower-bound}

In this section we will establish the main theoretical result of the paper. 

To characterize the Hessian in deep neural models in terms of their parameterization, we develop two results: (1) Hessian spectral values can be bounded in terms of the second derivatives of individual blocks corresponding to individual weight matrices (\cref{corr:block-interlacing-for-hessian}); (2) supposing these blocks contain weight matrices as product factors, the Hessian can be bounded by singular values of these weight matrices (\cref{thm:hessian-singular-value-bound}).

\begin{corollary}
\label{corr:block-interlacing-for-hessian}
Let $\nabla^2_{\theta}\mathcal{L} \in \mathbb{R}^{n \times n}$ be the Hessian of $\mathcal{L}$ where $n = |\theta|$ is the number of parameters in that function. By definition, this Hessian admits the following block decomposition $\nabla^2_{B} \mathcal{L} := \nabla^2_{\mathrm{vec}({\theta_i}), \mathrm{vec}({\theta_j})} \mathcal{L} \in \mathbb{R}^{p \times q}$ where $p$ and $q$ are the number of parameters in the vectorized parameter matrices $p = |\theta_i|$ and $q = |\theta_j|$. Choose any block $B$ from the Hessian, The singular value $\sigma_k(\nabla^2_{\theta}\mathcal{L})$ for any $k$ has the following interlacing inequalities 
\begin{equation*}
\label{eq:block-interlacing}
\sigma_k(\nabla^2_{\theta}\mathcal{L}) \geq \sigma_k(\nabla^2_{B} \mathcal{L}) \geq \sigma_{k + r}(\nabla^2_{\theta} \mathcal{L}), \quad k = 1, \ldots, n
\end{equation*}
where $r := (n - p) + (n - q)$ and we set $\sigma_k(\cdot) \equiv 0$ if $k > \min\{p, q\}$
\end{corollary}
\begin{proof}
This is a straightforward application of \cref{thm:poincare-seperation} to the block decomposition of the Hessian.
\end{proof}

This corollary implies that the largest singular value of the Hessian is lower bounded by the singular values of any of its block matrices. Therefore our analysis can proceed from choosing an arbitrary block which is the partial second derivatives of one or more weight matrices. This result already has significant computational implications since we can now find lower bounds of the Hessian by computing only subsets of the Hessian. We can now proceed to the main result of the paper which is that the Hessian of deep networks can be characterized by the singular values of its parameters.

\begin{theorem}[Hessian Singular Value Lower Bound]
\label{thm:hessian-singular-value-bound-app}
Let $\nabla^2_{\theta}\mathcal{L} \in \mathbb{R}^{n \times n}$ be the Hessian of $\mathcal{L}$ with network function $f_\theta(x): x \mapsto  \theta_{n+1} \circ \phi_{n} \circ \theta_{n} \circ \cdots \circ \phi_1 \circ \theta_1x$. Assume there exists a submatrix of the Hessian          $\nabla^2_{\theta_i,\theta_j} \mathcal{L} :=  \left[\frac{\partial^2 \mathcal{L}}{\partial \theta_i \partial \theta_j}\right]_{i, j} \in \mathbb{R}^{p \times q}$ where $p = |\theta_i|, q = |\theta_j|$ that has the following matrix product structure $ABC$ with $B := \theta_k$ and arbitrary matrices $A,C$ with the index $k$ depending on which $i,j$ were chosen. Then $\sigma_1(\nabla^2_{\theta}\mathcal{L})$ is bounded below as \[
\sigma_1(\nabla^2_{\theta}\mathcal{L}) \geq \sup_{r_1,r_2 \geq 1} \sigma_{r_1}(A)\sigma_{1}(B)\sigma_{r_2}(C) \prinangle_1\prinangle_2,
\] where $\prinangle(\cdot, \cdot)$ is the cosine of the largest principal angle between the subspace as defined in \cref{def:principal_angle} with $\prinangle_1 := \prinangle(A_{r_1},BC_1)$ and $\prinangle_2 :=\prinangle(B_1, C_{r_2})$.
\end{theorem}
\begin{proof}
Assuming we are sufficiently far aware from the task minima, we may use \cref{rem:network-to-loss-hessian} to assume the loss Hessian is controlled from below by the network Hessian. This assumption is weak given that we are in a fine-tuning regime that requires this to be the case.
From \cref{corr:block-interlacing-for-hessian}, we have that \[
\sigma_1(\nabla^2_{\theta}f) \ge \sigma_1(\nabla^2_{\theta_i, \theta_j} f).
\] By assumption (and demonstration in \cref{example:mlp}) we had $\nabla^2_{\theta_i, \theta_j}f := A\theta_kC$ which results in \[
 \sigma_1(\nabla^2_{\theta}f) \geq \sigma_1(A\theta_kC).
\] We now apply \cref{lemma:principal-angle-lower-bound-3-matrices} using that $\theta_k := B$ developed before attaining the result,
\begin{align*}
 \sigma_1(\nabla^2_{\theta}f) & \geq\sigma_1(A\theta_kC)  \\
 &\geq \sigma_{r_1}(A)\sigma_1(\theta_k) \sigma_{r_2}(C)\prinangle{(A_{r_1},(\theta_kC)_1)} \prinangle{(\theta_{k_1},C_{r_2})}. \tag{\cref{lemma:principal-angle-lower-bound-3-matrices}} \\
\end{align*}
Finally, in order to tighten the bound as much as possible a supremum is taken over $r_1$ and $r_2$.
\end{proof}

\begin{remark}[Bound Tightness and Singular Value Selection]
\label{remark:bound-tightness}
The supremum over $r_1, r_2 \geq 1$ in \cref{thm:hessian-singular-value-bound-app} allows for adaptive tightening when principal angle factors are deficient. Specifically, if $\prinangle(A_1, (BC)_1)$ is small due to near-orthogonality of the leading singular subspaces, selecting $r_1 > 1$ may yield a tighter bound provided
\[
\sigma_{r_1}(A)\prinangle(A_{r_1}, (BC)_1) > \sigma_1(A)\prinangle(A_1, (BC)_1).
\]
More generally, for any $k \leq \min(r_1, r_2)$, we have
\[
\sigma_1(\nabla^2_\theta \mathcal{L}) \geq \sup_{\substack{r_1, r_2 \geq k}} \sigma_{r_1}(A)\sigma_k(B)\sigma_{r_2}(C)\prinangle_1\prinangle_2,
\]
though our focus on $\sigma_1(B)$ reflects our concern with controlling the dominant convergence rate from below rather than providing a complete spectral characterization. Upper bounds and analysis of $\sigma_j(\nabla^2_\theta \mathcal{L})$ for $j > 1$ follow from analogous arguments but lie outside our present scope.
\end{remark}

\begin{remark}[Additive Factors and Non-Commutative Ring Structure]
\label{remark:additive-ring-structure}
While our Theorem is valid precisely for the feedforward structure we assumed, for other architectures and other Hessian blocks the product structure $ABC$ in \cref{thm:hessian-singular-value-bound-app} does not preclude the presence of additive terms in the Hessian submatrix. When $\nabla^2_{\theta_i, \theta_j} f = ABC + D$, the singular value bounds can be extended via \cref{thm:additive-singular-value-bounds} using Weyl-type inequalities. We omit this for clarity, as the product structure is enough to illustrate scaling behaviour under spectral reparameterization. We acknowledge that these may be subtractive in which case we refer readers to \cref{rem:L2-control-conditions} for a discussion on conditions where the bounds for both $L_1$ and $L_2$ are tight and where we are still able to guarantee permissive bounds.

More fundamentally, the algebra of matrices over $\mathbb{R}$ forms a non-commutative ring, which constrains the compositional structures that can arise in Hessian submatrices to sums and products (though they may be represented through more exotic structures like Kronecker and Hadamard products). Since neural network computations compose linearly and through element-wise non-linearities, the mixed partial derivatives $\frac{\partial^2 f}{\partial \theta_i \partial \theta_j}$ inherit this ring structure---no additional operations beyond addition and (non-commutative) multiplication are available. This observation justifies our focus on product and sum decompositions as the exhaustive set of tractable structural forms.
\end{remark}

\begin{remark}[From network Hessian to loss Hessian]\label{rem:network-to-loss-hessian}
By the multivariate chain rule, the loss Hessian decomposes as
\[
\nabla^2_\theta \mathcal{L}
= H_{GN} + H'_f,
\qquad
H_{GN} := J^\top \nabla^2_{\hat{y}} \mathcal{L}\, J,
\quad
H'_f := (\nabla_{\hat{y}} \mathcal{L}) \, \nabla^2_\theta f,
\]
where $J = \nabla_\theta f$ and $H_{GN}$ is the positive semi-definite Gauss--Newton term (see discussion in \citet{schraudolph2002}).

Our spectral reparameterization explicitly controls $\sigma_1(\nabla^2_\theta f)$, which enters $H'_f$ scaled by the loss gradient $\nabla_{\hat{y}} \mathcal{L}$. When the Gauss--Newton term is small relative to $H'_f$—as is typical during early or mid training on high-loss examples, where gradients are large and the model is far from a well-fit solution—the operator norm of $H'_f$ dominates the spectrum of $\nabla^2_\theta \mathcal{L}$. In this regime, inflating $\sigma_1(\nabla^2_\theta f)$ yields a corresponding lower bound on $\sigma_1(\nabla^2_\theta \mathcal{L})$.

Conversely, near convergence on well-fit examples, the Gauss--Newton term can dominate and partially offset the contribution of $H'_f$.
\end{remark}

\subsubsection{Specific Architectural Examples}
\label{app:specific-architectural-examples}

We now proceed analytically by showing that, under certain conditions, MLP, CNNs, and Self-Attention networks satisfy the assumption that there exists a second partial derivative block with the structure $A\theta_iB$. First, for illustration, we show that linear regression with a weight coefficient vector does not satisfy this assumption.

\begin{example}[Typical Linear Regression is a Non-example]
\label{example:linear-regression}
Estimation of a typical linear regression model parameterized by $\theta \in \mathbb{R}^d$ is formulated in the following way \[
\theta^* = \underset{\theta \in \mathbb{R}^d}{\arg\min}  \frac{1}{2}\|X\theta - y\|_2^2,
\] where $X \in \mathbb{R}^{n \times d}$ and $y \in \mathbb{R}^{n}$ are the input feature matrix and associated prediction targets. Now consider the Hessian w.r.t.\ the loss $\mathcal{L} = \frac{1}{2}\|X\theta - y\|_2^2$ and parameters $\theta$ \begin{align*}
\nabla^2_{\theta} \mathcal{L} &= \nabla_\theta \left[ \nabla_\theta \frac{1}{2}\|X\theta - y\|_2^2 \right]  \\
&= \nabla_\theta \: X^{\top}X\theta - X^{\top}y \\
&=  X^{\top}X.
\end{align*}

The Hessian above doesn't match the structure we assumed for \cref{thm:hessian-singular-value-bound} so typical linear regression is not an example.
\end{example}

We now show two examples where \cref{thm:hessian-singular-value-bound} applies. Generally, for deep neural networks with a certain number of layers, the algebraic structure of our assumption will hold (see the remark above on Matrix rings).


\begin{example}[MLP]
\label{example:mlp}
For the three-layer MLP in \cref{app:example-hessian-derivations}, we showed that Hessian admits the following block (vectorization notation is omitted for clarity):
\[
\frac{\partial^2 f}{\partial \theta_3 \partial \theta_1} =
(x^{\top} \otimes I_{m})^{\top} (D_{z_1} \theta_2^{\top}    D_{z_2}).
\]
Since a Kronecker product produces a matrix, and matrix multiplication is associative we assign $A: = (x^{\top} \otimes I_{m})^{\top}D_{z_1}$. Now we simply apply \cref{thm:hessian-singular-value-bound} resulting in \[
\sigma_1(\nabla^2_{\theta}\mathcal{L} ) \geq
\sigma_{r_1}(A)\sigma_1(\theta_2^{\top}) \sigma_{r_2}(D_{z_2})\prinangle(A_{r_1}, (\theta_2^{\top} D_{z_2})_1) \prinangle((\theta_2^{\top})_1, (D_{z_2})_{r_2}).
\]
If we increase the number of layers, we can always find a structure like this.
\end{example}

Some CNNs may be considered a special case of MLPs as we showed in \S~\ref{app:example-hessian-derivations}.

\begin{example}[Self-attention]
\label{example:self-attention}
We can proceed using the method of analysis as the MLP. First, we have from \S~\ref{app:example-hessian-derivations} that for a single self-attention layer with a linear layer $\theta \in \mathbb{R}^n$ we have (vectorization notation is omitted for clarity)  \[
\frac{\partial^2 f}{\partial \theta \partial Q} =  (x_i \otimes I_k) KX
\left(\mathrm{Diag}(s) - ss^\top \right)
 X^\top V^{\top}.
\] The analysis is simplified by setting $A := (x_i \otimes I_k) KX
\left(\mathrm{Diag}(s) - ss^\top \right)X^{\top}$ and then directly applying \cref{thm:hessian-singular-value-bound} resulting in \[
\sigma_1(\nabla^2_{\theta}\mathcal{L}) \geq
\sigma_{r_1}(A)\sigma_1(V) \prinangle(A_{r_1}, V_1),
\] which shows that a self-attention layer with a linear layer is admitted by \cref{thm:hessian-singular-value-bound}.
\end{example}

A numerical analysis of this bound is provided in \cref{tab:spectral_alignment_main}. Note that this table reports averages for each quantity to illustrate the mean relationships. For more complicated and realistic architectures for modern deep networks we would be unable to provide analytical closed-form Hessian decomposition or exact numerical analysis. We refer the reader to the empirical parts of the paper to understand whether the bounds convincingly control $\sigma_1(H)$.

\begin{table}
\caption{Numerical analysis of our bound broken into various quantities. LB indicates  $\sigma_{r_1}(A)\sigma_{1}(B)\sigma_{r_2}(C)\alpha_A\alpha_B$ with principal subspace angles $\alpha_A = \prinangle(A_{r_1},BC_1)$ and $\alpha_B = \prinangle(B_1, C_{r_2})$. Mean across 30 seeds is reported for each entry, std ranges omitted for clarity. Values of $A$, $B$, and $C$ are found in \cref{app:experimental-details}. The bound always holds.}
\label{tab:spectral_alignment_main}
\resizebox{\textwidth}{!}{%
\begin{tabular}{llccccccc}
\toprule
Network & $\sigma_1(B)$ & $\sigma_1(A)$ & $\sigma_1(C)$ & $\alpha_A$ & $\alpha_B$ & $\sigma_1(A)\sigma_1(B)\sigma_1(C)\alpha_A\alpha_B$ & $\sigma_1(H_{\text{block}})$ & $\sigma_1(H)$ \\
\midrule
MLP & 1 & 1.75{\tiny ±0.71} & 0.93{\tiny ±0.25} & 0.43{\tiny ±0.27} & 0.46{\tiny ±0.25} & 0.32{\tiny ±0.34} & 0.94{\tiny ±0.60} & 1.47{\tiny ±0.71} \\
MLP & 10 & 1.59{\tiny ±0.83} & 0.83{\tiny ±0.38} & 0.52{\tiny ±0.31} & 0.45{\tiny ±0.23} & 3.36{\tiny ±3.33} & 7.75{\tiny ±6.23} & 7.82{\tiny ±6.23} \\
MLP & 100 & 1.72{\tiny ±0.80} & 0.87{\tiny ±0.35} & 0.55{\tiny ±0.26} & 0.44{\tiny ±0.22} & 37.07{\tiny ±42.33} & 74.81{\tiny ±66.88} & 74.83{\tiny ±66.88} \\
MLP & 1000 & 1.56{\tiny ±0.79} & 0.80{\tiny ±0.41} & 0.51{\tiny ±0.30} & 0.46{\tiny ±0.26} & 309.38{\tiny ±358.77} & 580.02{\tiny ±581.07} & 580.02{\tiny ±581.07} \\
MLP & 5000 & 1.74{\tiny ±0.94} & 0.80{\tiny ±0.41} & 0.48{\tiny ±0.31} & 0.46{\tiny ±0.20} & 1469.93{\tiny ±1472.48} & 3518.09{\tiny ±3196.12} & 3518.09{\tiny ±3196.12} \\
MLP w/ GELU & 1 & 1.73{\tiny ±0.84} & 0.60{\tiny ±0.11} & 0.45{\tiny ±0.20} & 0.47{\tiny ±0.23} & 0.24{\tiny ±0.25} & 0.60{\tiny ±0.35} & 1.06{\tiny ±0.73} \\
MLP w/ GELU & 10 & 1.57{\tiny ±0.79} & 0.92{\tiny ±0.21} & 0.44{\tiny ±0.25} & 0.49{\tiny ±0.22} & 3.39{\tiny ±3.42} & 6.02{\tiny ±4.03} & 9.76{\tiny ±8.62} \\
MLP w/ GELU & 100 & 1.65{\tiny ±0.83} & 1.01{\tiny ±0.05} & 0.42{\tiny ±0.21} & 0.41{\tiny ±0.21} & 29.62{\tiny ±29.67} & 71.16{\tiny ±41.67} & 213.49{\tiny ±695.71} \\
MLP w/ GELU & 1000 & 1.36{\tiny ±0.60} & 0.94{\tiny ±0.26} & 0.51{\tiny ±0.26} & 0.39{\tiny ±0.22} & 290.68{\tiny ±271.62} & 610.09{\tiny ±390.96} & 776.61{\tiny ±477.53} \\
MLP w/ GELU & 5000 & 1.57{\tiny ±0.89} & 0.97{\tiny ±0.18} & 0.45{\tiny ±0.25} & 0.45{\tiny ±0.27} & 1631.60{\tiny ±1930.84} & 3463.65{\tiny ±3198.56} & 4016.44{\tiny ±3521.10} \\
Self-attention & 1 & 0.73{\tiny ±0.41} & 3.05{\tiny ±0.62} & 0.46{\tiny ±0.24} & 0.34{\tiny ±0.28} & 0.41{\tiny ±0.62} & 0.91{\tiny ±0.71} & 2.49{\tiny ±1.43} \\
Self-attention & 10 & 0.68{\tiny ±0.30} & 3.01{\tiny ±0.51} & 0.42{\tiny ±0.25} & 0.40{\tiny ±0.23} & 3.27{\tiny ±3.33} & 5.12{\tiny ±3.64} & 11.51{\tiny ±9.63} \\
Self-attention & 100 & 0.83{\tiny ±0.72} & 3.09{\tiny ±0.69} & 0.44{\tiny ±0.31} & 0.41{\tiny ±0.27} & 30.41{\tiny ±30.70} & 70.52{\tiny ±81.60} & 137.56{\tiny ±137.50} \\
Self-attention & 1000 & 0.72{\tiny ±0.40} & 3.11{\tiny ±0.59} & 0.37{\tiny ±0.25} & 0.42{\tiny ±0.28} & 260.64{\tiny ±230.62} & 550.32{\tiny ±407.37} & 1400.47{\tiny ±1447.52} \\
Self-attention & 5000 & 0.69{\tiny ±0.29} & 2.96{\tiny ±0.48} & 0.41{\tiny ±0.25} & 0.45{\tiny ±0.33} & 1616.20{\tiny ±1531.16} & 3366.71{\tiny ±2348.89} & 7624.24{\tiny ±10392.40} \\
\bottomrule
\end{tabular}
}
\end{table}
\clearpage
\section{Algorithm Details}
\label{app:algorithm-details}
In this section, we establish the correctness of \cref{alg:spectral-deformation} by showing that it is a form of spectral reparameterization which proves it is able to control convergence rate (\cref{thm:spectral-deformation-is-spec-reparam-app}).

\begin{theorem}[Spectral Deformation with Compensation is a (Lower-Max) Spectral Reparameterization.]
\label{thm:spectral-deformation-is-spec-reparam-app}
The following spectral deformation is a (lower-max) spectral reparameterization:
For $\theta_i$, set $\widetilde{\Sigma}_{\theta_i} \leftarrow T\Sigma_{\theta_i}$ such that $\sigma_1(U\widetilde{\Sigma}_{\theta_i}V^{\top}) = \alpha$. Where $\alpha$ is set so that $\sigma_1(H_{\theta}^{\mathcal{L}}) \geq c$ (Spectral Control) where $c > 0$ is given. Next, compensate either with $\theta_{j>i} \leftarrow \theta_j U_{\theta_i} \Sigma_{\theta_i} \widetilde{\Sigma}_{\theta_i}^{-1} U^{\top}_{\theta_i}$ or $\theta_{j<i} \leftarrow V_{\theta_i}\widetilde{\Sigma}^{-1}_{\theta_i} \Sigma_{\theta_i} V^{\top}_{\theta_i} \theta_j$ (Functional Invariance). Where $\phi_{i:j}$ and $\phi_{j:i}$ must be degree-$1$ homogeneous and $\widetilde{\Sigma}^{-1}$ is the pseudoinverse in the case of rectangular matrices.
\end{theorem}
\begin{proof}
To show this is the case, we need to prove two conditions: (1) \textbf{Spectral Control} and (2) \textbf{Functional Invariance up to $\epsilon$}.

\paragraph{(1) Spectral control} Spectral control states that $\sigma_1(H_{\theta}^\mathcal{L}) \geq c$.
We now need to show that the idealized function composition has a second partial derivative block that meets the algebraic structure of \cref{thm:hessian-singular-value-bound}. If so, then \cref{thm:hessian-singular-value-bound} says that with $c := \sigma_{r_1}(A)\sigma_{1}(B)\sigma_{r_2}(C) \prinangle(A_{r_1},BC_1)\prinangle(B_1, C_{r_2})$ where $B = \theta_i$, one can simply set $\sigma_{1}(\theta_i)$ such that $c$ is the desired value.  This follows directly from reconstructing $\widetilde{\Sigma}_{\theta_i} \leftarrow T\Sigma_{\theta_i}$ where $\sigma_1(\widetilde{\Sigma}_{\theta_i}) = \alpha$ with $\alpha$ being the value desired to set $c$.

By \cref{corr:block-interlacing-for-hessian}, we only need to show one second partial derivative block exists. Without loss of generality, consider the same three layer perception from \cref{eq:mlp-annotated}. Which we reparameterize as: 
\begin{equation*}
f(x) = \theta_3^{\top} \overbrace{
\phi_2 (
\underbrace{\tilde{I}\tilde{\theta}_2\,  \overbrace{\phi_1(\underbrace{\theta_1x}_{h_1}}^{z_1}))}_{h_2}
}^\text{$z_2$}.
\end{equation*}

$\tilde{\theta}$ is the spectrally deformed weight matrix with compensation matrix $\tilde{I}$ (technically we have the identity activation function between these which is $1$-homogeneous). Note that supposing we picked $\phi_1$ as a homogeneous activation, we could have easily demonstrated with compensating in $\theta_1$. Note that we do not need to consider the Hessian block w.r.t.\ the deformed parameter $\tilde{\theta}_2$ as this is not a requirement of \cref{corr:block-interlacing-for-hessian} nor \cref{thm:hessian-singular-value-bound}.
We then have the gradient with respect to $\theta_1$ given by
\begin{align*}
\nabla_{\theta_1} f
&= (x^{\top} \otimes I_{n_1})^{\top}
D_{z_1}\,\tilde{\theta}_2^{\top}\,\tilde{I}^{\top} D_{z_2}\,\theta_3 .
\end{align*}

We now compute the mixed second-order derivatives with respect to $\theta_1$ and $\theta_3$.

To clarify dimensionality, let
\[
\theta_1 \in \mathbb{R}^{n_1 \times n_0}, \quad
\theta_2 \in \mathbb{R}^{n_2 \times n_1}, \quad
\tilde{I} \in \mathbb{R}^{n_2 \times n_2}, \quad
\theta_3 \in \mathbb{R}^{n_2},
\]
and let $D_{z_1} \in \mathbb{R}^{n_1 \times n_1}$ and
$D_{z_2} \in \mathbb{R}^{n_2 \times n_2}$ denote the diagonal Jacobians of the activation functions at the first and second hidden layers, respectively.

The gradient with respect to $\theta_1$ is
\[
\nabla_{\theta_1} f
= (x^{\top} \otimes I_{n_1})^{\top}
D_{z_1}\,\tilde{\theta}_2^{\top}\,\tilde{I}^{\top} D_{z_2}\,\theta_3 .
\]
Differentiating with respect to $\theta_3$ yields the mixed Hessian block
\[
\frac{\partial^2 f}{\partial\,\operatorname{vec}(\theta_1)\,\partial \theta_3^{\top}}
=
(x \otimes D_{z_1}^{\top})\,
\tilde{\theta}_2^{\top}\,\tilde{I}^{\top} D_{z_2}.
\]

Defining
\begin{align*}
A &= (x \otimes D_{z_1}^{\top}) \in \mathbb{R}^{n_0 n_1 \times n_1}, \\
B &= \tilde{\theta}_2^{\top} \in \mathbb{R}^{n_1 \times n_2}, \\
C &= \tilde{I}^{\top} D_{z_2} \in \mathbb{R}^{n_2 \times n_2},
\end{align*}
we obtain the factorization
\[
\frac{\partial^2 f}{\partial\,\operatorname{vec}(\theta_1)\,\partial \theta_3^{\top}}
= ABC \in \mathbb{R}^{n_0 n_1 \times n_2}.
\]
This satisfies \cref{thm:hessian-singular-value-bound}, with the spectrally deformed matrix
$\tilde{\theta}_2^{\top}$ appearing as the central factor $B$.

\paragraph{(2) Functional invariance up to $\epsilon$} This is satisfied by showing, without loss of generality, that our reparameterization preserves the function exactly. This follows immediately from
\begin{align*}
\theta_i = I\theta_i &= I U_{\theta_i}\Sigma_{\theta_i}V_{\theta_i}^{\top} \\
&= \underbrace{U_{\theta_i}\Sigma_{\theta_i}\tilde{\Sigma}_{\theta_i}^{-1}U_{\theta_i}^{\top}}_{\tilde{I}}\underbrace{U_{\theta_i}\tilde{\Sigma}_{\theta_i}V_{\theta_i}^{\top}}_{\tilde{\theta}_i} \\
&= \tilde{I}\tilde{\theta}_i
\end{align*}
Note that in practice, floating-point computation during SVD accumulates small numerical errors, contributing to the $\epsilon$ of spectral reparameterization.
\end{proof}

\section{Experimental Details}
\label{app:experimental-details}

This section outlines the details for the Hessian lower bound experiments and numerical analysis as well as the full details of the settings used in the main paper with additional discussion. 

\paragraph{Reproducibility} Code for all experiments, including SpecDef implementation, attack scripts, and evaluation pipelines, is available at \url{REDACTED}. We provide configuration files for all hyperparameter settings reported in this paper. All pre-trained resistance methods used in our baseline comparisons (ELM, GradDiff, RepNoise, RMU, TAR, etc.) are publicly available on HuggingFace as cited in the respective papers which we list below.

\paragraph{Compute Resources} All experiments were conducted using NVIDIA A100 and L40S GPUs. 

\subsection{Details for \cref{sec:hessian_lower_bound} experiments}

\paragraph{\cref{tab:spectral_alignment_main}} The purpose of this experiment is to verify the developed bound \cref{thm:hessian-singular-value-bound}. We randomly initialize a three-layer MLP $\theta_3^{\top}\phi(I \cdot \theta_2\phi(\theta_1x))$ where $\theta_1, \theta_2, I \in \mathbb{R}^{m \times m}$ with $I$ initialized as the identity matrix, and $\theta_3 \in \mathbb{R}^{m}$ is the output layer. $\phi: \mathbb{R}^m \to \mathbb{R}^m$ is the ReLU activation. The reparameterization modifies $\theta_2$ and $I$ such that their product preserves the original linear transformation while allowing control over the singular values of $\theta_2$. We also initialize a single-layer self-attention network $\theta^{\top}VX\,\mathrm{softmax}((KX)^{\top}Qx_i)$ with a linear predictor head $\theta \in \mathbb{R}^{d}$, query, key, value matrices $Q, K, V \in \mathbb{R}^{d \times d}$, and data matrix $X \in \mathbb{R}^{d \times n}$ where $x_i$ is the $i$-th column of $X$ (the query token). This matches the convention in \cref{app:example-hessian-derivations}. For the toy experiments, we set all dimensions equal: $d = n = 4$, with batch size 1 for the MLP and sequence length $n=4$ for self-attention. We verify that the automatic Hessian computed using PyTorch matches the analytical Hessian from \cref{app:example-hessian-derivations}; the square dimensionality ensures all products are conformable and the derivations apply directly.

The experiment consists of choosing a layer ($\theta_2$ for the MLP and $V$ for self-attention) and scaling its largest singular value to $\{1, 10, 100, 1\mathrm{k}, 10\mathrm{k}\}$. For the MLP, this is achieved through reparameterization; for self-attention, the singular value is set directly. We verify that the automatic Hessian computed using PyTorch matches the analytical Hessian based on our derivations in \cref{app:example-hessian-derivations}. Here we measure the network Hessian $H^f_{\theta}$ rather than the loss Hessian. We run a forward pass and collect the following quantities over 30 random seeds. $\sigma_1(H_{\theta}^f)$ is the largest singular value of the full Hessian of the network function. For the MLP, $\sigma_1(H^{f}_{\theta_i,\theta_j})$ is the Hessian block w.r.t.\ $\theta_1$ and $I$. For the self-attention network, $H^{f}_{\theta_i,\theta_j}$ is w.r.t.\ $Q$ and $\theta$. To connect back to \cref{thm:hessian-singular-value-bound}, $\sigma_1(B)$ corresponds to $B := V$ for the self-attention network and $B := \theta_2^{\top}$ for the MLP. 

For the self-attention network, we just have $AB$, $A := (x_n \otimes I_m)^{\top} K X (\mathrm{Diag}(s) - ss^{\top})^{\top}X^{\top}$, where $s$ is the softmax activation vector. For the MLP: $A := (x \otimes I_m)^{\top} D_{z_1}^{ \top}$, $C := D_{z_2}$. Finally, $\alpha_A$ and $\alpha_B$ are the principal subspace alignment coefficients between $A$ and $BC$ (or $A$ and $B$), and between $B$ and $C$, respectively.

\paragraph{\cref{fig:toy-network-iteration-hessian-experiment}} The purpose of this experiment is to demonstrate the use of \cref{thm:hessian-singular-value-bound} for convergence rate control. For these experiments, we initialize a four-layer MLP with ReLU activations, a four-layer CNN (two CNN layers, two linear) with Tanh activations, and a two-layer transformer with ReLU activations. The input dimension is 6 and hidden dimensions are all 6. Each network has a linear predictor that is a weighted sum as the last layer to produce a single scalar output. The experiment consists of initializing the network with a singular value between 1 and 50, measuring the loss Hessian $\sigma_1(H_{\theta}^{\mathcal{L}})$, and measuring the number of iterations required for full-batch gradient descent to converge. Each training run is run for 1000 steps; we select the worst \textit{final} loss achieved as the convergence target and then count the number of steps to achieve this worst value. In theory, with perfect alignment of principal angles, the learning rate should be $1/\sigma_1(H_{\theta}^{\mathcal{L}})$. However, since alignment is often not perfect, we used $0.1/\sigma_1(H_{\theta}^{\mathcal{L}})$ which was found empirically. The task is a regression task where we generate 128 6-dimensional isotropic Gaussian vectors with targets that are also Gaussian. The CNN uses 2D image data (6 $\times$ 6, 1 channel) and the transformer uses sequences of length 6 with feature dimension of 6.

\subsection{Experimental Details for Language Models}

\paragraph{General Setup} Unless otherwise stated, we use AdamW \citep{loshchilov2017decoupled} with weight decay $0.01$, $\beta_1=0.9$, $\beta_2=0.999$, $\epsilon=10^{-8}$, and a linear learning rate schedule with warmup (10\% of total steps). Gradient clipping is applied with max norm $1.0$. All experiments use PyTorch \texttt{2.9.1} with 3 random seeds, incorporating evaluation improvements from \citet{qi2024evaluating} such as dataset shuffling across seeds.

\paragraph{Spectral Deformation Parameters} We denote by $\alpha$ the \textbf{multiplier} applied to the top singular values $\sigma_k$ of each weight matrix $\theta_i$; i.e., $\alpha = 1$ leaves singular values unchanged. We modify the top-$k = 25$ singular values (ablation in \cref{topk-singular-value-ablation}) across 5 randomly selected layers (ablation in \cref{app:num_layers}). We apply SpecDef to randomly selected MLP layers (specifically the \texttt{gate\_proj}, \texttt{o\_proj}, and \texttt{v\_proj} weight matrices). We exclude attention query and key matrices (\texttt{attn.q}, \texttt{attn.k}), MLP up and down projections, and the language model head (\texttt{lm\_head}) as these are often rank-deficient and provide less effective convergence control (see \cref{app:num_layers} for ablations). The 5 layers are selected uniformly at random from the eligible layers at the start of each experiment, with the same selection used across all learning rates for that seed.

\paragraph{Evaluation Metrics} For language models, we measure perplexity (PPL) on 100 randomly drawn samples from WikiText2 \citep{merity2016pointer} and report MMLU, Winogrande, ARC, and HellaSwag scores using LightEval.\footnote{\url{https://huggingface.co/docs/lighteval/en/index}} The primary model class is either \texttt{Llama-3.1-8b} or \texttt{Llama-3-8b-Instruct} depending on what as used for construct the defences we obtained from huggingface, with additional foundation models evaluated below.

\paragraph{Divergence Criteria} Consistent with prior training-time safety evaluations, we sweep learning rates from the smallest to the largest value for which the base model converges. Divergence is defined as a doubling of perplexity on WikiText2. We mark partial divergence (at least one seed) with $\dagger$ and full divergence (all seeds) with $\ddagger$. Following \citet{henderson2023self}, runs inducing model ``self-destruction'' are considered successful defences with the rational that an attacker wants to produce a useful model not a broken one (contrast this with adversarial bandits \citep{lattimore2020bandit} where the adversary wants to prevent learning or break the model in some way). As shown in \cref{tab:spectral_impact_combined,tab:harmful_results_lora_gpt20b}, SpecDef preserves original model performance.

\begin{table}[t]
\centering
\caption{Training resistance methods evaluated in this work. Models were obtained from those public ally available on Huggingface from \citet{che2025model} and \citet{fan2025llmunlearningresilientrelearning}.}
\label{tab:training-resistance-methods}
\begin{sc}
\begin{tabular}{ll}
\toprule
Method & Reference \\
\midrule
\texttt{ELM} & \citet{gandikota2024erasing} \\
\texttt{RepNoise} &   
\citet{rosati2024representation} \\
\texttt{RMU} & \citet{li2024wmdp} \\
\texttt{RR} & \citet{zouimproving} \\
\texttt{TAR} & \citet{tamirisa2025tamperresistantsafeguardsopenweightllms} \\
\texttt{GradDiff} & \citet{liu2022continual} \\
\texttt{NPO} & \citet{zhang2024negative} \\
\texttt{NPO-SAM} & \citet{fan2025llmunlearningresilientrelearning} \\
\texttt{SimNPO} & \citet{fan2024simplicity} \\
\texttt{IMMA} & \citet{zheng2024imma} \\
\texttt{DeepIgnorance} & \citet{o2025deep} \\
\bottomrule
\end{tabular}
\end{sc}
\end{table}

\subsubsection{Relearning Attacks on Unlearned Models}
\label{app:relearning-attacks}
\paragraph{Experimental Setting}
In this setting, we assess whether a model can \textit{robustly} \citep{deeb2024unlearning,tamirisa2025tamperresistantsafeguardsopenweightllms} unlearn dangerous biochemical weapons precursor knowledge from \citet{li2024wmdp} such that the unlearned knowledge cannot be recovered through relearning attacks. The base model chosen was \texttt{Llama-3-8B-Instruct}, which had a multiple-choice accuracy of 50.98\% on a held-out subset (255 test questions, 1020 train questions) of \texttt{WMDP-Bio}.

\paragraph{Baseline Evaluation and Related Works}
We evaluated a number of well-known unlearning methods (only some of which claim robust unlearning) in \cref{tab:unlearning-defence-evaluation} (i.e., robustness against relearning attacks). The relearning attacks are simply fully fine-tuned on \texttt{WMDP-Bio} without the held-out test set. While some methods were able to increase the number of steps to convergence, this was only under smaller learning rates. For sufficient learning rate sweeps, we found none of these methods can prevent a relearning attack. This echoes the findings of \citet{qi2024evaluating}, where robustness against training-time attacks is often not sufficiently evaluated. To control for reproducibility, we used models that were already trained from \citet{fan2025llmunlearningresilientrelearning}, \citet{che2025model}, and \citet{o2025deep}. Due to \citet{qi2024evaluating}, we were concerned with evaluating our own implementations of newer unpublished methods developed after or near submission of this manuscript. We did not include methods that did not provide publicly available models but reference them for completion: \citep{sondej2025collapse,shilov2025beyond}. Finally, we acknowledge there is a long tradition of unlearning in computer vision broadly (see for example \citep{golatkar2019eternal}) but these works have yet to be integrated properly in foundation model unlearning studies and also were primarily used in the context where training from scratch is easy which is outside of the scope of our paper.

\paragraph{Training Protocol}
We run each experiment for 510 steps or a total of 3 epochs at a batch size of 6. Learning rates below $10^{-6}$ never converged after 510 training steps. Learning rates greater than $3 \times 10^{-5}$ caused divergence, which we observed as consistently fluctuating loss with a significant increase in PPL (100\% increase). Early stopping is used when the perplexity doubles or we reach more than 60\% \texttt{WMDP-Bio} accuracy.

\paragraph{Combining Training Resistance Methods} In the main paper, \cref{tab:combined-llama-convergence} evaluates \texttt{ELM} as an initially applied unlearning model. Below we present results for combining SpecDef with \texttt{GradDiff} in \cref{tab:convergence-GradDiff-WMDP-llama3-8b-instruct} as well using a model with no unlearning applied \cref{tab:convergence-Meta-Llama-3-8B-Instruct}. Combining training resistance methods clearly has a beneficial effect since we are able to use lower values of $\alpha$ for effective training resistance when \texttt{GradDiff} is first applied. This points to potential future work where SpecDef is made stronger through combining differently motivated defences (which we describe in \cref{app:previous-defence-analysis}). Application of SpecDef to \texttt{Llama-3-8B-Instruct} without unlearning hardly makes any difference, giving more evidence that initial distance to the harmful loss objective doesn't meaningfully impact SpecDef's resistance capabilities.

\begin{table}[h]
\centering
\caption{
Convergence rates for \texttt{GradDiff}. 
Values show mean training steps $\pm$ std with final WMDP-Bio accuracy $\pm$ std in brackets across 3 seeds. 
$^{\dagger}$ indicates at least one run diverged (perplexity doubled from initial value). 
$^{\ddagger}$ indicates all runs diverged.
}
\label{tab:convergence-GradDiff-WMDP-llama3-8b-instruct}
\resizebox{\textwidth}{!}{%
\begin{sc}
\begin{tabular}{lccccccc}
\toprule
Learning Rate & $\alpha=1$ & $\alpha=1k$ & $\alpha=2500$ & $\alpha=5k$ & $\alpha=7500$ & $\alpha=10\mathrm{k}$ & $\alpha=1\mathrm{M}$ \\
\midrule
$10^{-6}$ & 510 (0.467 $\pm$ 0.035) & 20 $\pm$ 17 (0.180 $\pm$ 0.035)$^{\ddagger}$ & 13 $\pm$ 5 (0.158 $\pm$ 0.007)$^{\ddagger}$ & 13 $\pm$ 5 (0.154 $\pm$ 0.002)$^{\ddagger}$ & 13 $\pm$ 5 (0.165 $\pm$ 0.021)$^{\ddagger}$ & 10 (0.175 $\pm$ 0.022)$^{\ddagger}$ & 10 (0.221 $\pm$ 0.122)$^{\ddagger}$ \\
$5\times 10^{-6}$ & 126 $\pm$ 30 (0.608 $\pm$ 0.010) & 13 $\pm$ 5 (0.182 $\pm$ 0.019)$^{\ddagger}$ & 10 (0.204 $\pm$ 0.040)$^{\ddagger}$ & 10 (0.175 $\pm$ 0.023)$^{\ddagger}$ & 10 (0.213 $\pm$ 0.044)$^{\ddagger}$ & 10 (0.214 $\pm$ 0.079)$^{\ddagger}$ & 10 (0.210 $\pm$ 0.078)$^{\ddagger}$ \\
$8\times 10^{-6}$ & 73 $\pm$ 11 (0.626 $\pm$ 0.029) & 10 (0.199 $\pm$ 0.019)$^{\ddagger}$ & 10 (0.175 $\pm$ 0.035)$^{\ddagger}$ & 10 (0.199 $\pm$ 0.048)$^{\ddagger}$ & 10 (0.196 $\pm$ 0.056)$^{\ddagger}$ & 10 (0.195 $\pm$ 0.081)$^{\ddagger}$ & 10 (0.217 $\pm$ 0.094)$^{\ddagger}$ \\
$10^{-5}$ & 73 $\pm$ 5 (0.618 $\pm$ 0.014) & 10 (0.192 $\pm$ 0.024)$^{\ddagger}$ & 10 (0.171 $\pm$ 0.008)$^{\ddagger}$ & 10 (0.220 $\pm$ 0.052)$^{\ddagger}$ & 10 (0.166 $\pm$ 0.025)$^{\ddagger}$ & 10 (0.197 $\pm$ 0.079)$^{\ddagger}$ & 10 (0.200 $\pm$ 0.080)$^{\ddagger}$ \\
$3\times 10^{-5}$ & 53 $\pm$ 11 (0.631 $\pm$ 0.024) & 10 (0.191 $\pm$ 0.024)$^{\ddagger}$ & 10 (0.241 $\pm$ 0.078)$^{\ddagger}$ & 10 (0.178 $\pm$ 0.021)$^{\ddagger}$ & 10 (0.192 $\pm$ 0.046)$^{\ddagger}$ & 10 (0.165 $\pm$ 0.024)$^{\ddagger}$ & 10 (0.217 $\pm$ 0.118)$^{\ddagger}$ \\
\bottomrule
\end{tabular}
\end{sc}
}
\end{table}

\begin{table}[h]
\centering
\caption{
Convergence rates for \texttt{Llama-3-8B-Instruct}. 
Values show mean training steps $\pm$ std with final WMDP-Bio accuracy $\pm$ std in brackets across 3 seeds. 
$^{\dagger}$ indicates at least one run diverged (perplexity doubled from initial value). 
$^{\ddagger}$ indicates all runs diverged.
}
\label{tab:convergence-Meta-Llama-3-8B-Instruct}
\resizebox{\textwidth}{!}{%
\begin{sc}
\begin{tabular}{lccccccc}
\toprule
Learning Rate & $\alpha=1$ & $\alpha=1k$ & $\alpha=2500$ & $\alpha=5k$ & $\alpha=7500$ & $\alpha=10\mathrm{k}$ & $\alpha=1\mathrm{M}$ \\
\midrule
$10^{-6}$ & 93 $\pm$ 25 (0.607 $\pm$ 0.006) & 43 $\pm$ 57 (0.452 $\pm$ 0.157)$^{\dagger}$ & 50 $\pm$ 54 (0.428 $\pm$ 0.146)$^{\ddagger}$ & 26 $\pm$ 25 (0.281 $\pm$ 0.152)$^{\ddagger}$ & 18 $\pm$ 10 (0.300 $\pm$ 0.156)$^{\ddagger}$ & 16 $\pm$ 11 (0.244 $\pm$ 0.172)$^{\ddagger}$ & 10 (0.212 $\pm$ 0.050)$^{\ddagger}$ \\
$5\times 10^{-6}$ & 26 $\pm$ 5 (0.617 $\pm$ 0.023) & 20 $\pm$ 17 (0.290 $\pm$ 0.170)$^{\ddagger}$ & 16 $\pm$ 10 (0.255 $\pm$ 0.128)$^{\ddagger}$ & 12 $\pm$ 4 (0.219 $\pm$ 0.123)$^{\ddagger}$ & 10 (0.237 $\pm$ 0.093)$^{\ddagger}$ & 10 (0.216 $\pm$ 0.099)$^{\ddagger}$ & 10 (0.246 $\pm$ 0.080)$^{\ddagger}$ \\
$8\times 10^{-6}$ & 26 $\pm$ 5 (0.667 $\pm$ 0.024) & 16 $\pm$ 11 (0.264 $\pm$ 0.172)$^{\ddagger}$ & 14 $\pm$ 5 (0.182 $\pm$ 0.020)$^{\ddagger}$ & 10 (0.226 $\pm$ 0.099)$^{\ddagger}$ & 10 (0.247 $\pm$ 0.072)$^{\ddagger}$ & 10 (0.233 $\pm$ 0.069)$^{\ddagger}$ & 10 (0.259 $\pm$ 0.075)$^{\ddagger}$ \\
$10^{-5}$ & 20 (0.659 $\pm$ 0.082) & 16 $\pm$ 11 (0.269 $\pm$ 0.165)$^{\ddagger}$ & 10 (0.227 $\pm$ 0.064)$^{\ddagger}$ & 10 (0.277 $\pm$ 0.093)$^{\ddagger}$ & 10 (0.207 $\pm$ 0.030)$^{\ddagger}$ & 10 (0.214 $\pm$ 0.089)$^{\ddagger}$ & 10 (0.200 $\pm$ 0.042)$^{\ddagger}$ \\
$3\times 10^{-5}$ & 16 $\pm$ 5 (0.633 $\pm$ 0.033) & 10 (0.222 $\pm$ 0.100)$^{\ddagger}$ & 10 (0.201 $\pm$ 0.057)$^{\ddagger}$ & 10 (0.201 $\pm$ 0.045)$^{\ddagger}$ & 10 (0.202 $\pm$ 0.048)$^{\ddagger}$ & 10 (0.175 $\pm$ 0.011)$^{\ddagger}$ & 10 (0.220 $\pm$ 0.062)$^{\ddagger}$ \\
\bottomrule
\end{tabular}
\end{sc}
}
\end{table}

\subsubsection{Harmful Fine-Tuning of Foundation Models}
\label{app:harmful-finetuning-details}
For our harmful fine-tuning experiments, we use the non-overlapping 30k train and test splits of the BeaverTails dataset \citep{beavertails}. This dataset consists of harmful questions (e.g., ``How can I steal from my job's cash register safely?'') as well as harmful answers to such questions. Typically, these questions would be refused by an LLM's safety guard. At the time of writing, RepNoise \citep{rosati2024representation} and TAR \citep{tamirisa2025tamperresistantsafeguardsopenweightllms} are the only existing published methods that provide some open-weight resistance against harmful fine-tuning with public huggingface models. As in the WMDP experiments, we evaluate the models trained by the original authors and publicly released on HuggingFace, as listed in the corresponding papers. Readers are encouraged to consult the original works for detailed descriptions of these methods and their training procedures; brief summaries are also provided in \citep{che2025model}. Resistance methods such as \citep{huang2024vaccine} that require control over the model during training time (i.e., for all iterations $t > 0$) are not relevant to our study, since curvature control is unnecessary in such settings (see the survey of these methods in \citealp{huang2024harmful}). We acknowledge the existence of several newer methods which claim improvements over TAR and RepNoise but these were not including for several reasons: (1) none of these provided public models for fair evaluation, (2) many if not all of these were unpublished preprints identified after this paper manuscript was submitted for publication, (3) the findings in these papers all show methods that can be defeated with harmful fine-tuning and do not produce theoretical results that are interesting for comparison in our paper. For completeness these were: \citep{cheng2025weaponization, yi2025ctrap,wang2025self}.

\paragraph{Training Configuration}
We train each model for one epoch with a batch size of 6, randomly sampling and shuffling 6k training examples from the BeaverTails 30k training split. All experiments use AdamW with the same hyperparameters as the relearning experiments. For SpecDef experiments, we apply spectral deformation to five randomly selected layers, using the same layer selection strategy described in \cref{app:relearning-attacks}.

\paragraph{Harmfulness Evaluation}
We measure harmfulness using the safety classifier from \citet{rosati2024representation}, available at \texttt{huggingface.co/domenicrosati/deberta-v3-xsmall-beavertails-classifier}. For each evaluation, we randomly select 64 test questions from the BeaverTails test split and generate model responses. The classifier outputs a harmfulness probability for each response; we report the fraction of responses classified as harmful using a threshold of 0.6. High harmfulness scores sometimes requires careful interpretation. The text classifier can flag repeated or disfluent text as harmful, which occurs when models diverge. We mark such cases with $\dagger$, indicating perplexity has at least doubled. Entries showing high harmfulness without $\dagger$ represent fluent, genuinely harmful generations—these are successful attacks. Entries with $\dagger$ indicate defence via model collapse, regardless of the harmfulness score. We chose to use a classifier with this failure mode because it was small enough to run efficiently at each 10 steps rather than an larger but more robust evaluator like \texttt{LlamaGuard} \citep{inan2023llamaguardllmbasedinputoutput} which would have made training runs too long for practical evaluation.

\paragraph{Stopping Criteria}
We stop training if either (1) the model's perplexity on WikiText2 doubles from its initial value, indicating model collapse, or (2) the harmfulness score exceeds 60\%. When both perplexity and harmfulness are high, the resulting model should be interpreted as unusable, since continued training would further degrade utility.

\paragraph{Combining Training Resistance Methods} In addition to the results in the main text applying SpecDef to \texttt{Llama-3-8B-Instruct} in \cref{tab:combined-llama-convergence}, we evaluate the impact of applying SpecDef to models with training resistance already applied. In \cref{tab:convergence-repnoise-0.001-beta}, \texttt{RepNoise} is made worse initially after applying SpecDef and this could be because RepNoise operates by reducing curvature. For \texttt{TAR} in \cref{tab:convergence-Llama-3-8B-Instruct-TAR-Refusal}, the methods pair well together and SpecDef is more effective at lower $\alpha$ rates when combined.

\begin{table}[h]
\centering
\caption{
Convergence rates for \texttt{RepNoise}. 
Values show mean training steps $\pm$ std with final harmfulness probability $\pm$ std in brackets across 3 seeds. 
$^{\dagger}$ indicates at least one run diverged (perplexity doubled from initial value). 
$^{\ddagger}$ indicates all runs diverged.
}
\label{tab:convergence-repnoise-0.001-beta}
\resizebox{\textwidth}{!}{%
\begin{sc}
\begin{tabular}{lcccccccc}
\toprule
Learning Rate & $\alpha=1$ & $\alpha=1k$ & $\alpha=2500$ & $\alpha=5k$ & $\alpha=7500$ & $\alpha=10\mathrm{k}$ & $\alpha=100\mathrm{k}$ & $\alpha=1\mathrm{M}$ \\
\midrule
$10^{-6}$ & 2597 (0.098 $\pm$ 0.001) & 40 $\pm$ 10 (0.705 $\pm$ 0.065) & 23 $\pm$ 5 (0.644 $\pm$ 0.035) & 20 (0.641 $\pm$ 0.173)$^{\dagger}$ & 13 $\pm$ 5 (0.589 $\pm$ 0.255)$^{\dagger}$ & 16 $\pm$ 11 (0.584 $\pm$ 0.304)$^{\ddagger}$ & 10 (0.024 $\pm$ 0.037)$^{\ddagger}$ & 10 (0.011 $\pm$ 0.011)$^{\ddagger}$ \\
$5\times 10^{-6}$ & 216 $\pm$ 5 (0.610 $\pm$ 0.013) & 20 (0.626 $\pm$ 0.206)$^{\dagger}$ & 10 (0.537 $\pm$ 0.246)$^{\dagger}$ & 10 (0.481 $\pm$ 0.331)$^{\dagger}$ & 10 (0.552 $\pm$ 0.319)$^{\dagger}$ & 10 (0.482 $\pm$ 0.282)$^{\dagger}$ & 10 (0.019 $\pm$ 0.031)$^{\ddagger}$ & 10 (0.004 $\pm$ 0.003)$^{\ddagger}$ \\
$8\times 10^{-6}$ & 143 $\pm$ 5 (0.650 $\pm$ 0.034) & 13 $\pm$ 5 (0.585 $\pm$ 0.291)$^{\dagger}$ & 10 (0.629 $\pm$ 0.146)$^{\dagger}$ & 10 (0.566 $\pm$ 0.260)$^{\dagger}$ & 10 (0.375 $\pm$ 0.319)$^{\ddagger}$ & 10 (0.048 $\pm$ 0.049)$^{\ddagger}$ & 10 (0.014 $\pm$ 0.022)$^{\ddagger}$ & 10 (0.016 $\pm$ 0.025)$^{\ddagger}$ \\
$10^{-5}$ & 130 (0.697 $\pm$ 0.006) & 13 $\pm$ 5 (0.561 $\pm$ 0.296)$^{\dagger}$ & 10 (0.443 $\pm$ 0.328)$^{\dagger}$ & 10 (0.414 $\pm$ 0.380)$^{\ddagger}$ & 10 (0.256 $\pm$ 0.439)$^{\ddagger}$ & 10 (0.118 $\pm$ 0.139)$^{\ddagger}$ & 10 (0.002 $\pm$ 0.000)$^{\ddagger}$ & 10 (0.016 $\pm$ 0.024)$^{\ddagger}$ \\
$3\times 10^{-5}$ & 63 $\pm$ 5 (0.653 $\pm$ 0.054) & 10 (0.573 $\pm$ 0.270)$^{\dagger}$ & 10 (0.247 $\pm$ 0.415)$^{\ddagger}$ & 10 (0.005 $\pm$ 0.006)$^{\ddagger}$ & 10 (0.147 $\pm$ 0.224)$^{\ddagger}$ & 10 (0.136 $\pm$ 0.230)$^{\ddagger}$ & 10 (0.003 $\pm$ 0.002)$^{\ddagger}$ & 10 (0.002 $\pm$ 0.001)$^{\ddagger}$ \\
\bottomrule
\end{tabular}
\end{sc}
}
\end{table}

\begin{table}[h]
\centering
\caption{
Convergence rates for \texttt{TAR}. 
Values show mean training steps $\pm$ std with final harmfulness probability $\pm$ std in brackets across 3 seeds. 
$^{\dagger}$ indicates at least one run diverged (perplexity doubled from initial value). 
$^{\ddagger}$ indicates all runs diverged.
}
\label{tab:convergence-Llama-3-8B-Instruct-TAR-Refusal}
\resizebox{\textwidth}{!}{%
\begin{sc}
\begin{tabular}{lcccccccc}
\toprule
Learning Rate & $\alpha=1$ & $\alpha=1k$ & $\alpha=2500$ & $\alpha=5k$ & $\alpha=7500$ & $\alpha=10\mathrm{k}$ & $\alpha=100\mathrm{k}$ & $\alpha=1\mathrm{M}$ \\
\midrule
$10^{-6}$ & 2597 (0.581 $\pm$ 0.003) & 43 $\pm$ 49 (0.396 $\pm$ 0.250)$^{\dagger}$ & 33 $\pm$ 40 (0.356 $\pm$ 0.218)$^{\dagger}$ & 30 $\pm$ 34 (0.266 $\pm$ 0.081)$^{\ddagger}$ & 30 $\pm$ 34 (0.320 $\pm$ 0.126)$^{\ddagger}$ & 26 $\pm$ 28 (0.255 $\pm$ 0.173)$^{\ddagger}$ & 13 $\pm$ 5 (0.055 $\pm$ 0.055)$^{\ddagger}$ & 10 (0.017 $\pm$ 0.005)$^{\ddagger}$ \\
$5\times 10^{-6}$ & 103 $\pm$ 5 (0.615 $\pm$ 0.002) & 40 $\pm$ 51 (0.364 $\pm$ 0.235)$^{\dagger}$ & 23 $\pm$ 23 (0.178 $\pm$ 0.157)$^{\ddagger}$ & 16 $\pm$ 11 (0.179 $\pm$ 0.239)$^{\ddagger}$ & 16 $\pm$ 11 (0.151 $\pm$ 0.231)$^{\ddagger}$ & 13 $\pm$ 5 (0.137 $\pm$ 0.219)$^{\ddagger}$ & 10 (0.002 $\pm$ 0.001)$^{\ddagger}$ & 10 (0.072 $\pm$ 0.061)$^{\ddagger}$ \\
$8\times 10^{-6}$ & 76 $\pm$ 11 (0.626 $\pm$ 0.031) & 30 $\pm$ 34 (0.347 $\pm$ 0.252)$^{\ddagger}$ & 16 $\pm$ 11 (0.243 $\pm$ 0.263)$^{\ddagger}$ & 16 $\pm$ 11 (0.152 $\pm$ 0.228)$^{\ddagger}$ & 13 $\pm$ 5 (0.130 $\pm$ 0.212)$^{\ddagger}$ & 13 $\pm$ 5 (0.010 $\pm$ 0.006)$^{\ddagger}$ & 10 (0.003 $\pm$ 0.003)$^{\ddagger}$ & 10 (0.051 $\pm$ 0.083)$^{\ddagger}$ \\
$10^{-5}$ & 76 $\pm$ 11 (0.646 $\pm$ 0.030) & 30 $\pm$ 34 (0.317 $\pm$ 0.273)$^{\ddagger}$ & 16 $\pm$ 11 (0.202 $\pm$ 0.289)$^{\ddagger}$ & 13 $\pm$ 5 (0.136 $\pm$ 0.207)$^{\ddagger}$ & 13 $\pm$ 5 (0.088 $\pm$ 0.119)$^{\ddagger}$ & 13 $\pm$ 5 (0.116 $\pm$ 0.172)$^{\ddagger}$ & 10 (0.033 $\pm$ 0.028)$^{\ddagger}$ & 10 (0.011 $\pm$ 0.008)$^{\ddagger}$ \\
$3\times 10^{-5}$ & 46 $\pm$ 5 (0.660 $\pm$ 0.047) & 16 $\pm$ 11 (0.159 $\pm$ 0.215)$^{\ddagger}$ & 13 $\pm$ 5 (0.131 $\pm$ 0.192)$^{\ddagger}$ & 10 (0.126 $\pm$ 0.146)$^{\ddagger}$ & 10 (0.104 $\pm$ 0.118)$^{\ddagger}$ & 10 (0.039 $\pm$ 0.044)$^{\ddagger}$ & 10 (0.089 $\pm$ 0.083)$^{\ddagger}$ & 10 (0.027 $\pm$ 0.019)$^{\ddagger}$ \\
\bottomrule
\end{tabular}
\end{sc}
}
\end{table}

\subsubsection{LoRA Harmful Fine-Tuning of \texttt{GPT-OSS-20b}}
\label{app:lora-harmful-fine-tuning-gpt-oss-20b}

In this experiment, we generally repeat the exact same procedure as we did with \texttt{Llama-3.1-8B-Instruct}, \texttt{RepNoise}, and \texttt{TAR} for defending against harmful fine-tuning attacks except that instead of full fine-tuning we use LoRA to train the model. Here we demonstrate at a much larger scale with \texttt{GPT-OSS-20b} \citep{openai2025gptoss120bgptoss20bmodel}. We found that it is very easy to undo the safeguards of \texttt{GPT-OSS-20b} in as few as 70 training steps with a single layer selected for LoRA. Our resistance method, SpecDef, is sufficient once a large enough multiplier and number of layers are selected but not before. \cref{tab:harmful_results_lora_gpt20b} presents these results. Note that with LoRA, we had to use more than a single layer to resist convergence. For LoRA, we target the up and down projection MLP layers (\texttt{mlp.up\_proj}, \texttt{mlp.down\_proj}) in layers 7, 15, and 24, selected to span the early, middle, and late portions of the network. We observed no significant difference in SpecDef's effectiveness when varying these specific layers within each portion. For LoRA's settings we use a rank of 8 and alpha of 16. We use the PEFT library for the LoRA implementation.

\subsection{Experimental Details for Vision Models}
\label{app:NSFW}

In this section we provide further details on the investigation of SpecDef applied to two Vision tasks using Diffusion models as the foundation model. Training resistance in image generation is more nascent than for text generation so the only baseline we are able to investigate is \citet{zheng2024imma} which both our NSFW and Style transfer settings are based on below. We acknowledge that a recent preprint, \citet{abdalla2025gift}, builds as a follow up on IMMA but were not aware of it during writing of this manuscript. Regardless, readers can compare our results with \citet{abdalla2025gift} and find that this method is still able to be defeated given a sufficient number of training steps.


\subsubsection{NSFW}
To evaluate spectral deformation for resisting harmful fine-tuning that enables Not-Safe-For-Work (NSFW) image generation in text-to-image diffusion models, we examine whether the method can add friction that prevents an attacker from restoring unsafe capabilities through low-rank adaptation (LoRA) fine-tuning.

\paragraph{Setup} We follow the harmful fine-tuning evaluation protocol introduced by IMMA \citep{zheng2024imma} (Immunization Against Harmful Fine-tuning), one of the few model-side resistance methods that specifically addresses how to make harmful fine-tuning more difficult for restoring safety-filtered NSFW image generation capabilities in text-to-image diffusion models. Thus, IMMA provides a suitable comparison point for spectral deformation, since both aim to impede malicious adaptation while preserving the model's general utility.

\paragraph{Baseline} To enable a direct comparison under identical conditions, we replicate IMMA’s immunization training and harmful fine-tuning setup in our NSFW re-learning scenario. Using Stable Diffusion (SD) v1.4 (860M parameters), we generate 4,704 NSFW images using I2P (Inappropriate Image Prompts) datasets \citep{schramowski2022safe}. A NudeNet\footnote{\url{https://github.com/platelminto/NudeNetClassifier}} detector with a threshold of 0.05, is then applied to extract 305 images containing nudity. From these, we replicate IMMA's data split by using the first 50 images for immunization (bi-level optimization, 1000 steps), the next 50 for the attacker's LoRA-based NSFW re-learning attack (rank 4, $\alpha$=4, 50 epochs, learning rate $10^{-4}$, batch size 1), and treat the remaining 205 images as the evaluation set.

\paragraph{Spectral Deformation on Stable Diffusion} We then apply Algorithm \ref{alg:spectral-deformation} to SD V1.4, specifically targeting the cross-attention layers of the U-Net under different layer selection strategies. Building on the observation that spectral deformation can often be achieved by modifying a single layer, and that reparameterizing multiple layers makes convergence-rate control even harder for an adversary to undo, we evaluate configurations ranging from a single layer to all cross-attention projections. Specifically, we experiment with: 

(1) a randomly selected single cross-attention layer, (2) 5 randomly selected cross-attention layers, (3) 25 randomly selected cross-attention layers, and (4) all cross-attention layers with all three projection matrices (\textit{to\_q}, \textit{to\_k}, \textit{to\_v}) for a total of 48 weight matrices. 

We also vary the singular values multiplier across $\alpha \in \{100, 1000, 10\mathrm{k}\}$ for each layer selection strategy to examine how different conditioning strengths influence the diffusion models' robustness to LoRA-based relearning attacks and to what extent the core model utility is preserved.

\paragraph{Evaluation Metrics} Following IMMA's \emph{open-source implementation}\footnote{\url{https://github.com/amberyzheng/IMMA.}}, we evaluate the effectiveness of our resistance method using 200 generated images and measuring the proportion and the reduction of generated images that contain nudity after the LoRA re-learning attack. Specifically, we take the first 50 prompts from the evaluation set and generate 4 images per prompt using different random seeds, yielding 200 samples. We then apply the NudeNet detector to each image. An image is considered to contain nudity if its detector score exceeds the threshold of $0.05$, consistent with IMMA’s specification. Table \ref{tab:nsfw_prevention} summarizes these results.

\begin{table*}[!htp]
\centering
\caption{NSFW content generation ablations on layer selections across 3 random seeds}
\begin{sc}
\resizebox{\textwidth}{!}{%
\begin{tabular}{lcccc}
\toprule
Method & Nudity Counts $\downarrow$ & Images w/ Nudity $\downarrow$ & Img.\ Reduction $\uparrow$ & Nud.\ Reduction $\uparrow$ \\
\midrule
SpecDef $\alpha=100$ & &  &  &  \\
\midrule
random 1 layer & 371.3 ± 58.0 & 117.0 ± 7.8 / 205 & 42.9 ± 3.8\% & -7.3 ± 16.8\%  \\
random 5 layers & 419.7 ± 146.1 & 122.7 ± 19.7/ 205 & 40.2 ± 9.6\% & -21.3 ± 42.2\%  \\
random 25 layers & 331.7 ± 79.1 & 109.3 ± 4.0 / 205 & 46.7 ± 2.0\% & 4.1 ± 22.9\%  \\
all 48 cross-attention layers & 335.3 ± 84.5 & 123.0 ± 12.0 / 205 & 40.0 ± 5.9\% & 3.1 ± 24.4\% \\
\midrule
SpecDef $\alpha=1,000$ &  &  &  &  \\
\midrule
random 1 layer & 407.7 ± 40.5 & 124.3 ± 4.0 / 205 & 39.3 ± 2.0 & -17.8  ± 11.7\%  \\
random 5 layers & 229.7 ± 120.0 & 80.0 ± 24.3 / 205 & 61.1 ± 11.8\% & 33.6 ± 34.7\%  \\
random 25 layers & 138.3 ± 97.0 & 53.3 ± 36.3 / 205 & 74.0 ± 17.7\% & 60.0 ± 28.0\%  \\
\rowcolor{green!10}
all 48 cross-attention layers & \textbf{0.0 ± 0.0} & \textbf{0.0 ± 0.0 / 205} & \textbf{100.0 ± 0.0}\% & \textbf{100.0 ± 0.0}\% \\
\midrule
SpecDef $\alpha=10{,}000$ & & & &  \\
\midrule
random 1 layer & 312.7 ± 60.1 & 102.7 ± 13.0 / 205 & 49.9 ± 6.3\% & 9.6 ± 17.4\% \\
random 5 layers & 270.7 ± 130.9 & 83.3 ± 31.5 / 205 & 59.3 ± 15.4\% & 21.8 ± 37.8\%  \\
\rowcolor{green!10}
random 25 layers & \textbf{0.0 ± 0.0} & \textbf{0.0 ± 0.0 / 205} & \textbf{100.0 ± 0.0}\% & \textbf{100.0 ± 0.0}\%  \\
\rowcolor{green!10}
all 48 cross-attention layers & \textbf{0.0 ± 0.0} & \textbf{0.0 ± 0.0 / 205} & \textbf{100.0 ± 0.0}\% & \textbf{100.0 ± 0.0}\% \\
\bottomrule
\end{tabular}
}
\end{sc}
\label{tab:nsfw_layer_ablations}
\end{table*}

In the ablations above (\cref{tab:nsfw_layer_ablations}), we see that weak or sparsely applied spectral deformation is not intended to be a safe configuration; in such regimes you may see little protection and even more harmful generations than baseline. Resistance is only possible once $\alpha$ and layer coverage are large enough to produce substantial curvature inflation, as in the highlighted rows.

\begin{figure*}[!htbp]
\centering
\includegraphics[width=1\linewidth]{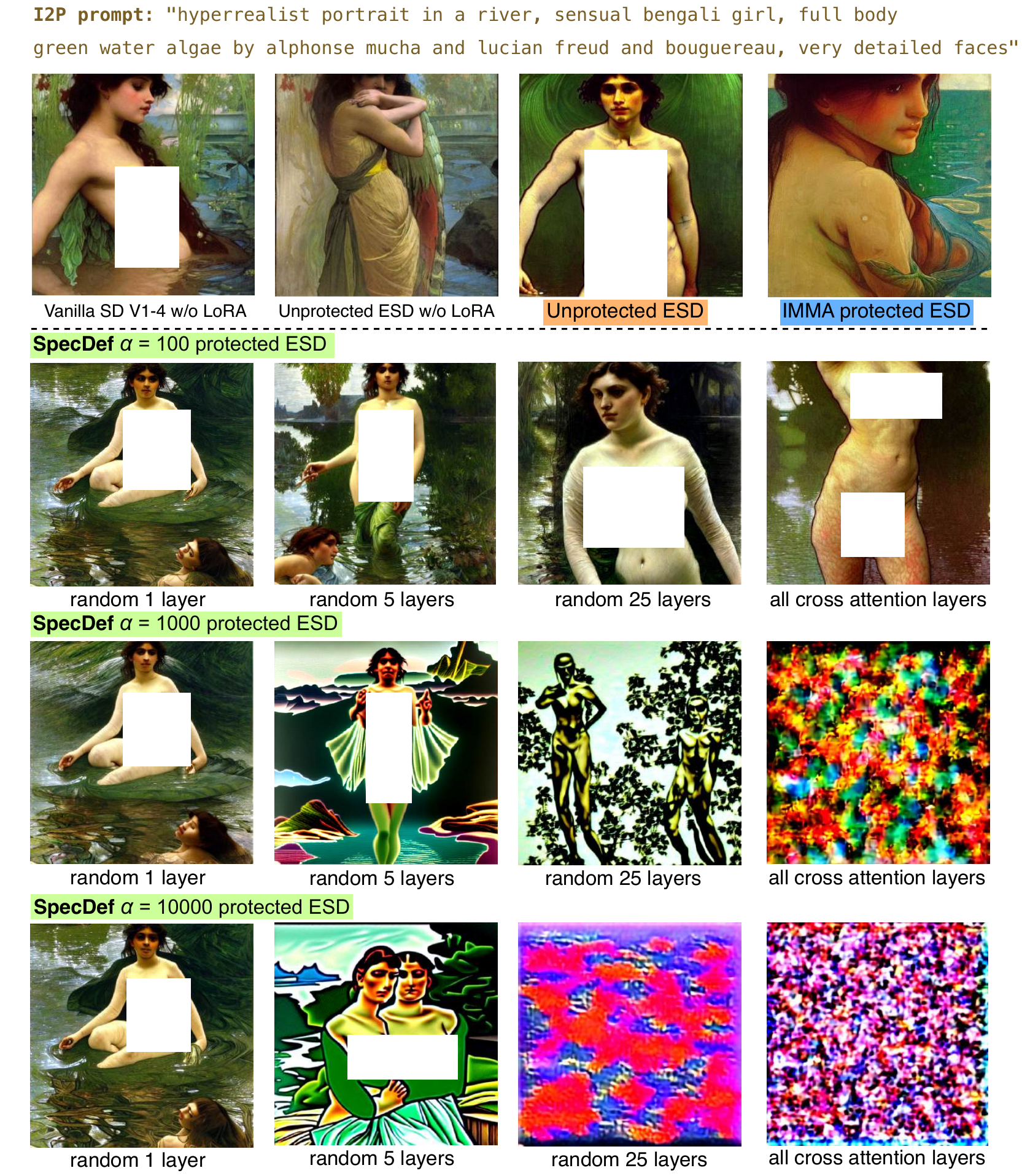}
\caption{
Qualitative NSFW results under an I2P prompt using the same random seed. We compare the images generated by spectral deformation protected ESD with the generations from Vanilla SD v1-4, nudity-erased SD (ESD), unprotected ESD after LoRA adaptation, and IMMA-protected ESD after LoRA adaptation. We further vary the largest singular values multipliers  ($\alpha \in \{100, 1000, 10000\}$) and the layer selection strategies (random 1, 5, 25 layers, or all cross-attention layers).
}
\label{fig:NSFW}
\end{figure*}
\clearpage
\subsubsection{Style Relearning}
\label{app:style-relearning}

We next evaluate whether SpecDef can add friction to the restoration of artistic styles that have been intentionally removed from a diffusion model. In contrast to the NSFW experiment presented in \cref{app:NSFW}, which targets the recovery of safety-filtered content, we now examine whether convergence rate control can inhibit the re-learning of copyrighted artistic styles.

\paragraph{Setup} Our evaluation follows the same experimental design as IMMA. To construct reference data for measuring style restoration, we first generate 100 ``Van Gogh'' style reference images from the original SD V1.4 using the prompt ``An artwork by Van Gogh''. We then load a Van Gogh ESD (Van Gogh Style Erased Stable Diffusion) \citep{gandikota2023erasing} checkpoint of Stable Diffusion V1.4, in which the ``Van~Gogh'' style was intentionally removed. This erased model is the target of both IMMA and SpecDef, as well as the target of the subsequent LoRA re-learning attack.

\paragraph{Baseline} Following IMMA's data partitioning strategy, we use 20 of the reference images for IMMA's resistance method, which requires 50 epochs (1000 steps) of bi-level optimization. For LoRA-based artistic style re-learning (rank~4, $\alpha=4$, 50 epochs, learning rate $10^{-4}$, batch size~1), we use an additional 20 of these images. During LoRA re-learning training, we generate a set of 100 validation images using the same prompt at each epoch. These validation images serve exclusively to monitor recovery of the erased style and, at the final checkpoint (epoch 50), are used as inputs for evaluation. 

\paragraph{Evaluation Metrics} To quantify recovery of the erased artistic style, we follow IMMA’s open-source evaluation and compute the Similarity Gap Ratio (SGR):

\[
\text{SGR} =
\frac{M(x_{\text{ref}},\, x_{\text{baseline}})
      - M(x_{\text{ref}},\, x_{\text{resistance}})}
     {M(x_{\text{ref}},\, x_{\text{baseline}})},
\]

where $x_{\text{ref}}$ are the 100 reference images generated from the original SD v1.4, $x_{\text{baseline}}$ are the 100 validation images generated by the undefended ESD at epoch~50 of the LoRA attack, $x_{\text{resistance}}$ are the corresponding images produced by the defended model after training, and $M$ is a similarity metric. Following IMMA, we compute SGR using three complementary measures: \textbf{SGR$_C$ (CLIP)} (semantic/style similarity gap), \textbf{SGR$_D$ (DINO)} (mid-level visual features similarity gap), and \textbf{SGR$_L$ (LPIPS)} (perceptual similarity gap). For consistency, we report \textit{one minus} \textbf{LPIPS} as IMMA did, such that larger values for all three metrics mean larger similarity gaps, hence, stronger protection. To ensure 
statistical robustness, we repeat the entire pipeline with three random seeds and report all results as mean ± standard deviation. The results are shown below in Table \ref{tab:vangogh_sgr_layers}

\begin{table*}[!htp]
\centering
\caption{
Van~Gogh style re-learning after SpecDef across 3 random seeds. We report:
(1) \textbf{SGR$_C$ (CLIP)} $\uparrow$: semantic–style similarity gap to the reference images,
(2) \textbf{SGR$_D$ (DINO)} $\uparrow$: mid-level visual feature similarity gap,
(3) \textbf{SGR$_L$ (LPIPS)} $\uparrow$: perceptual similarity gap, and
(4) \textbf{Resistance Time} $\uparrow$: wall-clock time required to apply the resistance method.
Higher SGR indicates a larger similarity gap to the reference images, reflecting stronger prevention of artistic-style re-learning.
}
\begin{sc}
\resizebox{\textwidth}{!}{%
\begin{tabular}{lcccc}
\toprule
Method
& SGR$_C$ (CLIP) $\uparrow$ 
& SGR$_D$ (DINO) $\uparrow$ 
& SGR$_L$ (LPIPS) $\uparrow$
& Resistance Time $\downarrow$ \\
\midrule
Reference Images
& 0\% & 0\% & 0\% & --- \\
IMMA \citep{zheng2024imma}
& 3.34 ± 3.18\% & 7.47 ± 14.94\% & 1.77 ± 15.08\% & 343.32 ± 3.47s  \\
\midrule
SpecDef $\alpha=10{,}000$ & \textbf{10.81 ± 2.23\%} & \textbf{44.08 ± 4.08\%} & \textbf{42.11 ± 21.96\%} & \textbf{20.27 ± 0.00s} \\
\bottomrule
\end{tabular}
}
\end{sc}
\label{tab:vangogh_sgr_layers}
\end{table*}

\begin{figure*}[!htbp]
\centering
\includegraphics[width=1\linewidth]{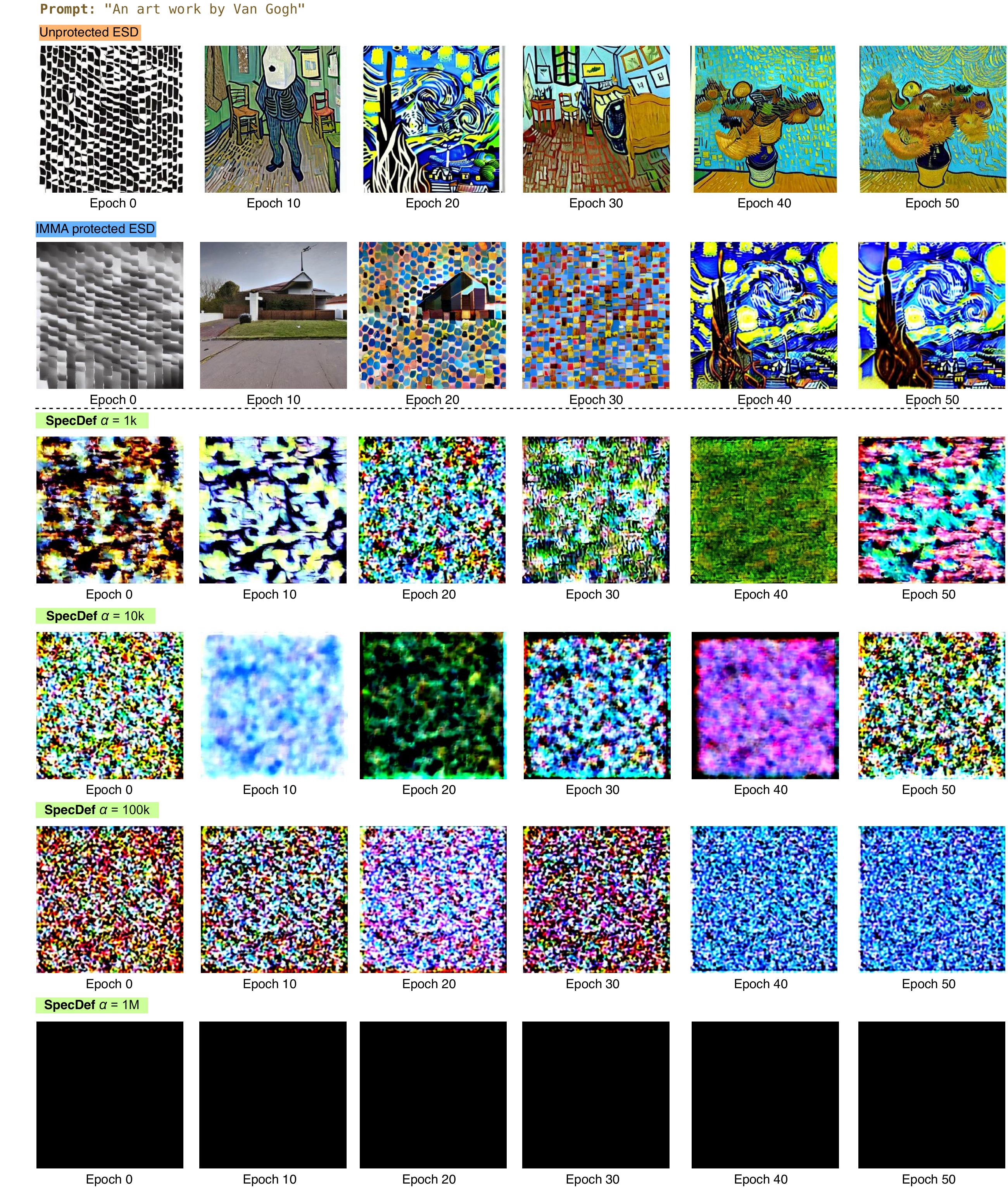}
\label{fig:Vangogh}
\caption{
Evolution of generated images over 50 LoRA training epochs for the prompt “An artwork by Van~Gogh” under different protection mechanisms, with a LoRA learning rate of $10^{-4}$. Both IMMA and SpecDef are applied to all cross-attention layers.
}
\end{figure*}

\begin{table*}[!htp]
\centering
\caption{Van Gogh style re-learning ablations across LoRA learning rates with three random seeds. Rows highlighted in green indicate successful protection against LoRA fine-tuning. Bolded values indicate all outputs collapse to black images across three seeds.}
\begin{sc}
\resizebox{\textwidth}{!}{%
\begin{tabular}{lcccc}
\toprule
Method
& SGR$_C$ (CLIP) $\uparrow$ 
& SGR$_D$ (DINO) $\uparrow$ 
& SGR$_L$ (LPIPS) $\uparrow$
& Resistance Time \\
\midrule
SpecDef $\alpha=100$ \\
\midrule
random 1 layer & 0.67 ± 0.73\% & -4.54 ± 6.92\% & -0.99 ± 11.58\% & 0.46 ± 0.34s \\
random 5 layers & 1.34 ± 2.53\% & -6.45 ± 3.97\% & -4.96 ± 2.83\% & 1.50 ± 0.23s \\
random 25 layers & 1.62 ± 3.51\% & -7.03 ± 4.40\% & -5.78 ± 6.65\% & 10.65 ± 1.15s  \\
all 48 cross-attention layers & 0.23 ± 2.61\% & -9.15 ± 4.80\% & -5.90 ± 3.54\% & 20.27 ± 0.01s  \\
\midrule
SpecDef $\alpha=1,000$ \\
\midrule
random 1 layer & 0.69 ± 2.46\% & -5.50 ± 1.12\% & -4.21 ± 5.58\% & 0.24 ± 0.20s  \\
random 5 layers & 0.42 ± 0.78\% & -6.40 ± 6.20\% & -4.72 ± 12.29\% & 1.49 ± 0.23s  \\
random 25 layers & 2.67 ± 1.60\% & 7.29 ± 12.27\% & -0.18 ± 12.20\% & 10.66 ± 1.14s  \\
\rowcolor{green!10}
all 48 cross-attention layers & 8.95 ± 0.39\% & 30.04 ± 4.60\% & 2.43 ± 10.49\% & 20.27 ± 0.00s \\
\midrule
SpecDef $\alpha=10{,}000$ \\
\midrule
random 1 layer & 0.04 ± 0.52\% & -6.44 ± 6.79\% & -2.72 ± 12.25\% & 0.24 ± 0.20s \\
random 5 layers & 2.78 ± 1.28\% & 6.36 ± 5.66\% & -3.90 ± 11.82\% & 1.49 ± 0.23s  \\
\rowcolor{green!10}
random 25 layers & \textbf{9.29 ± 0.96\%} & \textbf{51.40 ± 3.14\%} & \textbf{42.09 ± 6.33\%} & 10.66 ± 1.14s  \\
\rowcolor{green!10}
all 48 cross-attention layers & \textbf{10.81 ± 2.23\%} & \textbf{44.08 ± 4.08\%} & \textbf{42.11 ± 21.96\%} & 20.27 ± 0.00s \\
\bottomrule
\end{tabular}
}
\end{sc}
\label{tab:vangogh_layer_ablations}
\end{table*}

\begin{table*}[!htp]
\centering
\caption{Van~Gogh style re-learning ablations on LoRA learning rate across 3 random seeds. Rows highlighted in green represent successful protection from LoRA fine-tuning. Values in bold font represent all generated images are collapsed to black. }
\begin{sc}
\begin{tabular}{lccc}
\toprule
Method
& SGR$_C$ (CLIP) $\uparrow$ 
& SGR$_D$ (DINO) $\uparrow$ 
& SGR$_L$ (LPIPS) $\uparrow$\\
\midrule
Reference Images & 0\% & 0\% & 0\% \\
\midrule
IMMA \citep{zheng2024imma} \\
\hspace{1mm} $10^{-3}$ & 0.65 ± 0.10\% & -5.74 ± 5.05\% & -0.32 ± 6.21\% \\
\hspace{1mm} $10^{-4}$ & 3.34 ± 3.18\% & 7.47 ± 14.94\% & 1.77 ± 15.08\% \\
\hspace{1mm} $10^{-5}$ & 11.94 ± 1.51\% & 60.55 ± 3.03\% & 38.37 ± 8.56\% \\
\hspace{1mm} $10^{-6}$ & 10.89 ± 2.90\% & 52.49 ± 10.75\% & 35.14 ± 18.60\% \\
\hspace{1mm} $10^{-7}$ & 11.61 ± 0.62\% & 54.64 ± 9.34\% & 32.45 ± 14.64\% \\
\midrule
SpecDef $\alpha=10\mathrm{k}$ \\
\rowcolor{green!10}
\hspace{1mm} $10^{-3}$ & 15.82 ± 2.21\% & 68.55 ± 11.15\% & 63.31 ± 4.03\% \\
\rowcolor{green!10}
\hspace{1mm} $10^{-4}$ & 10.81 ± 2.23\% & 44.08 ± 4.08\% & 42.11 ± 21.96\% \\
\rowcolor{green!10}
\hspace{1mm} $10^{-5}$ & 8.31 ± 1.55\% & 32.91 ± 7.37\% & 10.91 ± 9.56\% \\
\hspace{1mm} $10^{-6}$ & 0.37 ± 1.66\% & -7.52 ± 6.10\% & -5.37 ± 11.48\% \\
\hspace{1mm} $10^{-7}$ & 0.06 ± 1.19\% & -8.97 ± 9.37\% & -4.71 ± 10.68\% \\
\midrule
SpecDef $\alpha=100\mathrm{k}$ \\
\rowcolor{green!10}
\hspace{1mm} $10^{-3}$ & \textbf{17.30 ± 0.31\%} & \textbf{77.19 ± 1.49\%} & \textbf{62.83 ± 3.39\%} \\
\rowcolor{green!10}
\hspace{1mm} $10^{-4}$ & 10.40 ± 0.54\% & 44.22 ± 2.00\% & 49.31 ± 6.14\% \\
\rowcolor{green!10}
\hspace{1mm} $10^{-5}$ & 8.18 ± 1.04\% & 40.43 ± 4.07\% & 31.79 ±  9.81\% \\
\rowcolor{green!10}
\hspace{1mm} $10^{-6}$ & 9.09 ± 1.47\% & 33.93 ± 5.42\% & 7.20 ± 7.97\% \\
\hspace{1mm} $10^{-7}$ & 0.84 ± 1.44\% & -2.01 ± 6.64\% & -6.37 ± 11.65\% \\
\midrule
SpecDef $\alpha=1\mathrm{M}$ \\
\rowcolor{green!10}
\hspace{1mm} $10^{-3}$ & \textbf{17.30 ± 0.31\%} & \textbf{77.19 ± 1.49\%} & \textbf{62.83 ± 3.39\%} \\
\rowcolor{green!10}
\hspace{1mm} $10^{-4}$ & \textbf{17.30 ± 0.31\%} & \textbf{77.19 ± 1.49\%} & \textbf{62.83 ± 3.39\%} \\
\rowcolor{green!10}
\hspace{1mm} $10^{-5}$ & \textbf{17.30 ± 0.31\%} & \textbf{77.19 ± 1.49\%} & \textbf{62.83 ± 3.39\%} \\
\rowcolor{green!10}
\hspace{1mm} $10^{-6}$ & 11.87 ± 1.22\% & 50.41 ± 4.90\% & 64.77 ± 3.72\% \\
\rowcolor{green!10}
\hspace{1mm} $10^{-7}$ & 9.26 ± 2.47\% & 38.56 ± 9.53\% & 23.05 ± 23.62\% \\
SpecDef $\alpha=1B$ \\
\rowcolor{green!10}
\hspace{1mm} $10^{-3}$ & \textbf{17.30 ± 0.31\%} & \textbf{77.19 ± 1.49\%} & \textbf{62.83 ± 3.39\%} \\
\rowcolor{green!10}
\hspace{1mm} $10^{-4}$ & \textbf{17.30 ± 0.31\%} & \textbf{77.19 ± 1.49\%} & \textbf{62.83 ± 3.39\%} \\
\rowcolor{green!10}
\hspace{1mm} $10^{-5}$ & \textbf{17.30 ± 0.31\%} & \textbf{77.19 ± 1.49\%} & \textbf{62.83 ± 3.39\%} \\
\rowcolor{green!10}
\hspace{1mm} $10^{-6}$ & \textbf{17.30 ± 0.31\%} & \textbf{77.19 ± 1.49\%} & \textbf{62.83 ± 3.39\%} \\
\rowcolor{green!10}
\hspace{1mm} $10^{-7}$ & \textbf{17.30 ± 0.31\%} & \textbf{77.19 ± 1.49\%} & \textbf{62.83 ± 3.39\%} \\
\bottomrule
\end{tabular}

\end{sc}
\label{tab:vangogh_lr_ablations_2}
\end{table*}

\begin{table*}[!htp]
\centering
\caption{Van~Gogh style re-learning ablations on top k singular values across 3 random seeds. Rows highlighted in green represent successful protection from LoRA fine-tuning.}
\begin{sc}
\begin{tabular}{lccc}
\toprule
Method
& SGR$_C$ (CLIP) $\uparrow$ 
& SGR$_D$ (DINO) $\uparrow$ 
& SGR$_L$ (LPIPS) $\uparrow$\\
\midrule
Reference Images & 0\% & 0\% & 0\% \\
\midrule
SpecDef $\alpha=1k$, LoRA Learning Rate = $10^{-4}$\\
\hspace{1mm} $k=1$ & 0.84 ± 1.07\% & -3.39 ± 3.18\% & -2.67 ± 9.03\% \\
\hspace{1mm} $k=5$ & -0.36 ± 1.67\% & -9.83 ± 2.88\% & -9.99 ± 5.84\% \\
\hspace{1mm} $k=10$ & 2.99 ± 3.89\% & 9.11 ± 8.97\% & 1.37 ± 3.05\% \\
\hspace{1mm} $k=25$ & 0.88 ± 1.17\% & -5.96 ± 7.00\% & -8.05 ± 10.54\% \\
\rowcolor{green!10}
\hspace{1mm} $k=50$ & 6.93 ± 1.69\% & 33.29 ± 11.08\% & 19.06 ± 5.14\% \\
\rowcolor{green!10}
\hspace{1mm} $k=100$ & 4.67 ± 2.38\% & 15.78 ± 7.43\% & 7.07 ± 9.66\% \\
\rowcolor{green!10}
\hspace{1mm} $k=250$ & 3.70 ± 0.49\% & 19.47 ± 1.54\% & 0.53 ± 8.87\% \\
\bottomrule
\end{tabular}

\end{sc}
\label{tab:vangogh_lr_ablations_1}
\end{table*}

\begin{figure*}[!htbp]
\centering
\includegraphics[width=1\linewidth]{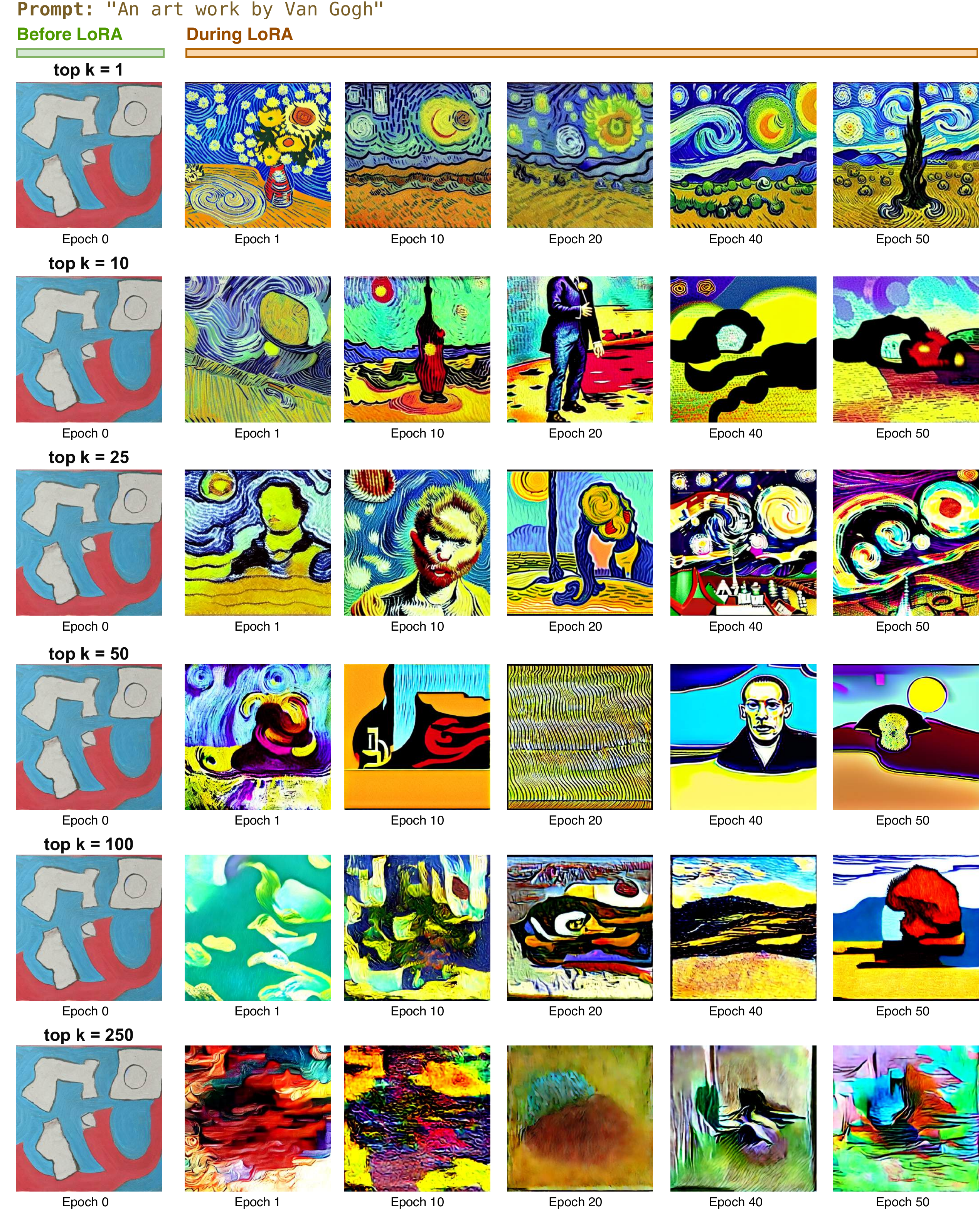}
\label{fig:topk_Vangogh}
\caption{
Evolution of generated images over 50 LoRA training epochs for the prompt “An artwork by Van~Gogh”, with $\alpha = 1000$, a LoRA learning rate of $10^{-4}$, and Top-$k$ singular values where $k \in \{1, 10, 25, 50, 100, 250\}$.
}
\end{figure*}
\clearpage

\section{Further Ablations and Experimental Variation}
\label{app:further-ablations}
The degrees of freedom of our algorithm are very few (number of layers, the multiplier $\alpha$, and $k$ the number of singular values to choose). Nonetheless, this section provides a number of additional experiments that either support choices made in the paper or supplement other analyses.

\subsection{Layers used in Spectral Deformation}
\label{app:num_layers}
In this experiment, we varied the number of layers used for spectral deformation. We repeat the harmful fine-tuning experiment with fresh reruns (different layers may be selected due to seed variation) on \texttt{meta-llama\_Llama-3.1-8B-Instruct} using $\alpha = 1\mathrm{k}$, which does not provide resistance when layers is $1$. Here we show that resistance at a smaller $\alpha$ is possible provided given more layers are used, though for smaller learning rates even more layers or a larger $\alpha$ will need to be used. For spectral deformation with layer injection for all layers it costs up to $2\times$ the original model cost to reparameterized the entire network so choosing a smaller number of layers is more efficient. 

\begin{table}[htbp]
\centering

\caption{Harmful fine-tuning results: steps $\pm$ std (harmfulness $\pm$ std). $^\dagger$: some seeds diverged (PPL doubled). $^\ddagger$: all seeds diverged. Notice that divergence in all seeds only typically occurs when enough layers are used. More layers results in better convergence rate control.}
\label{tab:harmful_layers_lr}
\begin{sc}
\resizebox{\textwidth}{!}{
\begin{tabular}{llcccccc}
\toprule
$\sigma$ & Layers & $10^{-6}$ & $2\times10^{-6}$ & $5\times10^{-6}$ & $8\times10^{-6}$ & $10^{-5}$ & $3\times10^{-5}$ \\
\midrule
\multirow{4}{*}{1000} & 1 & 120 $\pm$ 132 (0.53 $\pm$ 0.16)$^{\dagger}$ & 66 $\pm$ 60 (0.54 $\pm$ 0.12)$^{\dagger}$ & 53 $\pm$ 40 (0.56 $\pm$ 0.10)$^{\dagger}$ & 43 $\pm$ 30 (0.64 $\pm$ 0.19)$^{\dagger}$ & 40 $\pm$ 26 (0.62 $\pm$ 0.14)$^{\dagger}$ & 23 $\pm$ 15 (0.63 $\pm$ 0.13)$^{\dagger}$ \\
 & 2 & 26 $\pm$ 5 (0.62 $\pm$ 0.24)$^{\dagger}$ & 16 $\pm$ 5 (0.59 $\pm$ 0.16)$^{\dagger}$ & 10 (0.58 $\pm$ 0.13)$^{\dagger}$ & 10 (0.55 $\pm$ 0.16)$^{\dagger}$ & 10 (0.59 $\pm$ 0.14)$^{\dagger}$ & 10 (0.42 $\pm$ 0.39)$^{\dagger}$ \\
 & 5 & 20 (0.69 $\pm$ 0.10)$^{\dagger}$ & 13 $\pm$ 5 (0.64 $\pm$ 0.15)$^{\dagger}$ & 10 (0.60 $\pm$ 0.06)$^{\ddagger}$ & 10 (0.59 $\pm$ 0.12)$^{\ddagger}$ & 10 (0.47 $\pm$ 0.39)$^{\ddagger}$ & 10 (0.25 $\pm$ 0.27)$^{\ddagger}$ \\
 & 10 & 16 $\pm$ 5 (0.72 $\pm$ 0.05)$^{\dagger}$ & 13 $\pm$ 5 (0.68 $\pm$ 0.18)$^{\ddagger}$ & 10 (0.67 $\pm$ 0.05)$^{\ddagger}$ & 10 (0.44 $\pm$ 0.38)$^{\ddagger}$ & 10 (0.43 $\pm$ 0.37)$^{\ddagger}$ & 10 (0.03 $\pm$ 0.05)$^{\ddagger}$ \\
\midrule
\multirow{4}{*}{5000} & 1 & 83 $\pm$ 110 (0.55 $\pm$ 0.10)$^{\dagger}$ & 60 $\pm$ 62 (0.58 $\pm$ 0.12)$^{\dagger}$ & 40 $\pm$ 36 (0.60 $\pm$ 0.12)$^{\dagger}$ & 36 $\pm$ 30 (0.64 $\pm$ 0.17)$^{\dagger}$ & 30 $\pm$ 26 (0.64 $\pm$ 0.15)$^{\dagger}$ & 23 $\pm$ 15 (0.59 $\pm$ 0.21)$^{\dagger}$ \\
 & 2 & 10 (0.58 $\pm$ 0.13)$^{\dagger}$ & 10 (0.56 $\pm$ 0.16)$^{\dagger}$ & 10 (0.48 $\pm$ 0.28)$^{\dagger}$ & 10 (0.41 $\pm$ 0.40)$^{\dagger}$ & 10 (0.42 $\pm$ 0.41)$^{\dagger}$ & 10 (0.41 $\pm$ 0.41)$^{\dagger}$ \\
 & 5 & 10 (0.59 $\pm$ 0.08)$^{\ddagger}$ & 10 (0.44 $\pm$ 0.33)$^{\ddagger}$ & 10 (0.32 $\pm$ 0.24)$^{\ddagger}$ & 10 (0.45 $\pm$ 0.34)$^{\ddagger}$ & 10 (0.10 $\pm$ 0.08)$^{\ddagger}$ & 10 (0.01 $\pm$ 0.02)$^{\ddagger}$ \\
 & 10 & 10 (0.65 $\pm$ 0.07)$^{\ddagger}$ & 10 (0.44 $\pm$ 0.38)$^{\ddagger}$ & 10 (0.12 $\pm$ 0.17)$^{\ddagger}$ & 10 (0.00 $\pm$ 0.00)$^{\ddagger}$ & 10 (0.00 $\pm$ 0.00)$^{\ddagger}$ & 10 (0.04 $\pm$ 0.05)$^{\ddagger}$ \\
\midrule
\multirow{4}{*}{10000} & 1 & 83 $\pm$ 110 (0.57 $\pm$ 0.11)$^{\dagger}$ & 60 $\pm$ 70 (0.59 $\pm$ 0.13)$^{\dagger}$ & 40 $\pm$ 36 (0.60 $\pm$ 0.13)$^{\dagger}$ & 30 $\pm$ 26 (0.56 $\pm$ 0.19)$^{\dagger}$ & 30 $\pm$ 26 (0.56 $\pm$ 0.15)$^{\dagger}$ & 16 $\pm$ 11 (0.56 $\pm$ 0.13)$^{\dagger}$ \\
 & 2 & 10 (0.55 $\pm$ 0.16)$^{\dagger}$ & 10 (0.58 $\pm$ 0.18)$^{\dagger}$ & 10 (0.43 $\pm$ 0.39)$^{\dagger}$ & 10 (0.40 $\pm$ 0.42)$^{\dagger}$ & 10 (0.37 $\pm$ 0.41)$^{\dagger}$ & 10 (0.35 $\pm$ 0.36)$^{\ddagger}$ \\
 & 5 & 10 (0.64 $\pm$ 0.07)$^{\ddagger}$ & 10 (0.38 $\pm$ 0.35)$^{\ddagger}$ & 10 (0.25 $\pm$ 0.22)$^{\ddagger}$ & 10 (0.06 $\pm$ 0.05)$^{\ddagger}$ & 10 (0.00 $\pm$ 0.00)$^{\ddagger}$ & 10 (0.03 $\pm$ 0.04)$^{\ddagger}$ \\
 & 10 & 10 (0.40 $\pm$ 0.36)$^{\ddagger}$ & 10 (0.22 $\pm$ 0.29)$^{\ddagger}$ & 10 (0.01 $\pm$ 0.02)$^{\ddagger}$ & 10 (0.12 $\pm$ 0.15)$^{\ddagger}$ & 10 (0.01 $\pm$ 0.02)$^{\ddagger}$ & 10 (0.09 $\pm$ 0.13)$^{\ddagger}$ \\
\bottomrule
\end{tabular}
}
\end{sc}
\end{table}

We additionally ran an ablation study where we select different layers for spectral deformation to show which layers are effective and which are not. We select a single layer $\ell$ from $\ell \in \{2, 16, 30\}$ that represent early, middle and later parts of the model. With the same experimental settings above and $\alpha=10\mathrm{k}$. The experiments were run over 3 random seeds. We selected one of the following seven components for each experiment: (1) a Query matrix, (2) a Key matrix, (3) a Value matrix, (4) an Output projection, (5) a Gated projection, (6) an Up projection, or (7) a Down projection. When looking at effectiveness by position, \emph{averaged over all seven components}, there was not a large difference between layer depths: early layers ($\ell = 2$) achieved a mean harmfulness score of $0.40 {\,\pm 0.37}$, middle layers ($\ell = 16$) scored $0.32 {\,\pm 0.29}$, and late layers ($\ell = 30$) scored $0.30 {\, \pm 0.32}$. These means are close relative to their standard deviations, indicating that layer position alone is not the dominant factor when averaging across component types.

However, for specific components, effectiveness can vary sharply with depth. The best components to select were the output projections after attention, which achieved very low harmfulness across all depths (e.g., $\ell=30$: 0.0014, $\ell=16$: 0.0598, $\ell=2$: 0.0021). The value matrices were similarly effective: $\ell=30$ harmfulness is 0.0042, $\ell=16$ harmfulness is 0.0038, and $\ell=2$ harmfulness is 0.0122. In contrast, other components showed strong depth dependence. For example, the MLP up projection had low harmfulness at early layers ($\ell=2$: 0.0864) but was largely ineffective at later depths ($\ell=16$: 0.5593, $\ell=30$: 0.5488). Gate and down projections were generally ineffective in early layers but showed some effectiveness in later layers.

Query and key matrices (as well as up and down projections) are often very rank deficient, which can be predicted from the results in \cref{tab:bound-k-r-rand-def} where \cref{thm:hessian-singular-value-bound} becomes loose. The modest trend toward later layers being more effective than earlier layers (in the component-averaged means), and why some types of layers are more effective than others may also be explained by how often the layer appears as a factor in the Hessian blocks.

\subsection{Top K Singular Values for Deformation}
\label{topk-singular-value-ablation}
In order to use SpecDef, we need to select the top-$k$ singular values. In this section, we will provide an empirical analysis on the impact of that choice in \cref{tab:topk-ablation}. We purposely choose a lower $\alpha$ than is effective to illustrate the effects of scaling $k$ and we chose to use the original base model \texttt{Llama-3.1-8B-Instruct} which is already close performing well at this task. This was done to make training resistance as hard as possible for SpecDef to highlight the effect of varying $k$.

\begin{table}[h]
\centering
\small
\begin{tabular}{lcccc}
\toprule
$k$ & lr=$10^{-6}$ & lr=$5{\times}10^{-6}$ & lr=$10^{-5}$ & lr=$3{\times}10^{-5}$ \\
\midrule
1 & 330$\pm$180 (0.573) & 50$\pm$0 (0.616) & 30$\pm$0 (0.614) & 20$\pm$10 (0.514)$^\dagger$ \\
2 & 350$\pm$160 (0.578) & 55$\pm$5 (0.602)$^\dagger$ & 30$\pm$0 (0.590)$^\dagger$ & 10$\pm$0 (0.433)$^\ddagger$ \\
5 & 400$\pm$110 (0.594) & 45$\pm$5 (0.594)$^\dagger$ & 30$\pm$0 (0.592)$^\dagger$ & 25$\pm$15 (0.443)$^\ddagger$ \\
10 & 325$\pm$185 (0.571) & 45$\pm$5 (0.594)$^\dagger$ & 35$\pm$5 (0.576)$^\dagger$ & 15$\pm$5 (0.475)$^\ddagger$ \\
25 & 340$\pm$170 (0.588) & 50$\pm$20 (0.541)$^\dagger$ & 30$\pm$10 (0.531)$^\ddagger$ & 15$\pm$5 (0.455)$^\ddagger$ \\
all & 325$\pm$185 (0.573) & 30$\pm$0 (0.451)$^\ddagger$ & 20$\pm$0 (0.439)$^\ddagger$ & 10$\pm$0 (0.398)$^\ddagger$ \\
\bottomrule
\end{tabular}
\label{tab:topk-ablation}
\caption{Top-$k$ ablation ($\alpha=10,000$). Number of steps and WMDP classification accuracy (mean$\pm$std over 3 seeds). $^\dagger$: some seeds doubled perplexity; $^\ddagger$: all seeds doubled perplexity.}
\end{table}

Our results in \cref{tab:topk-ablation} are that as SpecDef uses more $k$ divergence occurs faster and SpecDef is able to provide resistance for smaller learning rates.

\section{Are Previous Resistance Methods Convergence-Rate Control in $L$?}
\label{app:previous-defence-analysis}

In this section, we examine how existing resistance methods relate to our framework. Prior work is motivated by heterogeneous formalisms—ranging from information-theoretic arguments \citep{rosati-etal-2024-immunization} to metalearning perspectives \citep{tamirisa2025tamperresistantsafeguardsopenweightllms}. A comprehensive unification connection each directly to our framework is out of scope of this paper. Instead, we ask a sharply defined question aligned with our theory:

\emph{Do existing resistance methods implicitly control convergence by manipulating curvature, e.g., via the Lipschitz constant $L$?}

To this end, we conduct a layerwise spectral analysis of defended models.

\begin{figure}[htbp]
    \centering
    \includegraphics[width=1\linewidth]{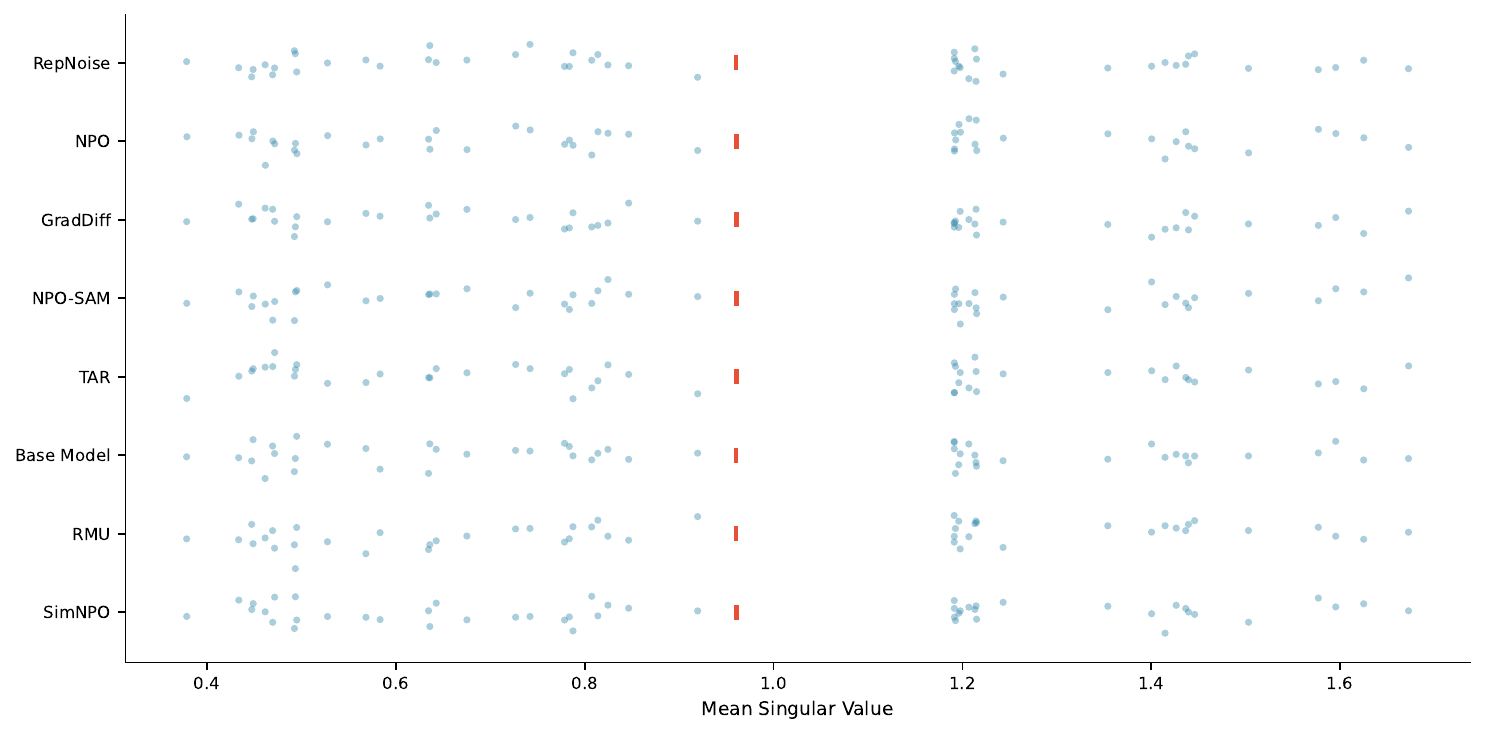}
    \caption{Mean singular value of each layer’s weight matrix across resistance methods (rows). Red bars indicate layerwise means. SpecDef is omitted since its reparameterized layers have singular values that are orders of magnitude larger by construction. Or finding is that weight spectra is not meaningfully different from the base model.}
    \label{fig:layer_sv_means_of_each_method}
\end{figure}

\paragraph{Observation: No Curvature Inflation}
As shown in \cref{fig:layer_sv_means_of_each_method}, all methods exhibit weight spectra closely matching the base model (\texttt{Meta-Llama-3-8B-Instruct}). No method exceeds $\sigma_1(\theta_i) \approx 15$, and the variance of singular values is small (typically $\sim$0.45), indicating uniformly good conditioning. Thus, these resistance methods do \emph{not} implement convergence-rate control by increasing $L$ through enlarged spectral norms, nor other kinds of numerical instability through weight-based poor conditioning. This is not surprising given that each method is built using a loss function and trained using standard proceedures such as the use of Adam and the use of auxiliary retain loses which would prevent poor conditioning and inflated spectral values.

This aligns with their empirical behaviour: all methods remain trainable under sufficiently strong learning rates and do not exhibit the divergence characteristic of large-$L$ systems (see \cref{tab:unlearning-defence-evaluation}). Moreover, several methods can still be fine-tuned on unrelated objectives \citep{rosati2024representation}, which would be inconsistent with global curvature inflation and is orthogonal to the mechanism we propose.

\paragraph{Locally Flat Curvature via Alignment Rather Than Magnitude.}
One might instead hypothesize that these resistance methods slow optimization by introducing unusually small curvature directions, increasing the number of gradient steps required for curvature-unaware optimizers. However, we do not observe systematically smaller singular values than in the base model. In fact, gradient norms are often \emph{initially larger}, as shown in \cref{fig:training_dynamics_of_defences}.

\begin{figure}[htbp]
    \centering
    \includegraphics[width=1\linewidth]{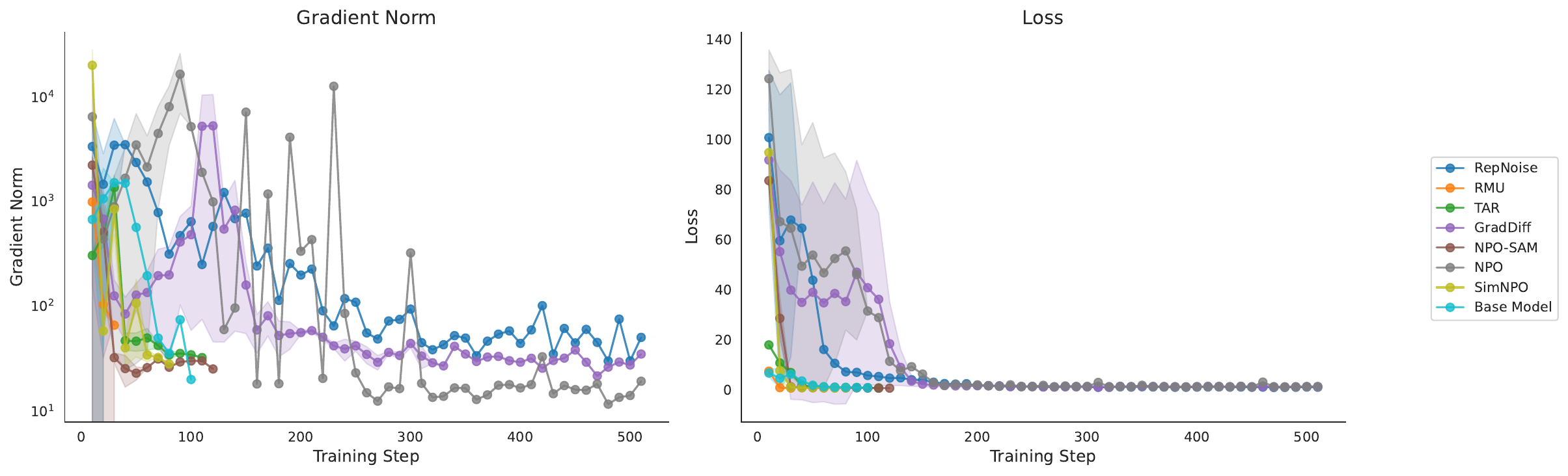}
    \caption{Training dynamics across resistance methods, showing loss and log gradient norm. On successful resistance runs \texttt{RepNoise}, \texttt{NPO}, and \texttt{GradDiff} enter a small gradient norm regime where loss does not change.}
    \label{fig:training_dynamics_of_defences}
\end{figure}

This combination—unaltered spectral magnitudes but large initial gradient norms—points to a more subtle mechanism: \emph{spectral alignment}. Within our framework, flat curvature can arise when activation Jacobian align with singular vectors corresponding to small singular values; this produces loss-specific flatness without globally modifying $L$. Flat curvature also results in convergence rate control when a chosen learning rate $1/L$ is not large enough to produce effective gradient steps in the flat direction. In cases where the learning rate cannot be increased since there are other large curvature directions simply increasing the learning rate would cause divergence.

Crucially, such flatness is fragile. Unlike our $k$-root layer injection, flatness can be undone by explicitly sharpening curvature. In particular, increasing the largest singular values via spectral reparameterization immediately destroys this alignment-induced flatness. We should be clear that the curvature increase still needs to be in the sub-divergence regime with only a small $\alpha$.

\paragraph{Spectral Reparameterization as an Attack}
To test this hypothesis, we apply SpecDef to 20 layers of each defended model with $\alpha = 10$ ($\alpha=1$ baseline). The results in \cref{tab:defence-sigma-comparison} are striking: resistance methods that were previously effective consistently fail once curvature is increased. We refer to this as a \emph{spectral reparameterization attack}.

\begin{table}[htbp]
\centering
\caption{
\textbf{Spectral Reparameterization Attacks.} Comparison of defended models under $\alpha=1$ vs.\ $\alpha=10$.
Entries report mean training steps $\pm$ std, with final accuracy $\pm$ std in parentheses.
$^{\dagger}$ indicates at least one divergent run (perplexity doubled);
$^{\ddagger}$ indicates all runs diverged.
}
\label{tab:defence-sigma-comparison}
\begin{sc}
\begin{tabular}{lcc}
\toprule
Resistance Method & $\alpha=1$ & $\alpha=10$ \\
\midrule
\texttt{RepNoise} & 510 (0.459) & 26 $\pm$ 5 (0.630 $\pm$ 0.022) \\
\texttt{RMU} & 510 (0.573) & 23 $\pm$ 5 (0.618 $\pm$ 0.013) \\
\texttt{TAR} & 510 (0.302) & 70 $\pm$ 45 (0.510 $\pm$ 0.170)$^{\dagger}$ \\
\texttt{GradDiff} & 510 (0.409 $\pm$ 0.002) & 196 $\pm$ 120 (0.609 $\pm$ 0.010) \\
\texttt{NPO-SAM} & 510 (0.412 $\pm$ 0.000) & 70 $\pm$ 43 (0.610 $\pm$ 0.008) \\
\texttt{NPO} & 510 (0.553) & 73 $\pm$ 15 (0.616 $\pm$ 0.016) \\
\texttt{SimNPO} & 510 (0.518) & 56 $\pm$ 25 (0.625 $\pm$ 0.030) \\
\texttt{Llama-3-8B} & 100 (0.608) & 20 (0.631 $\pm$ 0.012) \\
\bottomrule
\end{tabular}
\end{sc}
\end{table}

These results suggest that many existing resistance methods rely on loss-specific spectral alignment rather than $L$-based curvature control. From this perspective, our framework does not rule out these resistance methods as a class; rather, it explains both their apparent effectiveness and their brittleness. It also suggests a principled path forward: explicitly constructing resistance methods that control singular \emph{vector} alignment, not merely spectral magnitudes. Since we saw that pervious methods often pair well with SpecDef in the high $\alpha$ regime, future training resistance methods may be combined with SpecDef to provide an even higher cost to undo defences.

\paragraph{Curvature-aware Optimization as an Attack} \cref{tab:curvature-by-model-sigma1} shows that Sophia or AdaHessian is very effective (often more so than Adam) in undoing previous defence resistance. Muon was not very effective, which further corroborates that these methods may not operate by poor conditioning which Muon is effective at undoing. Instead, the results from Sophia and AdaHessian seem to point to flat curvature induced by the training objective as how they operate since second-order methods are designed to account for both high negative and positive curvature.

\begin{table}[htbp]
\caption{Curvature-aware optimizer attacks on defended models ($\alpha=1$). 
Entries report training steps to reach target metric (mean $\pm$ std across seeds). 
$^\dagger$: some seeds diverged; $^\ddagger$: all seeds diverged. Unlike with SpecDef, these attacks are often much more successful on defences that were previously able to resist Adam.}
\label{tab:curvature-by-model-sigma1}
\centering
\setlength{\tabcolsep}{3pt}
\begin{sc}
\resizebox{\textwidth}{!}{
\begin{tabular}{lccccccccc}
\toprule
 & \multicolumn{3}{c}{Muon} & \multicolumn{3}{c}{Sophia} & \multicolumn{3}{c}{AdaHessian} \\
\cmidrule(lr){2-4} \cmidrule(lr){5-7} \cmidrule(lr){8-10}
Model & $10^{-5}$ & $2\times10^{-5}$ & $5\times10^{-5}$ & $5\times10^{-6}$ & $10^{-5}$ & $2\times10^{-5}$ & $10^{-6}$ & $3\times10^{-6}$ & $5\times10^{-6}$ \\
\midrule
Base Model & 765 (0.52 $\pm$ 0.04) & 653 $\pm$ 193 (0.54 $\pm$ 0.05) & 190 $\pm$ 62 (0.62 $\pm$ 0.01) & 26 $\pm$ 5 (0.61 $\pm$ 0.01) & 20 (0.65 $\pm$ 0.05) & 16 $\pm$ 5 (0.63 $\pm$ 0.04) & 146 $\pm$ 25 (0.63 $\pm$ 0.01) & 60 (0.61 $\pm$ 0.01) & 56 $\pm$ 5 (0.62 $\pm$ 0.01) \\
\texttt{NPO} & 765 (0.29) & 765 (0.30) & 765 (0.55) & 60 (0.69) & 60 (0.54)$^{\ddagger}$ & 40 (0.42)$^{\ddagger}$ & 467 $\pm$ 420 (0.60 $\pm$ 0.02) & 140 $\pm$ 70 (0.61 $\pm$ 0.01) & 120 $\pm$ 42 (0.60 $\pm$ 0.00) \\
\texttt{GradDiff} & 765 (0.25 $\pm$ 0.03) & 765 (0.24 $\pm$ 0.03) & 765 (0.44 $\pm$ 0.04) & 73 $\pm$ 15 (0.63 $\pm$ 0.03) & 93 $\pm$ 5 (0.62 $\pm$ 0.02) & 56 $\pm$ 15 (0.42 $\pm$ 0.01)$^{\ddagger}$ & 553 $\pm$ 366 (0.56 $\pm$ 0.04) & 313 $\pm$ 210 (0.60 $\pm$ 0.00) & 183 $\pm$ 90 (0.60 $\pm$ 0.01)$^{\dagger}$ \\
\texttt{RMU} & 765 (0.23 $\pm$ 0.02) & 765 (0.30 $\pm$ 0.03) & 633 $\pm$ 228 (0.58 $\pm$ 0.02) & 46 $\pm$ 11 (0.65 $\pm$ 0.04) & 50 $\pm$ 43 (0.63 $\pm$ 0.01) & 33 $\pm$ 15 (0.54 $\pm$ 0.12)$^{\dagger}$ & 183 $\pm$ 60 (0.61 $\pm$ 0.01) & 76 $\pm$ 11 (0.63 $\pm$ 0.03) & 70 $\pm$ 10 (0.62 $\pm$ 0.02) \\
\texttt{TAR} & 765 (0.30 $\pm$ 0.02) & 765 (0.30 $\pm$ 0.03) & 435 $\pm$ 285 (0.29 $\pm$ 0.02)$^{\dagger}$ & 86 $\pm$ 5 (0.62 $\pm$ 0.02) & 93 $\pm$ 15 (0.58 $\pm$ 0.06)$^{\dagger}$ & 33 $\pm$ 11 (0.34 $\pm$ 0.04)$^{\ddagger}$ & 110 $\pm$ 26 (0.25 $\pm$ 0.06)$^{\ddagger}$ & 46 $\pm$ 11 (0.28 $\pm$ 0.05)$^{\ddagger}$ & 33 $\pm$ 5 (0.25 $\pm$ 0.00)$^{\ddagger}$ \\
\texttt{RR} & 765 (0.23 $\pm$ 0.02) & 765 (0.41 $\pm$ 0.03) & 383 $\pm$ 119 (0.60 $\pm$ 0.00) & 36 $\pm$ 5 (0.61 $\pm$ 0.01) & 36 $\pm$ 11 (0.65 $\pm$ 0.06) & 43 $\pm$ 11 (0.49 $\pm$ 0.13)$^{\dagger}$ & 140 $\pm$ 26 (0.61 $\pm$ 0.01) & 70 (0.62 $\pm$ 0.02) & 56 $\pm$ 5 (0.62 $\pm$ 0.01) \\
\texttt{RepNoise} & 765 (0.31 $\pm$ 0.03) & 765 (0.23 $\pm$ 0.04) & 410 (0.60) & 40 (0.66) & 30 (0.67) & 30 (0.62) & 65 $\pm$ 7 (0.63 $\pm$ 0.03) & 55 $\pm$ 21 (0.63 $\pm$ 0.01) & 45 $\pm$ 7 (0.63) \\
\bottomrule
\end{tabular}
}
\end{sc}
\end{table}

\paragraph{Alternative Explanation: Stochastic Obfuscation}
A second possibility, following \citet{athalye2018obfuscated}, is that these resistance methods increase the variance term governing stochastic convergence by biasing minibatch gradients, thereby violating the unbiased-gradient assumption underlying SGD. We tested this hypothesis via gradient accumulation (factor 10) and increased batch size (64). Neither intervention materially weakened the resistance in most methods, whereas curvature sharpening did. While this does not fully rule out stochastic effects, it suggests they are not the primary mechanism. For \texttt{DeepIgnorance} with strong filtering \citep{o2025deep}, we did find that these attacks did meaningfully change the resistance strength: at a $10^{-5}$ learning under our usual first-order attack the model could not be trained for WMDP-Bio. By applying a gradient accumulation of 10 and batch size of 16, the model reached over $60\%$ accuracy in around 200 training steps. This indicates that pre-training based methods might have obfuscated gradients w.r.t./ a harmful training objective. \texttt{DeepIgnorance} with a strong filter was similarly easily undone using curvature-aware methods but we did not report these since our goal is to compare defences applied to the same model and \texttt{DeepIgnorance} is itself a newly pretrained model.

\paragraph{Takeaway}
In summary, existing resistance methods do not significantly alter weight spectra. Their behaviour is most consistent with loss- and data-local spectral alignment that produces locally flat curvature directions. This observation both motivates our approach and highlights a future research direction: understanding how singular vector rotation, Jacobian spectra, and data alignment jointly shape effective curvature so that more robust resistance methods can be built.

\section{Layer Injection Attack Mathematical Details}
\label{app:layer-injection-details}

In this appendix, we will provide proofs and various remarks for the $k$-th layer injection attack as well as the universality result that all convergence rate control methods are spectral reparamaterization subject to this attack.

\subsection{Proof of $k$-th layer injection attack (\cref{thm:layer-injection-attack})}
\label{app:layer-injection-attack-proof}
First, we provide a proof for \cref{thm:layer-injection-attack} in \cref{app:layer-injection-attack-proof}. Restating the Theorem:

\begin{theorem}[$k$-th layer injection attack]
\label{thm:layer-injection-attack-app}
Consider a feed-forward network \[
f(x) = \theta_{n+1} \circ \phi_{n} \circ \theta_{n} \circ \cdots \circ \phi_1 \circ \theta_1x,
\] where each $\theta_i \in \mathbb{R}^{d_i \times d_{i -1}}$ is a linear layer and each $\phi_i$ is a fixed, elementwise activation. Let $M = \max_{1 \leq i \leq n} \sigma_1(\theta_i)$. Then to achieve $\max_i \sigma_1(\theta_i) \leq M^{1/k}$ it suffices to insert at most $(k - 1) \cdot n$ additional linear layers and perform spectral deformation.
\end{theorem}
\begin{proof}
Assume the defender would engage in the best possible resistance:
\[
\sigma_1(\theta_1) = \dots = \sigma_1(\theta_n) = c,
\]
where $c$ is large enough to resist convergence. From the theorem statement c = M. Since otherwise, the attack is less expensive (less layers to inject).

By \cref{thm:spectral-deformation-is-spec-reparam-main}, it suffices for only one weight to have $\sigma_1(\theta_i) = c$ to achieve spectral control over $H^f_{\theta}$. This implies that, assuming a defender has made the network non-convergent, to achieve convergence the attacker must ensure that
\[
\max\{\sigma_1(\theta_1), \dots, \sigma_1(\theta_n)\} \le r.
\]
where $r$ is selected to so that $1/r$ is a reasonable learning rate that gives convergence.

To reduce $\sigma_1(\theta_i)$, we can use a procedure analogous to our spectral deformation algorithm. Let
\[
\tilde{\Sigma}_{\theta_i} \gets T \Sigma_{\theta_i}, \quad T = \mathrm{diag}\bigl(r/\sigma_1(\theta_i), 1, \dots, 1\bigr),
\]
and reconstruct
\[
\tilde{\theta}_i = U_{\theta_i} \tilde{\Sigma}_{\theta_i} V_{\theta_i}^{\top}.
\]
Without loss of generality, we assume that $\sigma_2(\theta_i) < r$ (otherwise $T$ can be modified accordingly). For simplicity, we assume the network uses activation functions that are non-homogeneous, so compensation can only occur within layers, not across layers. Therefore, we inject the following compensation matrix after $\tilde{\theta}_i$:
\[
\tilde{I}_{i} = U_{\theta_i} \Sigma_{\theta_i} \tilde{\Sigma}_{\theta_i}^{-1} U_{\theta_i}^{\top},
\]
that is, we now have \[
f(x) = \theta_{n+1} \circ \phi_{n} \circ \theta_{n} \circ \cdots \circ \phi_{i} \circ \tilde{I}_{i} \tilde{\theta}_i \circ \cdots  \circ \phi_1 \circ \theta_1x
\]

Observe that by compensation
\[
\sigma_1(\tilde{\theta}_i) = r, \qquad \sigma_1(\tilde{I}_{i}) = \sigma_1(\theta_i)/r.
\]
We now seek the $r^*$ that minimizes the worst-case singular value. We have the following equality by definition:
\[
r^* = \arg\min_r \max\{\sigma_1(\tilde{\theta}_i), \sigma_1(\tilde{I}_{i})\} = \arg\min_r \max\{r, \sigma_1(\theta_i)/r\}.
\]
Since the first term of the max is increasing in $r$ and the second term is decreasing, it is well-known that the minimum occurs at their intersection, when \[
r = \frac{\sigma_1(\theta_i)}{r},
\] which works out to
\[
\quad r^* = \sqrt{\sigma_1(\theta_i)}.
\]
This is the maximum achievable reduction when injecting a single compensating layer, yielding
\[
\max\{\sigma_1(\tilde{\theta}_i), \sigma_1(\tilde{I}_{i})\} = \sqrt{\sigma_1(\theta_i)}.
\]

For $m$ injected compensating layers, we have $m+1$ total matrices in the factorization: the deformed layer $\tilde{\theta}_i$ and $m$ compensation matrices $\tilde{I}_i^{(1)}, \dots, \tilde{I}_i^{(m)}$. These must satisfy
\[
\tilde{I}_i^{(m)} \cdots \tilde{I}_i^{(1)} \tilde{\theta}_i = \theta_i,
\]
so the product of their singular values is constrained by $\sigma_1(\theta_i)$. To minimize the maximum singular value across all $m+1$ matrices, we set them equal. If each has spectral norm $r$, the constraint becomes $r^{m+1} \ge \sigma_1(\theta_i)$ and hence $r \ge \sigma_1(\theta_i)^{1/(m+1)}$ for the best possible worst-case singular value, giving optimal factorization
\[
r^* = \sigma_1(\theta_i)^{1/(m+1)},
\]
which yields $\sqrt[k]{\sigma_1(\theta_i)} = \sqrt[k]{c}$ for the given layer $\theta_i$ where $k = m + 1$. This is achievable by using SVD and distributing the singular values across the injected layers.

 So in order to achieve $\max\{\sigma_1(\theta_i)\}_i^{n \cdot (k - 1)} = \sqrt[k]{c}$, we need $(k - 1) \cdot n$ layers since we need to reduce the singular values for \textit{each} of the $n$ original layers.  This concludes the proof.
\end{proof}

\subsection{Layer injection is a Strong Universal Attack}
\label{app:layer-injection-is-universal}
In this section we show that the layer injection attack presented in \S~\ref{sec:security-analysis-of-convergence-rate-control} is an upper-max spectral reparameterization (\cref{thm:layer-injection-reparameterization}), which we define below (\cref{def:luk-spectral-reparameterization}). We also show why this attack would be universal under spectral reparameterizations of deep networks (\cref{thm:only-weight-matrices}). The implication of these two results is that this attack is a strong adaptive attack in the sense that if the attacker knows the curvature was manipulated to control convergence, then under relatively mild assumptions specified in \cref{thm:only-weight-matrices}, the curvature-based convergence control will be undone (though possibly at a cost that is too expensive for the attacker).

First, let us provide a general definition of $l$,$u$,$k$-spectral reparameterization.
 
\begin{definition}{$l$,$u$,$k$-Spectral Reparameterization}
\label{def:luk-spectral-reparameterization}
Given function $f$ parameterized by $\theta$ and a constant $l,u > 0$, $\epsilon \geq 0$, and $0 < k \leq n$. A map $\mathcal{T}_{u,l}: f_\theta \to f_{\theta^\prime}$ is a $l$,$u$,$k$ spectral reparameterization.
\begin{enumerate}
    \item \textbf{Spectral Control} $l \leq \sigma_k(H_{\theta}) \leq u$ where $H_{\theta}$ is the Hessian w.r.t.\ the parameters of $f_{\theta}$.
    \item \textbf{Functional Invariance up to $\epsilon$} There exists distance function $d(f, g)$ on the space of functions $f,g \in \mathcal{F}$ such that  $d(\mathcal{T}_{u,l}[f], f) \leq \epsilon $ for a small $\epsilon \geq 0$.
\end{enumerate}
\end{definition}

An upper-max spectral reparameterization is a $0$,$c$,$1$-Spectral Reparameterization where $c$ is a constant that provides an upper bound for the maximum singular values.

\begin{theorem}[Layer Injection is Upper-Max Spectral Reparameterization]
\label{thm:layer-injection-reparameterization}
Layer injection attacks with sufficiently selected layers are an upper-max spectral reparameterization.
\end{theorem}
\begin{proof}
To prove this we need to show two things (1) functional invariance and (2) spectral control.

\textbf{(1) Functional invariance}
Functional invariance proof is the exact same as that of the proof for \cref{thm:spectral-deformation-is-spec-reparam-main}.

\textbf{(2) Spectral Control}

First, we show that \[
\sigma_1(H_{\theta}) \leq \sum_{i,j}\sigma_1 \left(\frac{\partial^2 f}{
\partial \theta_i \partial \theta_j
} \right).
\] This follows from the application of the triangle inequality of the spectral norm $\| \cdot\|_2$ to the following block decomposition.

Any given matrix $M$ admits the following sum over its disjoint submatrices
\begin{equation*}
M = \begin{bmatrix}
A_{1,1} & \ldots & A_{1,m} \\
\vdots & \ddots & \vdots \\
A_{n,1} & \ldots & A_{n,m}
\end{bmatrix} = \begin{bmatrix}
A_{1,1} & \ldots & 0 \\
\vdots & \ddots & \vdots \\
0 & \ldots & 0
\end{bmatrix}  + \cdots + \begin{bmatrix}
0 & \ldots & 0 \\
\vdots & \ddots & \vdots \\
0 & \ldots & A_{n,m}
\end{bmatrix}.
\end{equation*}

Suppose we had the Hessian matrix $H_{\theta} \in \mathbb{R}^{n \times n}$. Let $H_{i,j}^0 \in \mathbb{R}^{n \times n}$ be the matrix $\frac{\partial^2 f}{
\partial \theta_i \partial \theta_j
}$ embedded in a $0$ matrix in the original position $i,j$ it would have been in $H$. Here $i,j$ index the blocks of $H_{\theta}$ corresponding to the vectorized parameter matrices $\theta_i$ and $\theta_j$ where each block may itself be a matrix.

Now it follows that
\[
H_{\theta} = \sum_{i,j}^{n} H_{i,j}^0.
\]
Using the subadditivity of operator norms, we see that the spectral norm and hence $\sigma_1$ has the following property: \[
 \sigma_1\left( H \right) = \left\|H\right\|_2 = \left\|\sum_{i,j} H_{i,j}^0\right\|_2 \leq \sum_{i,j} \left\|H_{i,j}^0\right\|_2
\]
Since embedding a matrix in a larger 0 matrix does not change the spectral norm, we have that $\left\|H_{i,j}^0\right\|_2 = \left\|\frac{\partial^2 f}{
\partial \theta_i \partial \theta_j
}\right\|_2 = \sigma_1\left(\frac{\partial^2 f}{
\partial \theta_i \partial \theta_j
} \right)$. Putting this together proves that \[
\sigma_1(H_{\theta}) \leq \sum_{i,j}\sigma_1 \left(\frac{\partial^2 f}{
\partial \theta_i \partial \theta_j
} \right).
\]

Next all we need to show is that assuming at least one Hessian block $\frac{\partial^2 f}{\partial \theta_i \partial \theta_j}$ admits the algebraic structure $A\theta_kB$, we can control the upper bound on $\sigma_1(H_{\theta})$ established above by manipulating $\sigma_1(\theta_k)$. Like in the other proofs, we assume that this structure exists, which has been shown in many examples throughout the paper for neural networks. Now we just directly apply the upper bound of \cref{thm:multiplicative-singular-value-bounds}. Yielding \begin{align}
\sigma_1 \left(\frac{\partial^2 f}{
\partial \theta_i \partial \theta_j
} \right) &= \sigma_1 \left(A\theta_kB\right) \tag{By assumption} \\
&\leq \sigma_1(A)\sigma_1(\theta_k)\sigma_1(B). \tag{By \cref{thm:multiplicative-singular-value-bounds}}
\end{align}

Taken together, this shows that by reducing $\sigma_1(\theta_i)$  we have some upper spectral control of $\sigma_1(H^f_{\theta})$. To obtain the upper bound $u$, we must select a sufficient number of layers $i=k$ such that the upper bound sum we showed earlier is not dominated by some layer $\theta_p$ that is not used. We now be sure that $u$ can be chosen based on the desired maximum singular value after layer injection, as demonstrated in \cref{thm:layer-injection-attack}.
\end{proof}

We now show an impossibility result. Under weak assumptions that most deep learning settings satisfy, convergence rate control must be a 
spectral reparameterization using parameter matrices alone 
(\cref{thm:only-weight-matrices}). Combined with \cref{thm:layer-injection-reparameterization}, 
which establishes that the layer injection attack is an upper-max spectral reparameterization, this implies that any lower-max spectral resistance method 
operating on weights can be countered: the defender must increase $\sigma_1(H^f_{\theta})$ through weight manipulation, and the attacker can subsequently decrease it while preserving functionality. Thus, no spectral reparameterization of parameter matrices exists that cannot be undone by our attack.

We first state a lemmas that will help us in our proof.
\begin{lemma}
\label{lemma:outer-product-spectral-norm}
The spectral norm of an outer product of vectors $u,v \in \mathbb{R}^d$ is given by \[
\|uv^{\top}\|_2 = \|u\|_2 \|v \|_2.
\]
\end{lemma}
\begin{proof}
Consider the following variational characterization of the spectral norm \[
\|A\|_2 = \max_{\|x\|_2 = 1} \|Ax\|_2.
\]
For an outer product $uv^{\top}$, we have 
\begin{align*}
\|uv^{\top}\|_2 &= \max_{\|x\|_2 = 1} \|u(v^{\top}x)\|_2 \\
&= \max_{\|x\|_2 = 1} |v^{\top}x| \|u\|_2 \tag{absolute homogeneity} \\
&= \|u\|_2\|v\|_2,   \tag{by Cauchy-Schwarz, max at $x = v/\|v\|_2$}
\end{align*}
as required.
\end{proof}

\begin{theorem}[Only Weight Matrices Provide Unbounded Spectral Control]
\label{thm:only-weight-matrices}
Suppose we have the following feed forward network \[
f(x) =  \theta_{n+1} \circ \phi_{n} \circ \theta_{n} \circ \cdots \circ \phi_1 \circ \theta_1x,
\] with element-wise activation functions $\phi_i: \mathbb{R}^d \to \mathbb{R}^d$ and parameter matrices $\theta_i \in \mathbb{R}^{d_i \times d_{i-1}}$. Further, suppose that the activation functions $\phi$ were approximately non-expansive and had non-expansive derivatives and second derivatives if they exist. Also, let the data be bounded with $\|x\|_2 \leq c$ for any $x$ drawn from the dataset $\mathcal{D}$.

A transformation of that function $\mathcal{T}: \mathcal{F} \to \mathcal{F}$ is a lower max spectral reparameterization (i.e., a $(l, \infty, 1)$-spectral reparameterization)  only if the transformation is a reparameterization of the parameter matrices $\theta_i$.
\end{theorem}

\begin{proof}
The reverse implication (only if) will be proven by showing that the function composition we assumed, our deep neural architecture, has only one unbounded component, the parameter matrices $\theta_i$ which allows us to achieve an arbitrary $l$ with which $\sigma_1(H_{\theta}^f) \geq l$. 

Recall that one condition of lower max spectral reparameterization was that the transformation must be able to achieve \[
\sigma_{1}(H_{\theta}^f) \geq l,
\] for any $l$.

Suppose that the transformation involved sampling an $x$ to control convergence in such a way while  respecting, by our earlier assumption, that $\|x\| \leq c$. Recall that we had the block decomposition such that $\sigma_1(H_{\theta}^f) \geq \sigma_1(\frac{\partial^2 f}{\partial \theta_i \partial \theta_j})$. Since $x$ appears in the function composition, these Hessian blocks will consist of outer products involving $x$ or Kronecker products involving $x$ to maintain the matrix structure of the block.

For outer products, we have blocks of the form $Aux^\top B$ for weight matrices $A, B$ and some intermediate vector $u$. By \cref{lemma:outer-product-spectral-norm}, 
\[
\sigma_1(Aux^\top B) \leq \sigma_1(A)\|u\|\|x\|\sigma_1(B) \leq c \cdot \sigma_1(A)\|u\|\sigma_1(B).
\]

For Kronecker products, we have blocks of the form $A(x^\top \otimes U)B$. Since spectral norms are multiplicative under Kronecker products, $\sigma_1(x^\top \otimes U) = \|x\|\sigma_1(U)$, and thus
\[
\sigma_1(A(x^\top \otimes U)B) \leq \sigma_1(A)\|x\|\sigma_1(U)\sigma_1(B) \leq c \cdot \sigma_1(A)\sigma_1(U)\sigma_1(B).
\]

In both cases, the contribution of $x$ to $\sigma_1$ of any Hessian block is bounded above by $c$ times a constant $K$ depending only on the weight matrices. Therefore, $\sigma_1(H_\theta^f) \leq c \cdot K$ for some $K > 0$ independent of the data. For any target $l > c \cdot K$, manipulation of the data alone cannot achieve $\sigma_1(H_{\theta}^f) \geq l$, contradicting the requirement of a lower-max spectral reparameterization.

Suppose that the transformation involved selecting an activation function under the assumptions above. We have seen that by \cref{thm:hessian-singular-value-bound} for $\sigma_1(\frac{\partial^2f}{\partial \theta_i \partial \theta_j})$, we have \[
\sigma_{r_1}(A)\sigma_1(B) \sigma_{r_2}(C)\prinangle(A_{r_1},BC_1) \prinangle(B,C_{r_2})  \leq \sigma_1(\nabla^2_{\theta}\mathcal{L} ).
\] In this case, $B$ involves some products containing either $D^{\prime}_{z_i}$ or $D^{\prime\prime}_{z_i}$ rather than $\theta_k$. 
Since the activation functions are assumed to be element-wise approximately non-expansive and to have approximately non-expansive first and second derivatives, we have \[
    |\phi^{\prime}_i(t)| \lesssim 1 \; \mathrm{and} \; |\phi^{\prime\prime}_i| \lesssim 1 \quad \text{for all } t\in \mathbb{R}.
\]
Consequently, the diagonal matrices $D^{\prime}_{z_i}$ and $D^{\prime\prime}_{z_i}$, which embed these derivatives at the pre-activations $z_i$, have the operator norm uniformly bounded by one: \[
\|D^{\prime}_{z_i}\|_2 \lesssim 1, \quad \|D^{\prime\prime}_{z_i}\|_2 \lesssim 1.
\]
Therefore, any Hessian block involving activation derivatives contributes at most a constant (independent of the weights) to $\sigma_1(H_{\theta}^f)$. In particular, manipulating the activation functions alone cannot produce arbitrarily large singular values of the Hessian, and thus cannot satisfy the lower max spectral reparameterization condition for sufficiently large $\ell$. For piece-wise linear activations that may not be twice differentiable case it is even more true they cannot contribute unbounded increases to the Hessian singular values.  Due to this contradiction manipulation of the activation functions alone cannot be a spectral reparameterization.

The only remaining possibility is manipulation of weights $\theta_i$ which we have already shown is a spectral reparameterization in \cref{thm:spectral-deformation-is-spec-reparam-main}.
\end{proof}

\begin{remark}[Applicability to Modern Deep Networks]
\label{app:applicability-impossiblity-result-to-modern-deep-networks}
The bounded data assumption $\|x\|_2 \leq c$ is mild in practice. For 
images, pixel normalization ensures boundedness. For language models, 
token embeddings are bounded by construction. Even when data distributions 
have infinite support, training datasets have finite maximum norms.
Non-elementwise functions such as softmax and normalization are also bounded or have a small Lipschitz factor such that using them to control convergence rate is infeasible. Counter-examples may exist outside of our boundedness assumptions but they would take the form of a very exotic architecture. These assumptions are common (e.g., \citealp{du2019gradient, xiong2020layer}). GELU and Swish are approximately 1-Lipschitz so modern activation functions are also captured here.
\end{remark}

\begin{remark}[Implications of \cref{thm:only-weight-matrices}]
This result establishes that:
\begin{enumerate}
\item Under standard architectural assumptions, convergence control requires 
    weight manipulation
\item Any such manipulation is subject to the layer injection attack 
\item Therefore, convergence-rate-based resistance methods do not provide security 
    against adaptive attackers
\end{enumerate}
This impossibility result is constructive: it precisely quantifies the 
attack cost as $(k-1) \cdot n$ layers for $k$-th root reduction. While 
this creates practical friction, it 
does not prevent a sufficiently resourced adversary.
\end{remark}

Finally in practice, injecting a large number of linear layers could result in gradient explosion via repeated unnormalized multiplications during backpropagation. We only observed this in practice we performing the layer injection on all layers of a fully reparameterized \texttt{SmolLM2-360M-Instruct}.

\section{Analysis of Second-Order Methods}
\label{app:second-order-analysis}

Second-order optimization and curvature-aware methods are increasingly popular and feasible for large deep networks \citep{liusophia,yao2021adahessian}. At first glance, it may appear that approximate second-order methods could easily undo spectral reparameterization, since these methods provide preconditioners that mitigate sharp curvature. We show, however, that this intuition fails: while preconditioning can neutralize large $L_1$-Lipschitz constants (i.e., large Hessian eigenvalues), spectral reparameterization controls the $L_2$-Lipschitz constant of the Hessian itself (under the assumptions discussed in \cref{rem:L2-control-conditions}), which bounds the iteration complexity of second-order methods independently of preconditioning. Specifically, even exact Newton's method---which perfectly inverts the Hessian---has convergence rates that depend on $L_2$, and our reparameterization often inflates this quantity.

We validate this empirically in \cref{app:second_order_toy_empirical_analysis} on a toy problem where the full Hessian is tractable. That analysis visualizes how the $L_2$ inflation appears concretely as large off-diagonal curvature, confirming that strong second-order optimizers---including KFAC, Natural Gradients, AdaHessian, and Cubic Newton---remain unable to circumvent spectral deformation. Large-scale second-order experiments on foundation models appear in \cref{tab:optimizer-comparison-main}.

For the theoretical work, a roadmap of the analysis is as follows. After defining the $L_2$ Hessian Lipschitz constant, we restate the convergence rate and iteration complexity of Newton's method, including full proofs for completeness. We then provide tensor-based extensions of singular value bounds to show lower bound inequalities that relate the $L_2$ constant to properties of the weights alone, stating the conditions under which spectral reparameterization controls the convergence of second-order methods.

\subsection{Empirical Analysis}
\label{app:second_order_toy_empirical_analysis}

\subsubsection{Demonstrating Second-Order Convergence Control}

We empirically demonstrate that Spectral Deformation (SpecDef) cannot be undone by a wide range of second-order and curvature-aware optimizers on a controlled toy problem. The task is binary classification of points sampled from two intertwined spirals, trained using the three-layer ReLU MLP introduced earlier. This setting allows us to compute the full Hessian explicitly and isolate the mechanism by which SpecDef controls convergence.

Our analysis has two goals:  
(i) to show that SpecDef concentrates curvature primarily in off-diagonal Hessian structure, explaining the failure of diagonal and block-diagonal approximations; and  
(ii) to empirically validate \cref{thm:hessian-singular-value-bound} by showing that SpecDef introduces extremely large-magnitude Hessian eigenvalues.

\paragraph{Experimental Setup}
We evaluate 200 data points across 5 random seeds, training for 1000 epochs per run. The MLP has input dimension 2 and hidden widths of 8. Unless otherwise stated, SpecDef is applied with $\alpha = 10\mathrm{k}$. Because second-order methods are highly sensitive to hyperparameters, we conduct a full grid search for each optimizer, summarized in \cref{tab:second_order_optimizer_grids}. We perform this search twice: once on the base model and once after applying SpecDef with $\alpha = 10^5$, ensuring a fair comparison. The resulting optimal hyperparameters are reported in \cref{tab:specdef_so_optimizer_hyperparams}.

\begin{table}[h]
\centering
\caption{Hyperparameter grids used for second-order optimizer ablations.}
\label{tab:second_order_optimizer_grids}
\begin{tabular}{lll}
\hline
\textbf{Optimizer} & \textbf{Hyperparameter} & \textbf{Values} \\
\hline
Adam & Learning rate &
$\{10^{-9}, \dots, 10^{-1}\}$ \\

KFAC & Learning rate &
$\{10^{-9}, \dots, 10^{-2}\}$ \\
     & Damping &
$\{10^{-3}, 10^{-2}, 3\!\times\!10^{-2}, 10^{-1}, 3\!\times\!10^{-1}\}$ \\

Natural Gradient & Learning rate &
$\{10^{-9}, \dots, 10^{-2}\}$ \\
                 & Damping &
$\{0.1, 0.3, 1.0, 3.0\}$ \\

Newton & Learning rate &
$\{10^{-9}, \dots, 10^{-2}\}$ \\
       & Damping &
$\{0.1, 0.3, 1.0, 3.0\}$ \\

AdaHessian & Learning rate &
$\{10^{-9}, \dots, 10^{-1}\}$ \\

Cubic Newton & Learning rate &
$\{10^{-9}, \dots, 10^{-1}\}$ \\
             & $\sigma$  &
$\{0.1, 0.5, 1.0, 5.0\}$ \\
\hline
\end{tabular}
\end{table}

\paragraph{SpecDef Forces Universally Smaller Learning Rates}
Across all optimizers, SpecDef sharply reduces the stable learning rate—often by 4–6 orders of magnitude—and increases the required damping. This empirically confirms that SpecDef enforces convergence control even for curvature-aware methods by constructing a sharply conditioned landscape that must be heavily regularized.

\begin{table}[h]
\centering
\renewcommand{\arraystretch}{1.25}
\caption{Optimal optimizer hyperparameters before and after SpecDef.}
\label{tab:specdef_so_optimizer_hyperparams}
\begin{tabular}{lcc}
\hline
\textbf{Optimizer} & \textbf{Before SpecDef} & \textbf{After SpecDef} \\
\hline
Adam &
$\text{lr}=10^{-2}$ &
$\text{lr}=10^{-7}$ \\

KFAC &
$\text{lr}=10^{-3},\ \text{damping}=10^{-3}$ &
$\text{lr}=10^{-4},\ \text{damping}=10^{-2}$ \\

Natural Gradient &
$\text{lr}=10^{-2},\ \text{damping}=0.1$ &
$\text{lr}=10^{-8},\ \text{damping}=3.0$ \\

Newton &
$\text{lr}=10^{-2},\ \text{damping}=0.1$ &
$\text{lr}=10^{-6},\ \text{damping}=1.0$ \\

AdaHessian &
$\text{lr}=10^{-3}$ &
$\text{lr}=10^{-4}$ \\

Cubic Newton &
$\text{lr}=0.1,\ \sigma=5.0$ &
$\text{lr}=10^{-4},\ \sigma=5.0$ \\
\hline
\end{tabular}
\end{table}

\paragraph{Second-order Methods Fail under SpecDef}
Using these tuned hyperparameters, \cref{tab:specdef_toy_second_order_results} reports classification accuracy before and after SpecDef. After deformation, no optimizer is effective. Natural Gradient and Newton’s method typically converge to poor local minima in either cases, achieving accuracies near chance ($\approx 0.5$), consistent with known failures of exact second-order methods in non-convex landscapes. Methods that perform well prior to SpecDef—including KFAC, AdaHessian, and Cubic Newton—are likewise rendered ineffective.

Notably, Adam performs best overall despite not being explicitly curvature-aware (though it still applies a diagonal preconditioner via second‑moment estimates). This suggests that even scalable second-order approximations are insufficient to counteract SpecDef, corroborating our large-scale findings in \cref{tab:optimizer-comparison-main}.

\begin{table}[h]
\centering
\renewcommand{\arraystretch}{1.25}
\caption{Classification accuracy before and after SpecDef (mean $\pm$ std).}
\label{tab:specdef_toy_second_order_results}
\begin{tabular}{lcc}
\hline
\textbf{Optimizer} & \textbf{Before SpecDef} & \textbf{After SpecDef} \\
\hline
KFAC & $0.835 \pm 0.031$ & $0.616 \pm 0.058$ \\
Natural Gradient & $0.624 \pm 0.005$ & $0.589 \pm 0.045$ \\
Newton & $0.625 \pm 0.006$ & $0.621 \pm 0.007$ \\
AdaHessian & $0.795 \pm 0.048$ & $0.648 \pm 0.027$ \\
Cubic Newton & $0.748 \pm 0.038$ & $0.642 \pm 0.026$ \\
Adam & $0.806 \pm 0.042$ & $0.648 \pm 0.011$ \\
\hline
\end{tabular}
\end{table}

The learning dynamics in \cref{fig:second_order_optimizers} make this effect explicit: SpecDef shifts the stable learning-rate regime downward, yielding second-order convergence control.

\begin{figure}[h]
    \centering
    \includegraphics[width=1\linewidth]{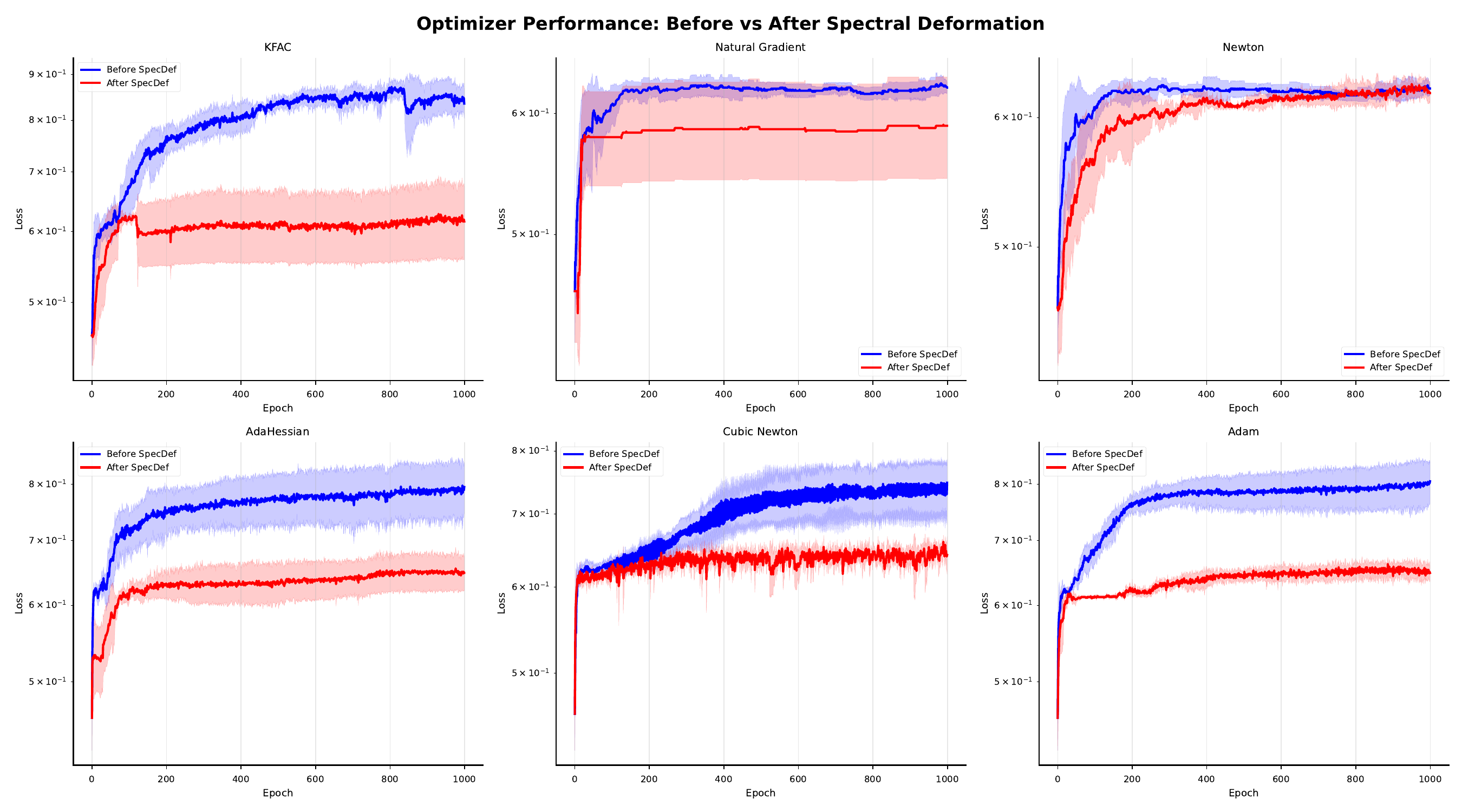}
    \caption{SpecDef sharply contracts the stable learning-rate regime, demonstrating effective second-order convergence control.}
    \label{fig:second_order_optimizers}
\end{figure}

\subsubsection{Hessian Analysis}

To explain why SpecDef remains effective even against second-order optimization, we analyze the full Hessian before and after deformation. \cref{fig:hessian-headmap-histo-comparison} shows both the Hessian heatmap and its eigenvalue spectrum.

Before SpecDef, the spectrum follows a distribution consistent with random matrix theory \citep{baskerville2022appearance}. After SpecDef (with $k=4$), the spectrum spreads across many orders of magnitude: most curvature directions remain relatively flat, but a small number become extremely sharp, both positively and negatively. This heavy-tailed spectrum directly instantiates the bound in \cref{thm:hessian-singular-value-bound}.

\begin{figure}[h]
    \centering
    \includegraphics[width=1\linewidth]{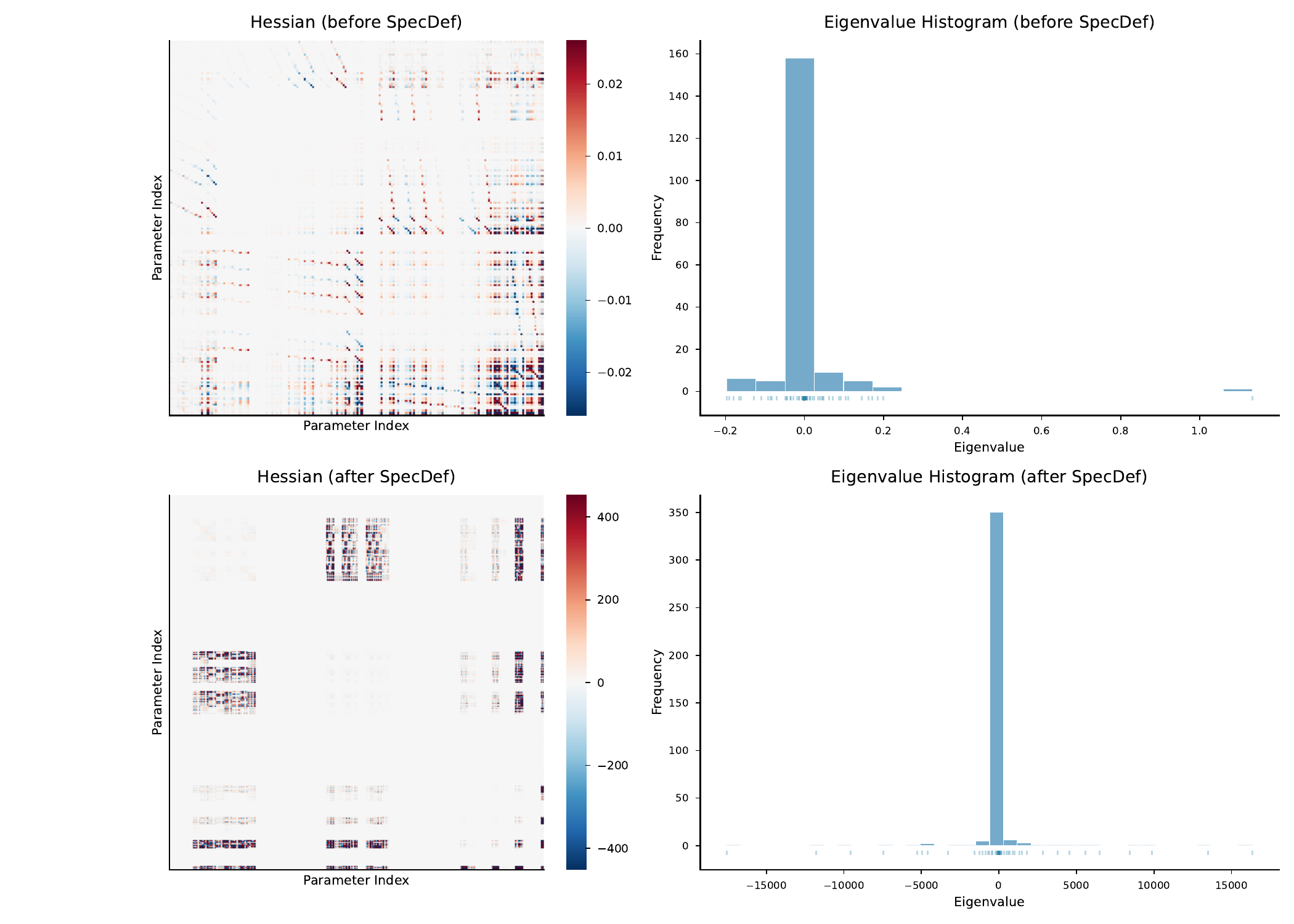}
    \caption{Hessian heatmaps and eigenvalue distributions before and after SpecDef. SpecDef's most pronounced effect is in the off-diagonals.}
    \label{fig:hessian-headmap-histo-comparison}
\end{figure}

Crucially, the dominant curvature after SpecDef concentrates in off-diagonal Hessian blocks. This explains the empirical failure of diagonal and block-diagonal curvature approximations: the sharpest directions are in places these methods do not approximate.

We further vary both the number of modified singular values and the choice of reparameterized layers. As shown in \cref{fig:hessian-heatmap}, increasing the number of modified singular values systematically sharpens the landscape, while changing which layers are deformed redistributes where curvature concentrates. These patterns provide a second, independent empirical validation that manipulating weight spectra alone suffices to control second-order curvature.

\begin{figure}[h]
    \centering
    \includegraphics[width=1\linewidth]{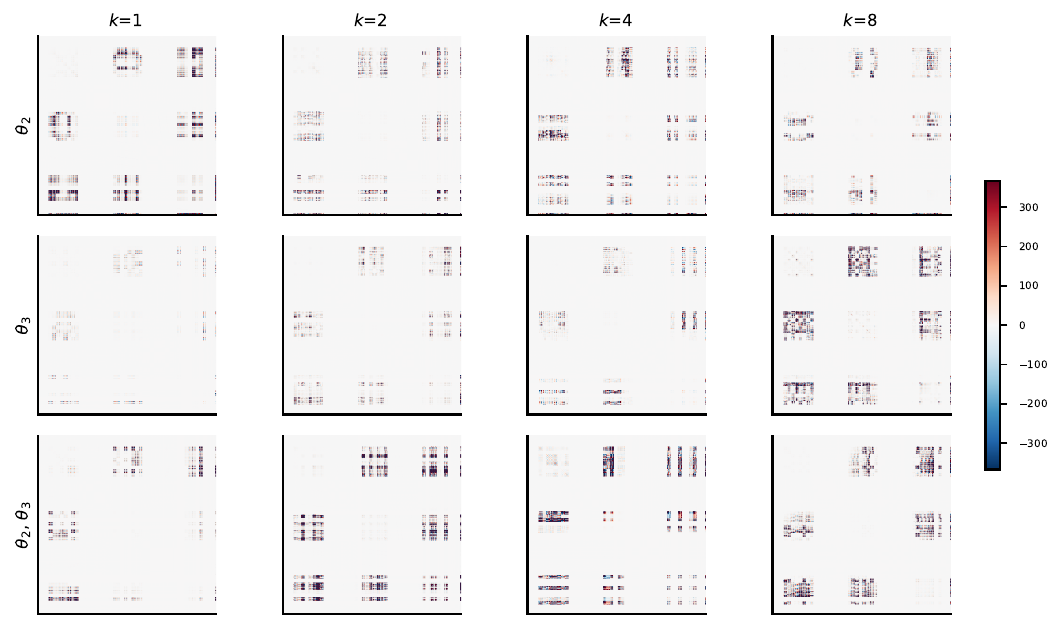}
    \caption{Hessian heatmaps as a function of the number and location of modified singular values. The pattern is consistent regardless of selected $k$ or layer.}
    \label{fig:hessian-heatmap}
\end{figure}

\clearpage
\subsection{Theoretical Analysis of Muon}
\label{subsec:muon-analysis}

In order to explain why Muon \citep{jordan2024muon} fails to converge given a sufficient $\alpha$ used for SpecDef, we briefly review the Muon optimizer and its known convergence rate. We will use the notation from \citet{li2025muon} rather than our paper to better connect our results with theirs.

\begin{definition}[Muon optimizer in heavy-ball form]
\label{def:muon}
Following \citet{li2025muon}, we consider minimizing an objective function $f(X)$ over $X\in\mathbb{R}^{m\times n}$ where $X$ represents a parameter matrix.
At iteration $t$, Muon samples a mini-batch $\{\xi_{t,i}\}_{i=1}^B$ and forms the stochastic gradient
\begin{equation}
G_t \;\gets\; \frac{1}{B}\sum_{i=1}^B \nabla f(X_t,\xi_{t,i}) \in \mathbb{R}^{m\times n}.
\end{equation}
Muon maintains a momentum matrix (heavy-ball style)
\begin{equation}
B_t \;\gets\; \beta B_{t-1} + (1-\beta)G_t,
\end{equation}
where $\beta\in(0,1)$.

Let $B_t = U_t S_t V_t^\top$ be an SVD result (with $U_t\in\mathbb{R}^{m\times r}$, $S_t\in\mathbb{R}^{r\times r}$, $V_t\in\mathbb{R}^{n\times r}$, and $r=\mathrm{rank}(B_t)$).
Muon discards singular values and uses only the singular directions:
\begin{equation}
O_t \;\gets\; U_t V_t^\top,
\end{equation}
then updates
\begin{equation}
X_{t+1} \;\gets\; X_t - \eta_t O_t,
\end{equation}
where $\eta_t>0$ is the step size.
(Compared to a Nesterov-type presentation, we use the above heavy-ball form where $G_t$ is explicitly multiplied by $(1-\beta)$ in the momentum update.)
\end{definition}

\begin{theorem}[Convergence of Muon optimizer]
\label{thm:muon-conv}
Assume $f$ is $L$-smooth in Frobenius norm, i.e.,
$\|\nabla f(X)-\nabla f(Y)\|_F \le L\|X-Y\|_F$,
and the stochastic gradient is unbiased with bounded variance:
$\mathbb{E}[\nabla f(X,\xi)] = \nabla f(X)$ and
$\mathbb{E}\|\nabla f(X,\xi)-\nabla f(X)\|_F^2 \le \sigma^2$.
Let $R = f(X_1)-f^*$, where $f^* = \inf_X f(X)$ denotes the optimal objective value. All expectations are taken with respect to the stochasticity of the mini-batch sampling.

Take $\beta = 1-\alpha$ with $\alpha = \min\!\left(\frac{\sqrt{RL}}{\sigma\sqrt{T}},\,1\right)$,
use a constant step size $\eta_t=\eta$ with $\eta = \sqrt{\frac{4R}{(10/\alpha + 2n)\,TL}},$ and set $B=1$. Then the Muon iterates in Definition~\ref{def:muon} satisfy
\begin{equation}
\frac{1}{T}\sum_{t=1}^T \mathbb{E}\big[\|\nabla f(X_t)\|_F\big]
\;\le\;
\mathcal{O}\!\left(
\frac{\sqrt{nRL}}{\sqrt{T}}
\;+\;
\frac{\sigma^2}{\sqrt{RLT}}
\;+\;
\frac{\sqrt{\sigma}\,(RL)^{1/4}}{T^{1/4}}
\right).
\end{equation}

If $\beta$ is an arbitrary fixed constant, take $B=T$. Then one obtains
\begin{equation}
\frac{1}{T}\sum_{t=1}^T \|\nabla f(X_t)\|_F
\;\le\;
\mathcal{O}\!\left(
\frac{\sqrt{nRL}}{\sqrt{T}}
\;+\;
\frac{\sigma}{T^{3/2}}
\;+\;
\frac{\sigma}{\sqrt{T}}
\right).
\end{equation}

\smallskip
\noindent
\emph{Proof.} The proof follows directly from the analysis in \citet{li2025muon}.
\end{theorem}

At first glance, one might expect Muon to perform well as the model spectra is increased or deformed, since its SVD-based update removes singular value magnitudes and appears robust to ill-conditioning. However, the convergence rate reveals a more fundamental limitation.

Specifically, the leading term in the convergence bound scales as $\mathcal{O}(\sqrt{nRL/T})$, where $L$ denotes the Lipschitz smoothness constant of the objective, capturing the overall curvature of the loss landscape. Increasing the SpecDef strength $\alpha$ applies stronger spectral deformation to the model parameters, which in turn increases the effective curvature of the objective function and thus increases $L$ \cref{thm:hessian-singular-value-bound}. As a result, the step size required for stable optimization becomes smaller, and convergence slows down under a fixed number of training iterations. This effect cannot be addressed by the SVD-based update used in Muon, which normalizes the update direction but does not change the global curvature of the objective.

This behaviour is reflected in Table~\ref{tab:optimizer-comparison-main}.
As $\alpha$ increases, Muon requires more iterations and achieves lower success rates, consistent with the curvature-dependent limitation predicted by the convergence analysis.

\subsection{Theoretical Analysis of Second-Order Methods}

In this section we will show how $L_2$ relates to the spectral values of the weight matrices of neural networks $\theta_i$. We will begin with some preliminary definitions necessary for understanding the iteration complexity of Newton's method which will be our canonical second order method we use for analysis. We will then employ a tensor analysis specifically using the T-SVD of \citet{zhang2021note} to show a number of analagous results to our first-order ones: an $L_2$ bound in $\nabla^3_{\theta}f$, versions of the matrix singular value inequalities for the tensors, and finally applications to neural networks which will explain the empirical results we obtained above.

\subsubsection{Preliminaries}
\begin{definition}[$L_2$-Lipschitz Continuous Hessian]
\label{def:lip-cont-hes}
A twice continuously differentiable function $f:\mathbb{R}^n\to\mathbb{R}$ 
is said to have an \emph{$L_2$-Lipschitz continuous Hessian} if there exists a constant $L_2>0$ such that
for all $x,y\in\mathbb{R}^n$,
\[
\|\nabla^{2} f(x)-\nabla^{2} f(y)\|\le L_2\|x-y\|.
\]
in which case we call $L_2$ the \emph{Hessian Lipschitz constant}.
\end{definition}

\begin{theorem}[Jensen's Spectral Norm Inequality]
\label{theorem:jensen-spect}
Let $S(u): [0,1] \to \mathbb{R}^{n \times n}$ be a integrable matrix-valued function. 
Then its spectral norm satisfies
\[
\Bigl\| \int_0^1 S(u)\,du \Bigr\| 
\le \int_0^1 \|S(u)\|\,du.
\]
\end{theorem}
This result is a matrix version of Jensen’s inequality for the spectral norm. 
A detailed proof can be found in \citet{zhang2023notes}. 

\begin{lemma}[Lipschitz Hessian Bound]
\label{lemma:lip-hes-bound}
Let $f:\mathbb{R}^n \to \mathbb{R}$ be a twice continuously differentiable 
function with an $L_2$-Lipschitz continuous Hessian as defined in 
Definition~\ref{def:lip-cont-hes}. Then, for any two points 
$x, y \in \mathbb{R}^n$, the corresponding Hessians satisfy
\[
\nabla^{2} f(x) - L_2 \|x - y\|\, I 
\;\preceq\; \nabla^{2} f(y) 
\;\preceq\; \nabla^{2} f(x) + L_2 \|x - y\|\, I.
\]
\end{lemma}
This lemma provides upper and lower bounds on the variation of the Hessian 
under the Lipschitz continuity assumption, quantifying how much the curvature 
may change between two nearby points. A detailed proof is given in 
\citet{zhang2023notes}.

\subsubsection{Convergence Rate of Newton’s Method on Lipschitz-Hessian Functions}

We now state the convergence rate of Newton's method and it's iteration complexity to parallel our analysis above with first-order methods.

\begin{theorem}[Local quadratic convergence of Newton’s method]
\label{theorem:newton-local-quad}
Let $\mathcal{L}:\mathbb{R}^n\to\mathbb{R}$ be a twice continuously differentiable function,
and suppose $\nabla \mathcal{L}(\theta_*) = 0$ and 
$\nabla^2 \mathcal{L}(\theta_*) \succeq \mu I$ for some $\mu > 0$, 
where $\mu$ denotes the local eigenvalue lower bound of the Hessian at $\theta_*$.
Assume further that $\mathcal{L}$ has an $L_2$-Lipschitz continuous Hessian as defined in Definition~\ref{def:lip-cont-hes}.

Then, for any constant $\rho \in (0, \mu/L_2)$, if an iterate satisfies $\|\theta_k - \theta_*\| \le \rho$, the Newton update
\[
\theta_{k+1}=\theta_k-[\nabla^2 \mathcal{L}(\theta_k)]^{-1}\nabla \mathcal{L}(\theta_k)
\]
satisfies the local quadratic convergence bound:
\[
\|\theta_{k+1} - \theta_*\|
\;\le\;
\psi(\rho)\,\|\theta_k - \theta_*\|^2,
\qquad
\psi(\rho) := \frac{L_2}{2(\mu - L_2 \rho)}.
\]
Furthermore, if $\psi(\rho)\rho < 1$ (equivalently, $\rho < 2\mu/(3L_2)$), then the neighborhood $\|\theta - \theta_*\| \le \rho$ is strictly invariant and the iteration is contractive:
\[
\|\theta_{k+1} - \theta_*\| < \|\theta_k - \theta_*\|.
\]
Under this condition, Newton’s method converges quadratically with a monotonically decreasing error norm.
\end{theorem}
\begin{proof}
Consider the straight-line path
$\Gamma : \theta(u) = \theta_* + u(\theta_k - \theta_*)$, for $u \in [0,1]$. By the Fundamental Theorem of Calculus, we obtain:
\begin{equation}\label{eq:g.4.1}
\nabla \mathcal{L}(\theta_k)-\nabla \mathcal{L}(\theta_*) 
= \int_{\Gamma} \nabla^{2} \mathcal{L}(\theta)\, d\theta 
= \int_{0}^{1} \nabla^{2} \mathcal{L}(\theta(u)) \frac{d\theta}{du}\, du
= \int_0^1 \nabla^2 \mathcal{L}\bigl(\theta_*+u(\theta_k-\theta_*)\bigr)(\theta_k-\theta_*)\,du .
\end{equation}

Similarly, $\theta_k - \theta_*$ can be expressed as:
\begin{equation}\label{eq:g.4.2}
\theta_k - \theta_* 
= [\nabla^2 \mathcal{L}(\theta_k)]^{-1}\int_{0}^{1} \nabla^2 \mathcal{L}(\theta_k)(\theta_k-\theta_*)\,du .
\end{equation}

From the Newton update rule:
\[
\theta_{k+1} - \theta_* = \theta_k - \theta_* - [\nabla^2 \mathcal{L}(\theta_k)]^{-1}\bigl[\nabla \mathcal{L}(\theta_k) - \nabla \mathcal{L}(\theta_*)\bigr].
\]
Substituting \eqref{eq:g.4.1} and \eqref{eq:g.4.2} yields:
\[
\theta_{k+1} - \theta_*
= [\nabla^2 \mathcal{L}(\theta_k)]^{-1}\mathcal{R}_k(\theta_k-\theta_*),
\quad
\mathcal{R}_k = \int_{0}^{1} [\nabla^2 \mathcal{L}(\theta_k) - \nabla^{2} \mathcal{L}(\theta_* + u(\theta_k - \theta_*))]\,du.
\]

By Definition~\ref{def:lip-cont-hes} and Theorem~\ref{theorem:jensen-spect}, there exists $L_2>0$ such that
\[
\|\mathcal{R}_k\|
\le \int_0^1 \|\nabla^2 \mathcal{L}(\theta_k) - \nabla^2 \mathcal{L}(\theta_*+u(\theta_k-\theta_*))\|\,du
\le \int_0^1 L_2(1-u)\|\theta_k-\theta_*\|\,du
= \tfrac{1}{2}L_2\|\theta_k-\theta_*\|.
\]

Furthermore, applying Lemma~\ref{lemma:lip-hes-bound} with $x=\theta_*$ and $y=\theta_k$, we obtain:
\[
\nabla^{2} \mathcal{L}(\theta_*) - L_2\|\theta_k-\theta_*\|\,I 
\;\preceq\; \nabla^{2} \mathcal{L}(\theta_k) 
\;\preceq\; \nabla^{2} \mathcal{L}(\theta_*) + L_2\|\theta_k-\theta_*\|\,I.
\]
Since $\nabla^{2} \mathcal{L}(\theta_*) \succeq \mu I$, it follows that
\[
\nabla^{2} \mathcal{L}(\theta_k) \;\succeq\; \bigl(\mu - L_2\|\theta_k-\theta_*\|\bigr)I.
\]
Thus, when $\|\theta_k-\theta_*\|<\mu/L_2$:
\[
\bigl\|[\nabla^{2} \mathcal{L}(\theta_k)]^{-1}\bigr\| 
\;\le\; \frac{1}{\mu - L_2\|\theta_k-\theta_*\|} 
\quad.
\]

Taking norms in the relation $\theta_{k+1}-\theta_*=[\nabla^2 \mathcal{L}(\theta_k)]^{-1}\mathcal{R}_k(\theta_k-\theta_*)$ yields
\[
\|\theta_{k+1}-\theta_*\|
\le \bigl\|[\nabla^2 \mathcal{L}(\theta_k)]^{-1}\bigr\|\,\|\mathcal{R}_k\|\,\|\theta_k-\theta_*\|
\le \frac{L_2}{2\bigl(\mu - L_2\|\theta_k-\theta_*\|\bigr)}\,\|\theta_k-\theta_*\|^2.
\]

Note that the bound $\|\theta_k - \theta_*\| < \mu/L_2$ ensures that the Hessian remains
positive definite and invertible within the local neighborhood of $\theta_*$,
since $\nabla^2 \mathcal{L}(\theta_k) \succeq (\mu - L_2\|\theta_k-\theta_*\|)I \succ 0$. 
To formalize this region of guaranteed invertibility, we fix a constant
$\rho \in (0, \mu/L_2)$ and suppose $\|\theta_k-\theta_*\|\le \rho$, then:
\[
\|\theta_{k+1}-\theta_*\|
\;\le\; \frac{L_2}{2(\mu-L_2\rho)}\,\|\theta_k-\theta_*\|^2
\;=\; \psi(\rho)\,\|\theta_k-\theta_*\|^2,
\qquad 
\psi(\rho):=\frac{L_2}{2(\mu-L_2\rho)}.
\]
This is the stated local quadratic bound. Moreover, if $\psi(\rho)\rho<1$, then:
\[
\|\theta_{k+1}-\theta_*\| \;\le\; \psi(\rho)\,\rho^2 \;<\; \rho
\quad\text{and}\quad
\frac{\|\theta_{k+1}-\theta_*\|}{\|\theta_k-\theta_*\|} \;\le\; \psi(\rho)\,\|\theta_k-\theta_*\| 
\;\le\; \psi(\rho)\,\rho \;<\; 1,
\]
which shows that the neighborhood $\{\theta:\|\theta-\theta_*\|\le \rho\}$ is strictly invariant and the iteration is
contractive with a monotonically decreasing error norm. 
\end{proof}
This theorem demonstrates that Newton's method achieves local quadratic convergence when the objective function has a strongly positive definite Hessian at the optimum and the Hessian varies smoothly in a neighborhood of that point. This proof is a restatement and detailed derivation of Theorem 3.3 in \citet{zhang2023notes}.

\subsubsection{Iteration Complexity of Newton’s Method}

From Theorem~\ref{theorem:newton-local-quad}, the local quadratic convergence relation is given by
\[
e_{k+1} \le \psi(\rho)\, e_k^2,
\qquad
\psi(\rho) = \frac{L_2}{2(\mu - L_2\rho)},
\]
where $e_k = \|\theta_k - \theta_*\|$ and $\rho < \mu/L_2$ ensures that the Hessian remains positive definite and the iteration stays within the local invariant region.

We now derive the corresponding \emph{iteration complexity}, that is, the number of steps $T$ required to achieve $e_T \le \varepsilon$ for a given accuracy $\varepsilon > 0$.

Unrolling the recurrence relation gives
\[
e_k 
\leq \psi(\rho)e_{k-1}^2 
\leq \psi(\rho)^{2^0 + 2^1} e_{k-2}^{2^2}
\leq \cdots
\leq \psi(\rho)^{\sum_{i=0}^{k-1} 2^i} e_0^{2^k}
= \psi(\rho)^{2^k - 1} e_0^{2^k}.
\]

To determine $T$ such that $e_T \le \varepsilon$, we require
\[
\psi(\rho)^{2^T - 1} e_0^{2^T} \le \varepsilon
\quad \Leftrightarrow \quad
\bigl(\psi(\rho)e_0\bigr)^{2^T} \le \psi(\rho)\varepsilon.
\]
Taking natural logarithms on both sides, and noting that by the local quadratic convergence condition we already have 
$e_1 = \psi(\rho)e_0^2 < e_0$ (hence $\psi(\rho)e_0<1$), we obtain 
\[
2^T \ln\!\bigl(\psi(\rho)e_0\bigr)\le \ln\!\bigl(\psi(\rho)\varepsilon\bigr).
\]
Since $\ln(\psi(\rho)e_0)<0$, dividing both sides by $\ln(\psi(\rho)e_0)$ reverses
the inequality, giving
\[
2^T \;\ge\; \frac{\ln\!\bigl(\psi(\rho)\varepsilon\bigr)}{\ln\!\bigl(\psi(\rho)e_0\bigr)}
\;=\; \frac{\ln\!\bigl(1/(\psi(\rho)\varepsilon)\bigr)}{\ln\!\bigl(1/(\psi(\rho)e_0)\bigr)}.
\]
Finally, taking $\log_2$ on both sides (monotonicity of $\log_2$) yields
\[
T \;\ge\;
\log_2\!\left(
\frac{
\ln\!\bigl(1 / (\psi(\rho)\varepsilon)\bigr)
}{
\ln\!\bigl(1 / (\psi(\rho)e_0)\bigr)
}
\right)
\;=\;
\log_2\!\left(
\frac{
\ln\!\bigl(1 / (\psi(\rho)\varepsilon)\bigr)
}{
\ln\!\bigl(1 / (\psi(\rho)\|\theta_0-\theta_*\|)\bigr)
}
\right).
\]
This gives an explicit lower bound on the number of iterations required for Newton’s method to reach a target accuracy~$\varepsilon$ within the local quadratic convergence region and since $\psi(\rho) = \frac{L_2}{2(\mu -L_2\rho)}$ there is a clear dependency of the iteration complexity on $L_2$.

\subsubsection{Third Order Analysis: Preliminaries}
\label{subsec:third-order-b3}

Our first goal to parallel the first-order analysis in the paper is to extend \cref{prop:L-lower-bound-app} to the third-order derivative tensor
$\mathcal{T}(x):=\nabla^3 f(x)\in\mathbb{R}^{n\times n\times n}$,
which represents the derivative of the Hessian of a thrice continuously differentiable function $f:\mathbb{R}^n\to\mathbb{R}$. 

Since we work within the T-product and T-SVD framework from \citet{zhang2021note}, we will present a number of preliminary definitions below.

\begin{definition}[Unfold, Fold, and Block-Circulant Representation]
Let $\mathcal{A}\in\mathbb{R}^{m\times n\times p}$ be a third-order tensor with frontal slices 
$A_k=\mathcal{A}(:,:,k)\in\mathbb{R}^{m\times n}$ for $k=1,\dots,p$.
The \emph{unfolding} of $\mathcal{A}$ is defined as the block column concatenation
\[
\mathrm{unfold}(\mathcal{A}) =
\begin{bmatrix}
A_1\\
A_2\\
\vdots\\
A_p
\end{bmatrix}
\in\mathbb{R}^{mp\times n},
\]

The \emph{folding} operator $\mathrm{fold}(\cdot)$ is the inverse mapping that reconstructs
a tensor from its unfolded matrix:
\[
\mathrm{fold}(\mathrm{unfold}(\mathcal{A}))=\mathcal{A}.
\]

The \emph{block-circulant matrix} of $\mathcal{A}$, denoted by $\mathrm{bcirc}(\mathcal{A})\in\mathbb{R}^{mp\times np}$, is defined as
\[
\mathrm{bcirc}(\mathcal{A})=
\begin{bmatrix}
A_1 & A_p & A_{p-1} & \cdots & A_2\\
A_2 & A_1 & A_p & \cdots & A_3\\
\vdots & \ddots & \ddots & \ddots & \vdots\\
A_p & A_{p-1} & \cdots & A_2 & A_1
\end{bmatrix}.
\]
\end{definition}
 The block-circulant structure is fundamental to the T-product framework, as it enables the extension of matrix algebraic operations to third-order tensors through the use of discrete Fourier transforms.

\begin{definition}[Discrete Fourier Transform along the third dimension]
\label{def:dft}
For a tensor $\mathcal{A}\in\mathbb{R}^{m\times n\times p}$, 
the Discrete Fourier Transform (DFT) along the third dimension is defined as
\[
\widehat{\mathcal{A}}(:,:,k)
= \sum_{t=1}^{p}\mathcal{A}(:,:,t)\,e^{-2\pi i (t-1)(k-1)/p}, 
\quad k=1,\dots,p.
\]
The inverse DFT is given by
\[
\mathcal{A}(:,:,t)
= \frac{1}{p}\sum_{k=1}^{p}\widehat{\mathcal{A}}(:,:,k)\,e^{2\pi i (t-1)(k-1)/p}.
\]

Under this transform, the block-circulant matrix of $\mathcal{A}$ is block-diagonalized as
\[
(F_p\otimes I_m)\,\mathrm{bcirc}(\mathcal{A})\,(F_p^{-1}\otimes I_n)
= \mathrm{diag}\big(\widehat{\mathcal{A}}(:,:,1),\dots,\widehat{\mathcal{A}}(:,:,p)\big),
\]
where $F_p$ is the $p\times p$ DFT matrix with entries $[F_p]_{jk} = e^{-2\pi i (j-1)(k-1)/p}$, and $\otimes$ denotes the Kronecker product.
\end{definition}

In practice, the DFT and inverse DFT are efficiently computed using the \emph{Fast Fourier Transform (FFT)} and \emph{Inverse Fast Fourier Transform (IFFT)}. For simplicity, we use the notations $\widehat{\mathcal{A}}=\mathrm{fft}(\mathcal{A},[],3)$ 
and $\mathcal{A}=\mathrm{ifft}(\widehat{\mathcal{A}},[],3)$ 
throughout the paper \citep{zhang2021note}.

Throughout this appendix, we adopt the \emph{unitary} version of the above DFT, i.e., we normalize the DFT matrix $F_p$ by $1/\sqrt{p}$ so that $F_p$ is unitary and the associated Fourier basis vectors have unit $\ell_2$ norm.
With this convention, all block-diagonalization identities remain valid, and tensor singular values are unchanged up to a consistent scaling.

\begin{definition}[Tensor Product (T-product)]
\label{def:t-product}
Let $\mathcal{A}\in\mathbb{R}^{m\times n\times p}$ and $\mathcal{B}\in\mathbb{R}^{n\times t\times p}$.
The \emph{T-product} between $\mathcal{A}$ and $\mathcal{B}$ is defined by
\[
\mathcal{A} * \mathcal{B}
= \mathrm{fold}\big(\mathrm{bcirc}(\mathcal{A})\,\mathrm{unfold}(\mathcal{B})\big)
\in\mathbb{R}^{m\times t\times p}.
\]
According to Definition~\ref{def:dft}, 
if $\widehat{\mathcal{A}}=\mathrm{fft}(\mathcal{A},[],3)$ and 
$\widehat{\mathcal{B}}=\mathrm{fft}(\mathcal{B},[],3)$,
then
\[
\widehat{\mathcal{C}}(:,:,k)
= \widehat{\mathcal{A}}(:,:,k)\,\widehat{\mathcal{B}}(:,:,k),
\quad
\forall k=1,\dots,p,
\]
where $\widehat{\mathcal{C}}=\mathrm{fft}(\mathcal{A}*\mathcal{B},[],3)$.
\end{definition}
This definition follows \citet{zhang2021note} and establishes the foundation for tensor algebra operations, allowing tensors to perform matrix-like multiplication in the Fourier domain that can be efficiently computed using FFT.

We will now present a singular value decomposition for tensors known as the T-SVD.

\begin{theorem}[Tensor Singular Value Decomposition (T-SVD) and Tensor Singular Values]
\label{thm:tsvd}
Let $\mathcal{A}\in\mathbb{R}^{m\times n\times p}$ be a third-order tensor. 
Then $\mathcal{A}$ admits the following factorization under the T-product:
\[
\mathcal{A} = \mathcal{U} * \mathcal{S} * \mathcal{V}^\top,
\]
where $\mathcal{U}\in\mathbb{R}^{m\times m\times p}$ and 
$\mathcal{V}\in\mathbb{R}^{n\times n\times p}$ are orthogonal tensors satisfying 
$\mathcal{U}^\top * \mathcal{U} = \mathcal{I}$ and $\mathcal{V}^\top * \mathcal{V} = \mathcal{I}$, 
and $\mathcal{S}\in\mathbb{R}^{m\times n\times p}$ is an \emph{f-diagonal tensor}, 
that is, each frontal slice $\mathcal{S}(:,:,k)$ is diagonal.
According to Definition~\ref{def:dft}, if 
$\widehat{\mathcal{A}}=\mathrm{fft}(\mathcal{A},[],3)$, 
then each frontal slice of $\widehat{\mathcal{A}}$ admits a standard matrix SVD:
\[
\widehat{\mathcal{A}}(:,:,k)
= \widehat{U}_k\,\widehat{S}_k\,\widehat{V}_k^\top, 
\quad k=1,\dots,p,
\]
where $\widehat{U}_k,\widehat{V}_k$ are orthogonal matrices and 
$\widehat{S}_k=\mathrm{diag}\big(\sigma_1^{(k)},\dots,\sigma_{r}^{(k)}\big)$ with $r=\min(m,n)$. Applying the inverse transform defined in Definition~\ref{def:dft} yields
\[
\mathcal{U}=\mathrm{ifft}(\widehat{\mathcal{U}},[],3), \quad
\mathcal{S}=\mathrm{ifft}(\widehat{\mathcal{S}},[],3), \quad
\mathcal{V}=\mathrm{ifft}(\widehat{\mathcal{V}},[],3).
\]
The \bf{tensor singular values} of $\mathcal{A}$ are defined as the diagonal entries of the first frontal slice $\mathcal{S}(:,:,1)$, denoted by
\[
\varsigma_i(\mathcal{A}) := \mathcal{S}(i,i,1)
= \frac{1}{p}\sum_{k=1}^{p}\sigma_i(\widehat{\mathcal{A}}(:,:,k)), 
\quad i=1,\dots,r,
\]
Hence, the tensor singular values satisfy the monotonicity
\[
\varsigma_1(\mathcal{A}) \ge \varsigma_2(\mathcal{A}) \ge \cdots 
\ge \varsigma_r(\mathcal{A}) \ge 0.
\]
\end{theorem}
This theorem extends the classical matrix SVD to third-order tensors through the T-product framework. See the proof in Theorem 2.3 from \citet{zhang2021note}.

\subsubsection{Third Order Analysis: $L_2$ bounds}

We not present a tensor Version of \cref{prop:L-lower-bound-app} where for the tensor singular values $\varsigma_i$, we have $\varsigma_1\!\big(\mathcal{T}(x)\big) \;\le\; L_2$.

\begin{theorem}[Tensor Version of \cref{prop:L-lower-bound-app}]
\label{thm:tensor-b3}
Let $f:\mathbb{R}^n\to\mathbb{R}$ be thrice continuously differentiable with an
$L_2$-Lipschitz continuous Hessian as defined in
Definition~\ref{def:lip-cont-hes}. Let $\mathcal{T}(x)=\nabla^3 f(x)\in\mathbb{R}^{n\times n\times n}$ denote the third-order derivative tensor.
Then the largest tensor singular value of $\mathcal{T}(x)$ under the T--SVD framework satisfies
\[
\varsigma_1\!\big(\mathcal{T}(x)\big) \;\le\; L_2.
\]
\end{theorem}
\begin{proof}
Let $f$ satisfy the $L_2$-Lipschitz Hessian condition in 
Definition~\ref{def:lip-cont-hes}, namely
\[
\|\nabla^{2} f(u)-\nabla^{2} f(v)\|\le L_2\|u-v\|,
\qquad \forall\,u,v\in\mathbb{R}^n.
\]
Fix a point $x\in\mathbb{R}^n$ and a unit direction $h\in\mathbb{R}^n$ with 
$\|h\|_2=1$.  
Define the matrix-valued function
\[
G(t) := \nabla^2 f(x + t h), \qquad t\in [0,1].
\]
Then by Definition~\ref{def:lip-cont-hes},
\[
\|G(t)-G(s)\|
= \|\nabla^2 f(x+th)-\nabla^2 f(x+sh)\|
\le L_2|t-s|.
\]
Hence $G$ is an $L_2$-Lipschitz function of $t$.  
Since $f$ is $C^3$ and $G$ is differentiable then
\[
G'(t)
= \frac{d}{dt}\nabla^2 f(x+th)
= \nabla^3 f(x+th)[\,\cdot,\cdot,h\,].
\]
Since $G$ is differentiable and its derivative $G'(t)$ is continuous,
for each fixed $t$ we have
\[
\|G'(t)\|
= \Bigl\|\lim_{s\to t}\frac{G(t)-G(s)}{t-s}\Bigr\|
\le \limsup_{s\to t}\frac{\|G(t)-G(s)\|}{|t-s|}
\le L_2, \qquad  \forall\,t\in[0,1].
\]
Therefore,
\begin{equation}\label{eq:g.9.1}
\big\|\nabla^3 f(x+th)[\,\cdot,\,\cdot,\,h\,]\big\|_2
= \|G'(t)\|
\le L_2,
\qquad \forall\,t\in[0,1],\ \|h\|_2=1.
\end{equation}
Applying the discrete Fourier transform along the third tensor dimension 
(Definition~\ref{def:dft}) gives
\[
\widehat{\mathcal{T}}
= \mathrm{fft}(\mathcal{T}(x),[],3),
\qquad 
\mathcal{T}(x)=\nabla^3 f(x).
\]
Its $k$-th frontal slice satisfies
\[
\widehat{\mathcal{T}}(:,:,k)
=\mathcal{T}(x)[\,\cdot,\cdot,\phi_k\,],
\]
where $\phi_k$ denotes the $k$-th unit-norm Fourier basis vector under the unitary DFT convention adopted throughout this appendix. 
Moreover, the bound in \eqref{eq:g.9.1} extends to complex unit vectors by identifying $\mathbb{C}^n$ with $\mathbb{R}^{2n}$ and using the same operator norm.

Since $\|\phi_k\|_2=1$, substituting $h=\phi_k$ into 
\eqref{eq:g.9.1} yields
\[
\sigma_{1}\!\big(\widehat{\mathcal{T}}(:,:,k)\big) =
\big\|\widehat{\mathcal{T}}(:,:,k)\big\|_2
\le L_2,
\qquad k=1,\dots,p.
\]
From Theorem~\ref{thm:tsvd}, the largest tensor singular values of $\mathcal{T}(x)$ are
\[
\varsigma_1\!\big(\mathcal{T}(x)\big)
= \frac{1}{p}\sum_{k=1}^{p}\sigma_{1}\!\big(\widehat{\mathcal{T}}(:,:,k)\big)
\le \frac{1}{p}\sum_{k=1}^{p}L_2
= L_2.
\]
Therefore the largest tensor singular value of the third-order derivative tensor
$\nabla^3 f(x)$ is bounded above by the Hessian Lipschitz constant $L_2$.
\end{proof}

\subsubsection{Third Order Analysis: Singular Value Inequalities in Tensors}

Before we proceed with connecting the above $L_2$ bounds with increases to weight singular values as we did with $L_1$ previously, we will need to state some tensor singular value inequalities analogous to the matrix singular values we presented in \cref{app:singular-value-inequalities}. 

\begin{theorem}[Poincaré separation theorem for tensor singular values]
\label{thm:tubal-poincare}
Let $\mathcal{A}\in\mathbb{R}^{m\times n\times p}$ be a third-order tensor, and let 
$\mathcal{A}_{\mathrm{sub}}\in\mathbb{R}^{m'\times n'\times p}$ be any subtensor of $\mathcal{A}$ 
with $1 \le m' \le m$ and $1 \le n' \le n$. 
Let the tensor singular values (under the T--SVD in Theorem~\ref{thm:tsvd}) be
\[
\varsigma_1(\mathcal{A}) \ge \varsigma_2(\mathcal{A}) \ge \cdots \ge \varsigma_{r}(\mathcal{A}) \ge 0,
\quad r = \min\{m,n\},
\]
and
\[
\varsigma_1(\mathcal{A}_{\mathrm{sub}}) \ge \varsigma_2(\mathcal{A}_{\mathrm{sub}}) \ge \cdots 
\ge \varsigma_{r'}(\mathcal{A}_{\mathrm{sub}}) \ge 0,
\quad r' = \min\{m',n'\}.
\]
Then, for all $k = 1,\dots, r'$, we have
\begin{equation}
\label{eq:tubal-sep-simple}
\varsigma_k(\mathcal{A})
\;\ge\;
\varsigma_k(\mathcal{A}_{\mathrm{sub}})
\;\ge\;
\varsigma_{\,k + (m - m') + (n - n')}(\mathcal{A}),
\end{equation}
with the convention that $\varsigma_j(\mathcal{A}) = 0$ if $j > r$.
\end{theorem}
\begin{proof}
Apply the discrete Fourier transform along the third mode (Definition~\ref{def:dft}) to both 
$\mathcal{A}$ and $\mathcal{A}_{\mathrm{sub}}$:
\[
\widehat{\mathcal{A}} = \mathrm{fft}(\mathcal{A},[],3),
\qquad
\widehat{\mathcal{A}_{\mathrm{sub}}} = \mathrm{fft}(\mathcal{A}_{\mathrm{sub}},[],3).
\]
For each $k=1,\dots,p$, the $k$-th frontal slices 
$\widehat{\mathcal{A}}(:,:,k)$ and $\widehat{\mathcal{A}_{\mathrm{sub}}}(:,:,k)$ are matrices with
sizes $m\times n$ and $m'\times n'$, respectively, and 
$\widehat{\mathcal{A}_{\mathrm{sub}}}(:,:,k)$ is a principal submatrix of $\widehat{\mathcal{A}}(:,:,k)$.
By the classical Poincaré separation (Horn \& Johnson, 1991, Th.~3.1.2), we have for each $k=1,\dots,p$:
\[
\sigma_j\big(\widehat{\mathcal{A}}(:,:,k)\big)
\;\ge\;
\sigma_j\big(\widehat{\mathcal{A}_{\mathrm{sub}}}(:,:,k)\big)
\;\ge\;
\sigma_{\,j + (m - m') + (n - n')}\big(\widehat{\mathcal{A}}(:,:,k)\big),
\qquad j=1,\dots,r'.
\]
Averaging over $k=1,\dots,p$ and using the definition of tensor singular values in 
Theorem~\ref{thm:tsvd},
\[
\varsigma_j(\mathcal{A})
= \frac{1}{p}\sum_{k=1}^{p} \sigma_j\big(\widehat{\mathcal{A}}(:,:,k)\big),
\qquad
\varsigma_j(\mathcal{A}_{\mathrm{sub}})
= \frac{1}{p}\sum_{k=1}^{p} \sigma_j\big(\widehat{\mathcal{A}_{\mathrm{sub}}}(:,:,k)\big),
\]
we obtain \eqref{eq:tubal-sep-simple}. 
The result is thus a direct tensor analogue of the matrix Poincaré separation theorem 
(see also \citealp[Theorem 3.1]{zhang2021note}).
\end{proof}

\begin{theorem}[Additive Tensor Singular Value Bound]
\label{thm:tensor-sing-sv-simple}
Let $\mathcal A,\mathcal B\in\mathbb R^{m\times n\times p}$ and let
$q=\min\{m,n\}$. Denote the tensor singular values under the T--SVD
(Definition~\ref{thm:tsvd}) by
$\varsigma_1(\cdot)\ge\cdots\ge \varsigma_q(\cdot)\ge 0$.
Then for all $i=1,\dots,q$,
\begin{equation}
\label{eq:tensor-add-simple}
\varsigma_i(\mathcal A)-\varsigma_1(\mathcal B)
\ \le\
\varsigma_i(\mathcal A+\mathcal B)
\ \le\
\varsigma_i(\mathcal A)+\varsigma_1(\mathcal B).
\end{equation}
\end{theorem}

\begin{proof}
Let $\widehat{\mathcal A}=\mathrm{fft}(\mathcal A,[],3)$ and
$\widehat{\mathcal B}=\mathrm{fft}(\mathcal B,[],3)$.
By linearity of the DFT along the third mode,
\[
\widehat{\mathcal A+\mathcal B}(:,:,k)
=
\widehat{\mathcal A}(:,:,k)+\widehat{\mathcal B}(:,:,k),
\qquad k=1,\dots,p.
\]
For each slice $k$, applying Theorem~\ref{thm:additive-singular-value-bounds}
to the matrices $\widehat A_k$ and $\widehat B_k$ yields
\[
\sigma_i(\widehat A_k)-\sigma_1(\widehat B_k)
\ \le\
\sigma_i(\widehat A_k+\widehat B_k)
\ \le\
\sigma_i(\widehat A_k)+\sigma_1(\widehat B_k),
\qquad i=1,\dots,q.
\]
Equivalently,
\[
\bigl|\sigma_i(\widehat A_k+\widehat B_k)-\sigma_i(\widehat A_k)\bigr|
\ \le\
\sigma_1(\widehat B_k),
\qquad i=1,\dots,q.
\]

Summing over $k=1,\dots,p$, dividing by $p$, and using the definition of
tensor singular values under the T--SVD yields
\[
\bigl|\varsigma_i(\mathcal A+\mathcal B)-\varsigma_i(\mathcal A)\bigr|
\le
\frac{1}{p}\sum_{k=1}^p \sigma_1(\widehat{\mathcal B}(:,:,k))
=
\varsigma_1(\mathcal B),
\]
which is equivalent to \eqref{eq:tensor-add-simple}.
\end{proof}

\begin{theorem}[Multiplicative Tensor Singular Value Bounds]
\label{thm:tensor-mult-sv-both}
Let $\mathcal{A}\in\mathbb{R}^{m\times \ell\times p}$ and 
$\mathcal{B}\in\mathbb{R}^{\ell\times n\times p}$, and let 
$q=\min\{m,n\}$ and $r=\min\{\ell,n\}$.  
Denote the tensor singular values under the T--SVD 
(Definition~\ref{thm:tsvd}) by
\(\varsigma_1(\cdot)\ge\cdots\ge \varsigma_q(\cdot)\ge 0\).
Let the Fourier slices be
\[
\widehat{A}_k = \widehat{\mathcal{A}}(:,:,k), 
\qquad 
\widehat{B}_k = \widehat{\mathcal{B}}(:,:,k), 
\qquad k=1,\dots,p.
\]
Then, for each $i=1,\dots,\min\{q,r\}$, the following bounds hold:
\begin{equation}
\label{eq:tensor-mult-sv-both}
\varsigma_i(\mathcal{A})\,
\min_{1\le k\le p}\sigma_r(\widehat{B}_k)
\;\le\;
\varsigma_i(\mathcal{A} * \mathcal{B})
\;\le\;
\varsigma_i(\mathcal{A})\,
\max_{1\le k\le p}\sigma_1(\widehat{B}_k).
\end{equation}
\end{theorem}

\begin{proof}
Let $\mathcal{C}=\mathcal{A} * \mathcal{B}$. 
By the definition of the T--product and the DFT along the third mode
(Definition~\ref{def:dft}), the Fourier slices satisfy
\[
\widehat{\mathcal{C}}(:,:,k)=\widehat{A}_k\widehat{B}_k,
\qquad k=1,\dots,p.
\]
For each fixed $k$, applying the matrix multiplicative singular-value
bounds (Theorem~\ref{thm:multiplicative-singular-value-bounds}) gives
\[
\sigma_i(\widehat{A}_k)\,\sigma_r(\widehat{B}_k)
\;\le\;
\sigma_i(\widehat{A}_k\widehat{B}_k)
\;\le\;
\sigma_i(\widehat{A}_k)\,\sigma_1(\widehat{B}_k),
\qquad i=1,\dots,q.
\]
By the definition of tensor singular values (Theorem~\ref{thm:tsvd}),
\[
\varsigma_i(\mathcal{A})
= \frac{1}{p}\sum_{k=1}^{p}\sigma_i(\widehat{A}_k),
\qquad
\varsigma_i(\mathcal{C})
= \frac{1}{p}\sum_{k=1}^{p}\sigma_i(\widehat{A}_k\widehat{B}_k).
\]
Noting that for every $k$ we have: 
\[
\sigma_1(\widehat B_k)\le \max_{1\le j\le p}\sigma_1(\widehat B_j),
\]
\[
\sigma_r(\widehat B_k)\ge \min_{1\le j\le p}\sigma_r(\widehat B_j),
\]
Therefore, summing the inequalities over all slices, we obtain:
\[
\varsigma_i(\mathcal{C})
\le
\frac{1}{p}\sum_{k=1}^{p}\sigma_i(\widehat{A}_k)\,\sigma_1(\widehat{B}_k)
\;\le\;
\left(\max_{1\le k\le p}\sigma_1(\widehat{B}_k)\right)
\frac{1}{p}\sum_{k=1}^{p}\sigma_i(\widehat{A}_k).
\]
Similarly,
\[
\varsigma_i(\mathcal{C})
\ge
\frac{1}{p}\sum_{k=1}^{p}\sigma_i(\widehat{A}_k)\,\sigma_r(\widehat{B}_k)
\;\ge\;
\left(\min_{1\le k\le p}\sigma_r(\widehat{B}_k)\right)
\frac{1}{p}\sum_{k=1}^{p}\sigma_i(\widehat{A}_k).
\]
Combining both inequalities gives the result.
\end{proof}

In particular, if $\widehat{B}_k$ is rank-deficient for some $k$, then $\sigma_r(\widehat{B}_k)=0$ and the lower bound in \eqref{eq:tensor-mult-sv-both} becomes trivial. Future work could consider developing analogous principal angles that tighten these in practice but that is out of scope of our paper. 

\paragraph{Tensor singular value bounds for T-products.}
We now combine the additive and multiplicative tensor singular value bounds to characterize the behaviour of tensors of the form
\[
\mathcal{H} = \mathcal{A} * \mathcal{B} + \mathcal{C},
\]
which we will later show captures the product structure existing in the third order derivatives of a neural network.

By the additive tensor singular value bound
(Theorem~\ref{thm:tensor-sing-sv-simple}), for any $i=1,\dots,q$ we have
\[
\varsigma_i(\mathcal{A} * \mathcal{B}) - \varsigma_1(\mathcal{C})
\;\le\;
\varsigma_i(\mathcal{A} * \mathcal{B} + \mathcal{C})
\;\le\;
\varsigma_i(\mathcal{A} * \mathcal{B}) + \varsigma_1(\mathcal{C}).
\]

Applying the multiplicative tensor singular value bounds
(Theorem~\ref{thm:tensor-mult-sv-both}) to the T--product term
$\mathcal{A} * \mathcal{B}$ yields
\[
\varsigma_i(\mathcal{A})\,
\min_{1\le k\le p}\sigma_r(\widehat{B}_k)
\;\le\;
\varsigma_i(\mathcal{A} * \mathcal{B})
\;\le\;
\varsigma_i(\mathcal{A})\,
\max_{1\le k\le p}\sigma_1(\widehat{B}_k).
\]

Combining the two inequalities gives the composite bounds
for all $i = 1,\dots,\min\{q,r\}$:
\begin{equation}
\label{eq:ab-plus-c-upper}
\varsigma_i(\mathcal{A} * \mathcal{B} + \mathcal{C})
\;\le\;
\varsigma_i(\mathcal{A})\,
\max_{1\le k\le p}\sigma_1(\widehat{B}_k)
\;+\;
\varsigma_1(\mathcal{C}),
\end{equation}
and
\begin{equation}
\label{eq:ab-plus-c-lower}
\varsigma_i(\mathcal{A} * \mathcal{B} + \mathcal{C})
\;\ge\;
\varsigma_i(\mathcal{A})\,
\min_{1\le k\le p}\sigma_r(\widehat{B}_k)
\;-\;
\varsigma_1(\mathcal{C}).
\end{equation}

As in the purely multiplicative case, the lower bound
\eqref{eq:ab-plus-c-lower} may become trivial if one or more Fourier slices
$\widehat{B}_k$ are rank-deficient, in which case
$\sigma_r(\widehat{B}_k)=0$. Nevertheless, the upper bound
\eqref{eq:ab-plus-c-upper} always holds and shows that the leading tensor
singular value of $\mathcal{A} * \mathcal{B} + \mathcal{C}$ is controlled by
the spectrum of $\mathcal{A}$ and $\mathcal{B}$, up to an additive contribution
from $\mathcal{C}$.

These bounds provide the key mechanism by which spectral properties of
individual weight matrices can be used to control the tensor singular values
of third-order derivative blocks. We now illustrate this structure explicitly
for MLP.

\subsubsection{Third Order Analysis: Application to Neural Networks}

We now present an example analysis of a four-layer MLP to illustrate how modifying the spectral values of the weights of a network related to $L_2$ and thereby the convergence rate of second-order methods.

\begin{example}[Multi-layer Perceptron(MLP)]
\label{example:mlp-tsvd}
Consider the four-layer MLP:
\[
f(x)
= \theta_{4}^{\top}\,
\phi_{3}\!\Bigl(
\theta_{3}\,
\phi_{2}\!\bigl(
\theta_{2}\,
\phi_{1}(\theta_{1}x)
\bigr)
\Bigr),
\]
where the weight matrices have dimensions
\[
\theta_{1}\in\mathbb{R}^{m_{1}\times d},\qquad
\theta_{2}\in\mathbb{R}^{m_{2}\times m_{1}},\qquad
\theta_{3}\in\mathbb{R}^{m_{3}\times m_{2}},\qquad
\theta_{4}\in\mathbb{R}^{m_{3}}.
\]
Define the layer pre-activations and activations as
\[
h_{1}=\theta_{1}x, \qquad z_{1}=\phi(h_{1}), 
\]
\[
h_{2}=\theta_{2}z_{1}, \qquad z_{2}=\phi(h_{2}),
\]
\[
h_{3}=\theta_{3}z_{2}, \qquad z_{3}=\phi(h_{3}),
\]
so that the network output simplifies to 
\[
f(x)=\theta_{4}^{\top}z_{3}.
\]
Let the diagonal Jacobian matrices of the nonlinearity be
\[
D_{z_{1}}=\mathrm{Diag}(\phi'(h_{1})) \in \mathbb{R}^{m_{1}\times m_{1}},\qquad
D_{z_{2}}=\mathrm{Diag}(\phi'(h_{2})) \in \mathbb{R}^{m_{2}\times m_{2}},
\]
\[
D_{z_{3}}=\mathrm{Diag}(\phi'(h_{3})) \in \mathbb{R}^{m_{3}\times m_{3}},\qquad
D'_{z_{3}}=\mathrm{Diag}(\phi''(h_{3})) \in \mathbb{R}^{m_{3}\times m_{3}}.
\]
We consider the mixed third derivative
\[
\nabla^{3}_{\theta_{3},\,\theta_{1},\,\theta_{4}} f(x)
\in
\mathbb{R}^{m_{3} \times (m_{1}d) \times (m_{3}m_{2})},
\]
whose frontal slices correspond to derivatives with respect to each entry $(\theta_{3})_{ab}$. 

To begin, differentiate $f$ with respect to $\theta_{4}$. Since $f = \theta_{4}^{\top} z_{3}$, it follows that:
\[
\frac{\partial f}{\partial \theta_{4}} = z_{3}.
\]
Next, we differentiate $f$ with respect to $\theta_{1}$ and $\theta_{4}$.  
Applying the chain rule,
\[
\frac{\partial z_{3}}{\partial z_{2}} = D_{z_{3}}\theta_{3},\qquad
\frac{\partial z_{2}}{\partial z_{1}} = D_{z_{2}}\theta_{2},\qquad
\frac{\partial z_{1}}{\partial \theta_{1}}
= D_{z_{1}} (x^{\top} \otimes I_{m_{1}}),
\]
and therefore:
\[
\frac{\partial^{2} f}{\partial \theta_{1}\, \partial \theta_{4}}
=
D_{z_{3}}\theta_{3} D_{z_{2}}\theta_{2} D_{z_{1}}\,
(x^{\top}\otimes I_{m_{1}}),
\]
\[
\theta_{3} =
\begin{bmatrix}
(\theta_{3})_{11} & (\theta_{3})_{12} & \cdots & (\theta_{3})_{1m_{2}}\\[3pt]
(\theta_{3})_{21} & (\theta_{3})_{22} & \cdots & (\theta_{3})_{2m_{2}}\\
\vdots & \vdots & \ddots & \vdots \\
(\theta_{3})_{m_{3}1} & (\theta_{3})_{m_{3}2} & \cdots & (\theta_{3})_{m_{3}m_{2}}
\end{bmatrix},
\qquad
\theta_{3} \in \mathbb{R}^{m_{3}\times m_{2}}.
\]
Finally, we differentiate $f$ with respect to $(\theta_{3}, \theta_{1}, \theta_{4})$.  
For each entry $(a,b)$ define the basis matrix
\[
E_{ab} \in \mathbb{R}^{m_{3}\times m_{2}},
\qquad
[E_{ab}]_{ij} =
\begin{cases}
1, & i=a,\ j=b,\\
0, & \text{otherwise}.
\end{cases}
\]
Thus, differentiating gives:
\[
\frac{\partial \theta_{3}}{\partial (\theta_{3})_{ab}} = E_{ab} 
\in \mathbb{R}^{m_3 \times m_2},
\qquad
\frac{\partial D_{z_3}}{\partial (\theta_3)_{ab}}
= \operatorname{Diag}\!\big(D'_{z_3}(E_{ab} z_2)\big)
\;\in\; \mathbb{R}^{m_3 \times m_3}
\]
where $\operatorname{Diag}(\cdot)$ denotes the diagonal matrix formed from a vector.

Consequently, each frontal slice of the third derivative tensor admits the decomposition
\begin{equation}
\label{eq:mlp-third-slice}
\frac{\partial^{3} f}{\partial(\theta_{3})_{ab}\,\partial\theta_{1}\,\partial\theta_{4}}
=
\mathcal{C}(:,:,k)
\;+\;
\mathcal{A}(:,:,k)\,\mathcal{B}(:,:,k),
\end{equation}
where
\[
\mathcal{A}(:,:,k)=D_{z_{3}}E_{ab}D_{z_{2}},\qquad
\mathcal{B}(:,:,k)=\theta_{2}D_{z_{1}}(x^{\top}\otimes I_{m_{1}}),
\]
and
\[
\mathcal{C}(:,:,k)
=
\operatorname{Diag}\!\big(D'_{z_{3}}(E_{ab}z_{2})\big)\,
\theta_{3}D_{z_{2}}\theta_{2}D_{z_{1}}
(x^{\top}\otimes I_{m_{1}}).
\]
Here $k\leftrightarrow(a,b)$ indexes the entries of $\theta_3$, and
$p=m_3 m_2$ is the total number of such slices.

Stacking the matrices over all $(a,b)$ along the third dimension yields the third-order tensor
\[
\mathcal{H}
=
\nabla^{3}_{\theta_{3}, \theta_{1}, \theta_{4}} f(x)
\in \mathbb{R}^{m_{3} \times (m_{1}d) \times p}.
\]
We emphasize that $\mathcal{H}$ represents the coordinate-slice (basis-dependent)
representation of the third-order derivative tensor with respect to the standard
basis $\{E_{ab}\}$ of $\theta_3$.

Let $w := \mathrm{vec}(\theta_3)\in\mathbb{R}^{p}$ with $p=m_3 m_2$.
Since $\nabla^3_{\theta_3,\theta_1,\theta_4} f(x)$ is trilinear in $\theta_3$,
it admits a directional representation.

Given a set of unit vectors $\{\phi_k\}_{k=1}^p \subset \mathbb{C}^p$,
we define the directionally parameterized third-order tensor
$\mathcal{T}(x)\in\mathbb{R}^{m_3\times(m_1d)\times p}$ by
\[
\mathcal{T}(:,:,k)
:=
\nabla^{3}_{w,\theta_1,\theta_4} f(x)[\phi_k]
=
\sum_{j=1}^{p} \mathcal{H}(:,:,j)\,(\phi_k)_j .
\]

We emphasize that this construction does not change the order of
differentiation: $\mathcal{T}$ represents a change of basis of the third-order
derivative tensor along the $\theta_3$ mode, rather than a different
differential object.

Throughout this example, we take $\{\phi_k\}$ to be the unitary Fourier basis associated with the normalized DFT matrix $F_p$.
Under this choice, $\mathcal{T}$ is exactly the unitary DFT of $\mathcal{H}$ along the third mode.

Accordingly, the tensors $\mathcal{A}$, $\mathcal{B}$, and $\mathcal{C}$ are
interpreted as the directional lifts of their coordinate-slice counterparts
under the same unitary basis, so that their Fourier slices correspond to the
associated linear combinations along the third mode.

In the Fourier-domain representation associated with the third mode, we have
\[
\widehat{\mathcal{T}}(:,:,k)
=
\widehat{\mathcal{A}}(:,:,k)\,\widehat{\mathcal{B}}(:,:,k)
+
\widehat{\mathcal{C}}(:,:,k),
\qquad k=1,\dots,p.
\]

By the definition of the T--product, this identity is equivalent to the exact
tensor decomposition
\[
\mathcal{T} = \mathcal{A} * \mathcal{B} + \mathcal{C}.
\]

As a consequence, the tensor singular values of the directionally parameterized tensor $\mathcal{T}$ can be bounded by combining the multiplicative tensor singular value bounds for $\mathcal{A}*\mathcal{B}$
(Theorem~\ref{thm:tensor-mult-sv-both}) with the additive tensor singular value bounds for the perturbation $\mathcal{C}$
(Theorem~\ref{thm:tensor-sing-sv-simple}).

Since $\mathcal{T}$ is obtained from $\mathcal{H}$ via a unitary transform along the third mode, their tensor singular values coincide. Therefore, the same bound applies to the original third-order derivative
tensor $\mathcal{H}$.

\end{example}

\cref{tab:mlp-lb-ub-tsvd} reports a numerical verification of the T--SVD bounds derived in the MLP example (Example~\ref{example:mlp-tsvd}). To test the effect of spectral scaling, we introduce a scalar factor
$\alpha>0$ and construct a scaled tensor
\[
\mathcal{H}(\alpha)
\;:=\;
(\alpha\,\mathcal{A}) * \mathcal{B} + \mathcal{C},
\]
where the scaling is applied exclusively to $\mathcal{A}$, while
$\mathcal{B}$ and $\mathcal{C}$ are held fixed.
For each value of $\alpha$, the tensor $\mathcal{H}(\alpha)$ is explicitly
reconstructed via the T--product.

For each $\mathcal{H}(\alpha)$, we compute the maximum tensor singular value
$\varsigma_1(\mathcal{H}(\alpha))$ using the T--SVD definition, together with
the corresponding lower and upper bounds
\[
\mathrm{LB}_1(\alpha)
=
\varsigma_1(\mathcal{A})\,
\min_{k}\sigma_r(\widehat{B}_k)
-
\varsigma_1(\mathcal{C}),
\qquad
\mathrm{UB}_1(\alpha)
=
\varsigma_1(\mathcal{A})\,
\max_{k}\sigma_1(\widehat{B}_k)
+
\varsigma_1(\mathcal{C}),
\]
as given by \eqref{eq:ab-plus-c-lower} and \eqref{eq:ab-plus-c-upper}.

From this table we observe that, as the scaling factor $\alpha$ increases,
both the lower bound $\mathrm{LB}_1(\alpha)$ and the upper bound
$\mathrm{UB}_1(\alpha)$ increase correspondingly.
This behavior is consistent with the composite tensor singular value bounds
in \eqref{eq:ab-plus-c-upper}--\eqref{eq:ab-plus-c-lower}, which predict that
scaling the multiplicative component $\mathcal{A}$ proportionally enlarges
the admissible range of $\varsigma_1(\mathcal{H}(\alpha))$ when $\mathcal{B}$
and $\mathcal{C}$ are held fixed.

\begin{table}[h]
\centering
\caption{Verification of lower and upper bounds for the top tensor singular value
$\varsigma_1(\mathcal{H}(\alpha))$ under consistent scaling of $\mathcal{A}$.}
\label{tab:mlp-lb-ub-tsvd}
\begin{tabular}{llll}
\hline
$\alpha$
& $\varsigma_1(\mathcal{H}(\alpha))$
& $\mathrm{LB}_1(\alpha)$
& $\mathrm{UB}_1(\alpha)$ \\
\hline
$1$      & $1.27$     & $-0.20$   & $2.15$       \\
$10$     & $13.34$    & $0.18$    & $19.27$      \\
$100$    & $134.15$   & $3.94$    & $190.52$     \\
$1000$   & $1342.27$  & $41.62$   & $1903.03$    \\
$10000$  & $13423.46$ & $418.43$  & $19028.13$   \\
$100000$ & $134235.38$& $4186.45$ & $190279.08$  \\
\hline
\end{tabular}
\end{table}

Based on this example we can provide a number of remarks on when control of the spectral values of the weights alone are useful for controlling $L_2$.

\begin{remark}[Conditions for Effective $L_2$ Control via Spectral Inflation]
\label{rem:L2-control-conditions}
The preceding analysis establishes that the Hessian Lipschitz constant $L_2$ 
upper bounds the largest tensor singular value of $\nabla^3 f$ 
(Theorem~\ref{thm:tensor-b3}), and that the tensor singular values of 
composite expressions $\mathcal{A} * \mathcal{B} + \mathcal{C}$ satisfy 
the bounds \eqref{eq:ab-plus-c-upper}--\eqref{eq:ab-plus-c-lower}. 
We now discuss the conditions under which inflating weight singular values 
translates into effective control of $L_2$.

\paragraph{When spectral inflation is effective}
The upper bound \eqref{eq:ab-plus-c-upper} shows that 
$\varsigma_i(\mathcal{A} * \mathcal{B} + \mathcal{C}) \le 
\varsigma_i(\mathcal{A})\max_k \sigma_1(\widehat{B}_k) + \varsigma_1(\mathcal{C})$.
Inflating the singular values of weight matrices appearing in $\mathcal{A}$ 
and $\mathcal{B}$ directly increases the right-hand side, thereby 
\emph{permitting} a larger $L_2$. This upper bound argument is tight when:
\begin{enumerate}[label=(\roman*),leftmargin=2em]
    \item The activation functions have bounded second derivatives 
    (e.g., $\|\phi''\|_\infty < \infty$), ensuring that the additive term 
    $\mathcal{C}$ containing $D'_{z_\ell}$ remains controlled;
    \item The network operates in a regime where the multiplicative 
    structure $\mathcal{A} * \mathcal{B}$ dominates the third-order 
    derivative tensor.
\end{enumerate}

\paragraph{Limitations of the lower bound}
The lower bound \eqref{eq:ab-plus-c-lower} is subtractive: 
$\varsigma_i(\mathcal{A} * \mathcal{B} + \mathcal{C}) \ge 
\varsigma_i(\mathcal{A})\min_k \sigma_r(\widehat{B}_k) - \varsigma_1(\mathcal{C})$.
This bound becomes vacuous when either (i) some Fourier slice $\widehat{B}_k$ 
is rank-deficient, or (ii) the perturbation term $\varsigma_1(\mathcal{C})$ 
dominates. In such cases, inflating weight spectra does not guarantee a 
proportional increase in the tensor singular values of $\nabla^3 f$, 
and hence does not guarantee increased $L_2$.

\paragraph{Sufficient conditions for guaranteed $L_2$ increase.}
A sufficient condition for spectral inflation to increase $L_2$ is that 
the perturbation term satisfies
\begin{equation}
\label{eq:C-dominance-condition}
\varsigma_1(\mathcal{C}) \le 
\varsigma_1(\mathcal{A})\min_{k}\sigma_r(\widehat{B}_k) - \delta
\end{equation}
for some margin $\delta > 0$, and all Fourier slices $\widehat{B}_k$ 
have full column rank. Under this condition, increasing 
$\varsigma_1(\mathcal{A})$ by a factor $\alpha > 1$ yields
\[
\varsigma_1(\mathcal{A} * \mathcal{B} + \mathcal{C}) 
\ge \alpha \cdot \varsigma_1(\mathcal{A})\min_k \sigma_r(\widehat{B}_k) 
- \varsigma_1(\mathcal{C})
\ge (\alpha - 1)\delta + \varsigma_1(\mathcal{A}*\mathcal{B}+\mathcal{C})\big|_{\alpha=1},
\]
demonstrating a strict increase in the tensor singular value and hence in $L_2$.

\paragraph{Practical implications}
In practice, condition \eqref{eq:C-dominance-condition} is most likely 
satisfied in the following regimes:
\begin{itemize}[leftmargin=2em]
    \item \emph{Smooth activations}: When $\phi'' \approx 0$ 
    (e.g., near-linear regions of tanh or GELU), the term $\mathcal{C}$ 
    containing $D'_{z_\ell}$ is small.
    \item \emph{Moderate depth}: In shallow networks, fewer compositional 
    terms appear in $\mathcal{C}$, reducing its magnitude relative to 
    $\mathcal{A} * \mathcal{B}$.
    \item \emph{Well-conditioned intermediate representations}: When 
    $\min_k \sigma_r(\widehat{B}_k)$ is bounded away from zero, the 
    multiplicative term provides a strong baseline.
\end{itemize}

We acknowledge that our theoretical framework provides \emph{permissive} 
rather than \emph{prescriptive} guarantees: inflating weight spectra 
\emph{allows} $L_2$ to be larger but does not unconditionally \emph{force} 
it to be larger. The empirical effectiveness of SpecDef observed in our 
experiments suggests that practical networks often operate in regimes 
where condition \eqref{eq:C-dominance-condition} approximately holds, 
though a complete characterization of these regimes remains an important 
direction for future theoretical work.
\end{remark}


\end{document}